\documentclass[11pt]{amsart}
\usepackage[table]{xcolor}
\usepackage{graphicx}
\usepackage[mathcal]{euscript}
\usepackage{float}
\usepackage{cite}
\usepackage{verbatim}
\usepackage{url}
\usepackage{tikz}
\usepackage{adjustbox}
\usepackage{amsmath}
\usepackage{multicol}
\usepackage{rotating}

\usetikzlibrary{decorations.pathreplacing}
\usetikzlibrary{arrows.meta}

\usetikzlibrary{arrows, shapes, decorations, decorations.markings, backgrounds, patterns, hobby, knots, calc, positioning, calligraphy}

\pgfdeclarelayer{background}
\pgfdeclarelayer{background2}
\pgfdeclarelayer{background2a}
\pgfdeclarelayer{background2b}
\pgfdeclarelayer{background3}
\pgfdeclarelayer{background4}
\pgfdeclarelayer{background5}
\pgfdeclarelayer{background6}
\pgfdeclarelayer{background7}

\pgfsetlayers{background7,background6,background5,background4,background3,background2b,background2a,background2,background,main}

\definecolor{lsupurple}{RGB}{70,29,124}
\definecolor{lsugold}{RGB}{253,208, 35}

\restylefloat{figure}

\newtheorem{theorem}{Theorem}[section]
\newtheorem{lemma}[theorem]{Lemma}

\newtheorem*{corollaryconj*}{Corollary of Conjecture \ref{conjecture:Signature}}

\theoremstyle{definition}
\newtheorem{definition}[theorem]{Definition}
\newtheorem{example}[theorem]{Example}

\theoremstyle{remark}

\numberwithin{equation}{section}

\def\pos{\operatorname{pos}}
\def\alt{\operatorname{alt}}

\def\adj{\operatorname{adj}}

\def\Kh{\underline{Kh}}
\def\ukh{\underline{Kh}}
\def\girth{\operatorname{girth}}

\usepackage{tikz}

\usetikzlibrary{arrows,shapes,decorations,backgrounds,patterns}

\pgfdeclarelayer{background}
\pgfdeclarelayer{background2}
\pgfdeclarelayer{background2a}
\pgfdeclarelayer{background2b}
\pgfdeclarelayer{background3}
\pgfdeclarelayer{background4}
\pgfdeclarelayer{background5}
\pgfdeclarelayer{background6}
\pgfdeclarelayer{background7}

\pgfsetlayers{background7,background6,background5,background4,background3,background2b,background2a,background2,background,main}

\definecolor{lsupurple}{RGB}{70,29,124}
\definecolor{lsugold}{RGB}{253,208, 35}

\begin{document}

\title{\small{Near extremal Khovanov homology of Turaev genus one links}}

\author[\tiny{T. Beldon}]{Theo Beldon}
\address{\tiny{Department of Mathematics and Statistics\\
Vassar College\\
Poughkeepsie, NY}} 
\email{cbeldon@vassar.edu}

\author[\tiny{M. Destefano}]{Mia Destefano}
\address{\tiny{Department of Mathematics and Statistics\\
Vassar College\\
Poughkeepsie, NY}}
\email{mia.destefano@gmail.com}

\author[A. M. Lowrance]{Adam M. Lowrance}
\address{\tiny{Department of Mathematics and Statistics\\
Vassar College\\
Poughkeepsie, NY}} 
\email{adlowrance@vassar.edu}

\author[W. Milgrim]{Wyatt Milgrim}
\address{\tiny{Department of Mathematics and Statistics\\
Vassar College\\
Poughkeepsie, NY}} 
\email{wmilgrim@vassar.edu}

\author[C. Villase\~nor]{Cecilia Villase\~nor}
\address{\tiny{Department of Mathematics and Statistics\\
Vassar College\\
Poughkeepsie, NY}} 
\email{ceciqvilla@gmail.com}

\thanks{This paper is the result of a summer research project in Vassar Colleges Undergraduate Research Science Institute. The third author is supported by NSF grant DMS-1811344.}

\begin{abstract}
The Turaev surface of a link diagram $D$ is a closed, oriented surface constructed from a cobordism between the all-$A$ and all-$B$ Kauffman states of $D$. The Turaev genus of a link $L$ is the minimum genus of the Turaev surface of any diagram $D$ of $L$. A link is alternating if and only if its Turaev genus is zero, and so one can view Turaev genus one links as being close to alternating links. In this paper, we study the Khovanov homology of a Turaev genus one link in the first and last two polynomial gradings where the homology is nontrivial.  We show that a particular summand in the Khovanov homology of a Turaev genus one link is trivial. This trivial summand leads to a computation of the Rasmussen $s$ invariant and to bounds on the smooth four genus for certain Turaev genus one knots. 
\end{abstract}

\maketitle

\section{Introduction}

The Turaev surface of a link diagram $D$ is a closed, oriented surface constructed from a cobordism between the all-$A$ and all-$B$ Kauffman states of $D$ whose saddle points correspond to the crossings of $D$. Turaev \cite{Turaev} originally constructed this surface to give a topological proof that the span of the Jones polynomial gives a lower bound on the crossing number of a link (originally proven by Kauffman \cite{Kauffman:StateModels}, Murasugi \cite{Murasugi:Jones}, and Thistlethwaite \cite{Thistlethwaite:Spanning}). The minimum genus of the Turaev surface of any diagram of a link $L$ is the \textit{Turaev genus} $g_T(L)$ of $L$.

Because the Turaev genus of a link is zero if and only if the link is alternating, one can view the Turaev genus of a link as a measure of the link's distance from being alternating (see \cite{Low:AltDist}).  Kim \cite{Kim:Turaev} and independently Armond and Lowrance \cite{ArmLow} showed that every nonsplit Turaev genus one link has a diagram as in Figure \ref{figure:tg1}. Dasbach and Lowrance used this characterization of Turaev genus one links to prove that at least one of the leading or trailing coefficients of the Jones polynomial of a Turaev genus one link has absolute value one \cite{DasLow:Invariants} and to prove that the Khovanov homology of a Turaev genus one link in at least one of its extremal polynomial gradings is isomorphic to $\mathbb{Z}$ \cite{DasLow:Extremal}. Similarly, Lowrance and Spyropoulos \cite{LowSp:Jones} used the characterization to give a formulas for the second and penultimate coefficients of the Jones polynomial of a Turaev genus one link and to show that no Turaev genus one link has trivial Jones polynomial. In the current article, we study the second and penultimate polynomial gradings of the Khovanov homology of a Turaev genus one link. In order to precisely state our results, we recall the definitions of and some facts about almost alternating links and $A$- and $B$-adequate links.

\begin{figure}
\[\begin{tikzpicture}[thick, scale = .8]
\draw [bend left] (0,.5) edge (3,.5);
\draw [bend right] (0,-.5) edge (3, -.5);
\draw [bend left] (3,.5) edge (6,.5);
\draw [bend right] (3,-.5) edge (6, -.5);
\draw [bend left] (6,.5) edge (9,.5);
\draw [bend right] (6,-.5) edge (9, -.5);
\draw [bend left] (7.5,.5) edge (10.5,.5);
\draw [bend right] (7.5,-.5) edge (10.5,-.5);

\fill[white] (7.5,1) rectangle (9.1,-1);
\draw (8.25,0) node{\Large{$\dots$}};

\draw (-.8,.6) arc (270:90:.8cm);
\draw (-.8,-.6) arc (90:270:.8cm);
\draw (11.3,.6) arc (-90:90:.8cm);
\draw (11.3,-.6) arc (90:-90:.8cm);
\draw (-.8,2.2) -- (11.3,2.2);
\draw (-.8,-2.2) -- (11.3,-2.2);

\fill[white] (0,0) circle (1cm);
\draw (0,0) node {$R_1$};
\draw (0,0) circle (1cm);
\fill[white] (3,0) circle (1cm);
\draw (3,0) node {$R_2$};
\draw (3,0) circle (1cm);
\fill[white] (6,0) circle (1cm);
\draw (6,0) node {$R_3$};
\draw (6,0) circle (1cm);
\fill[white] (10.5,0) circle (1cm);
\draw (10.5,0) node {$R_{2k}$};
\draw (10.5,0) circle (1cm);

\draw (-.5,.6) node{$-$};
\draw (-.5,-.6) node{$+$};
\draw (.4,.7) node{$+$};
\draw (.4,-.7) node{$-$};

\draw (2.6,.7) node{$+$};
\draw (2.6,-.7) node{$-$};
\draw (3.4, .7) node{$-$};
\draw (3.4,-.7) node{$+$};

\draw (5.6, .7) node{$-$};
\draw (5.6,-.7) node{$+$};
\draw (6.4,.7) node{$+$};
\draw (6.4,-.7) node{$-$};

\draw (10.1, .7) node{$+$};
\draw (10.1,-.7) node{$-$};
\draw (11,.6) node{$-$};
\draw (11,-.6) node{$+$};

\end{tikzpicture}\]
\caption{If a non-split link $L$ has Turaev genus one, then $L$ has a diagram as above, where each $R_i$ is an alternating tangle whose closures are connected link diagrams. The label ``$+$'' on an incoming strand indicates that the strand passes over another strand in the first crossing it encounters. Similarly, the label ``$-$'' on an incoming strand indicates that the strand passes under another strand in the first crossing it encounters. We use this convention throughout the paper.}
\label{figure:tg1}
\end{figure}
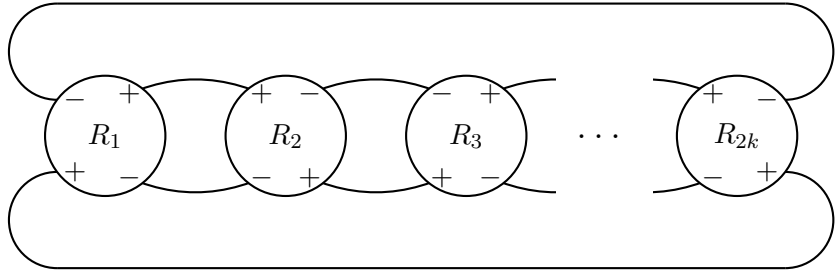

A link is \textit{almost alternating} if it is nonalternating and has a diagram that can be transformed into an alternating diagram via one crossing change. Such a diagram is called an \textit{almost alternating diagram}, and the crossing that is changed to obtain an alternating diagram is known as the \textit{dealternator}. Adams et al. \cite{Adams:Almost} first defined almost alternating links, proved that an almost alternating knot is either a torus knot or a hyperbolic knot, and proved some properties of the span of the Jones polynomial of an almost alternating link. Define $u_1$ and $u_2$ to be the two complementary regions of an almost alternating diagram incident to the dealternator  $\tikz[baseline=.6ex, scale = .4]{
\draw (0,0) -- (.3,.3);
\draw (.7,.7) -- (1,1);
\draw (0,1) -- (1,0);
}
~$ that are joined by an $A$-resolution$~ \tikz[baseline=.6ex, scale = .4]{
\draw[rounded corners = 1mm] (0,0) -- (.5,.45) -- (1,0);
\draw[rounded corners = 1mm] (0,1) -- (.5,.55) -- (1,1);
}~$.
Define $v_1$ and $v_2$ to be the two complementary regions incident to the dealternator that are joined by a $B$-resolution $~ \tikz[baseline=.6ex, scale = .4]{
\draw[rounded corners = 1mm] (0,0) -- (.45,.5) -- (0,1);
\draw[rounded corners = 1mm] (1,0) -- (.55,.5) -- (1,1);
}~$. Color the regions of the link diagram in a checkerboard fashion such that $v_1$ and $v_2$ are black and $u_1$ and $u_2$ are white. Let $R$ be the alternating tangle containing all of the crossings of $D$ other than the dealternator (see Figure \ref{figure:aadiagram}).

\begin{definition}
\label{definition:ABalmostalternating}
An \textit{$A$-almost alternating diagram} is an  almost alternating diagram satisfying the following conditions.
\begin{itemize}
\item [(1)] The regions $u_1$ and $u_2$ are distinct, and the regions $v_1$ and $v_2$ are distinct.
\item [(2)] There is no crossing other than the dealternator that is in the boundary of $u_1$ and $u_2$, and there is no crossing other than the dealternator that is in the boundary of $v_1$ and $v_2$.
\item [(3A)] There is no white region in $R$ that shares a crossing with each of $u_1$ and $u_2$.
\end{itemize}
An \textit{$A$-almost alternating link} is a nonalternating link with an $A$-almost alternating diagram. An almost alternating diagram is \textit{$B$-almost alternating} if it satisfies conditions (1) and (2) above as well as condition (3B).
\begin{itemize}
\item [(3B)] There is no black region in $R$ that shares a crossing with each of $v_1$ and $v_2$.
\end{itemize}
A \textit{$B$-almost alternating} link is a nonalternating link with a $B$-almost alternating diagram.
\end{definition}

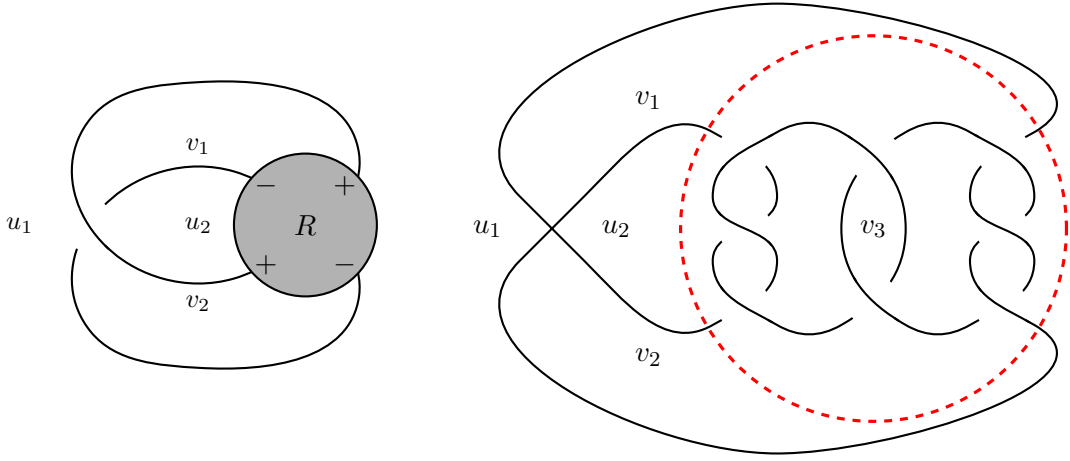
\begin{figure}
\[\begin{tikzpicture}[thick, scale = .95]

\begin{knot}[
	consider self intersections,
 	clip width = 5,
 	ignore endpoint intersections = true,
	end tolerance = 1pt
 ]
 \strand
 (4,0) to [curve through = {(4.75,1) .. (2.5,2) .. (1,1.5)  .. (2,-.75) .. (3,-.75)}] (4,0);
 \strand 
 (4,0) to [curve through = {(4.75,-1) .. (2.5, -2) .. (1,-1.5) .. (2,.75) .. (3,.75)}] (4,0);

\end{knot}

\fill[white!70!black] (4,0) circle (1cm);
\draw[thick] (4,0) circle (1cm);
\draw (4,0) node{$R$};
\draw (3.45,.55) node{$-$};
\draw (4.55,.55) node{$+$};
\draw (4.55,-.55) node{$-$};
\draw (3.45,-.55)node{$+$};

\draw (2.5,1.1) node{\small{$v_1$}};
\draw (2.5,-1.1) node{\small{$v_2$}};
\draw (2.5,0) node{\small{$u_2$}};
\draw (0,0) node{\small{$u_1$}};

\begin{scope}[xshift =  7 cm, yshift = -.5cm, scale = .9]

\begin{knot}[
    consider self intersections,
    clip width = 5,
    ignore endpoint intersections = true,
    end tolerance = 2pt
    ]
    \flipcrossings{1, 2, 4, 6, 8}
    \strand (0,0) to [out = 45, in = 225]
    (1,1) to [out=45, in =150]
    (3,2) to [out = -30, in =90]
    (4,1) to [out =270, in =90]
    (3,0) to [out = 270, in = 150]
    (4,-1) to [out =-30, in = 210]
    (5,-1) to [out = 30, in = 270]
    (6,.5) to [out = 90, in =-30]
    (5,2) to [out = 150, in = 30]
    (4,2) to [out = 210, in = 90]
    (3,1) to [out=270, in = 90]
    (4,0) to [out = 270, in = 30]
    (3,-1) to [out = 210, in = -45]
    (1,0) to [out = 135, in = -45]
    (0,1) to [out = 135, in = 180]
    (4,4) to [out = 0, in = 30]
    (8,2) to [out = 210, in = 90]
    (7,1) to [out = 270, in = 90]
    (8,0) to [out = 270, in = 30]
    (7,-1) to [out = 210, in = -30]
    (6,-1) to [out = 150, in = 270]
    (5,.5) to [out = 90, in = 210]
    (6,2) to [out = 30, in = 150]
    (7,2) to [out = -30, in = 90]
    (8,1) to [out = 270, in = 90]
    (7,0) to [out = 270, in = 150]
    (8,-1) to  [out = -30, in = 0]
    (4,-3) to [out = 180, in = 225]
    (0,0);
\end{knot}

\draw[very thick, dashed, red] (5.5,.5) circle (3cm);

\draw (-.5,.5) node{$u_1$};
\draw (1.5,.5) node{$u_2$};
\draw (2,2.5) node{$v_1$};
\draw (2,-1.5) node{$v_2$};
\draw (5.5,.5) node{$v_3$};
\end{scope}

\end{tikzpicture}\]
\caption{An arbitrary almost alternating diagram is on the left. The tangle $R$ is alternating, and the ``$+$'' and ``$-$'' labels have the same meaning as in Figure \ref{figure:tg1}. An $A$-almost alternating diagram is on the right where the dashed red circle encloses the alternating tangle $R$. Since the region $v_3$ shares crossings with both $v_1$ and $v_2$, the diagram is not $B$-almost alternating.}
\label{figure:aadiagram}
\end{figure}

Dasbach and Lowrance \cite{DasLow:Extremal} proved that every almost alternating link is $A$-almost alternating or $B$-almost alternating (or both).

A \textit{Kauffman state} of $D$ is the collection of simple closed curves obtained by resolving each crossing $\tikz[baseline=.6ex, scale = .4]{
\draw (0,0) -- (1,1);
\draw (1,0) -- (.7,.3);
\draw (.3,.7) -- (0,1);
}
~$ with an $A$-resolution $~ \tikz[baseline=.6ex, scale = .4]{
\draw[rounded corners = 1mm] (0,0) -- (.45,.5) -- (0,1);
\draw[rounded corners = 1mm] (1,0) -- (.55,.5) -- (1,1);
}~$or with a $B$-resolution $~ \tikz[baseline=.6ex, scale = .4]{
\draw[rounded corners = 1mm] (0,0) -- (.5,.45) -- (1,0);
\draw[rounded corners = 1mm] (0,1) -- (.5,.55) -- (1,1);
}~$. The \textit{all-$A$ state} of $D$ is the collection of curves obtained by choosing an $A$-resolution for every crossing of $D$, and similarly the \textit{all-$B$ state} of $D$ is the collection of curves obtained by choosing a $B$-resolution for every crossing of $D$. If no two arcs in the $A$-resolution (respectively $B$-resolution) of any crossing of the link diagram $D$ are contained in the same component of the all-$A$ (all-$B$) state of $D$, then $D$ is called \textit{$A$-adequate} (\textit{$B$-adequate}). A link is \textit{$A$-adequate} (respectively \textit{$B$-adequate}) if it has an $A$-adequate ($B$-adequate) diagram. A link diagram that is both $A$-adequate and $B$-adequate is called \textit{adequate}, and any link having such a diagram is also called \textit{adequate}. Kim \cite{Kim:Turaev} proved that every Turaev genus one link has a diagram as in Figure \ref{figure:tg1} that is either $A$-adequate, $B$-adequate, or almost alternating. 

\begin{definition}
A link is \textit{$A$-Turaev genus one} if it is nonalternating and has an $A$-almost alternating diagram or an $A$-adequate diagram whose Turaev surface is genus one. Similarly, a link is \textit{$B$-Turaev genus one} if it is nonalternating and has a $B$-almost alternating diagram or a $B$-adequate diagram whose Turaev surface is genus one. Every link that has Turaev genus is one is $A$-Turaev genus one, $B$-Turaev genus one, or both.
\end{definition}

The Khovanov homology of a link $L$ is a categorification of the Jones polynomial of $L$ \cite{Khovanov:Categorification}. If $R$ is a commutative ring with identity, then the Khovanov homology of $L$ with coefficients in $R$ is denoted $Kh(L;R)$. In this paper, the ring $R$ will most commonly be the integers $\mathbb{Z}$, in which case we denote the Khovanov homology by $Kh(L)$; occasionally, we consider when the ring $R$ is the rationals $\mathbb{Q}$.
There is a direct sum decomposition $Kh(L;R)\cong \bigoplus_{i,j\in\mathbb{Z}}Kh^{i,j}(L;R)$ where $Kh^{i,j}(L;R)$ is the summand in homological grading $i$ and polynomial grading $j$. Define 
\begin{alignat*}{4}
    j_{\min}(L) = &\;  \min\{j~|~Kh^{i,j}(L)\neq 0\},& 
    j_{\max}(L) = &\; \max\{j~|~Kh^{i,j}(L)\neq 0\},\\
    i_{\min}(L) = &\; \min\{i~|~Kh^{i,j}(L)\neq 0\},&
    i_{\max}(L)= & \; \max\{i~|~Kh^{i,j}(L)\neq 0\},\\
    \delta_{\min}(L) = & \; \min\{2i-j~|~Kh^{i,j}(L)\neq 0\},&~\text{and}~
    \delta_{\max}(L) = & \; \max \{2i-j~|~Kh^{i,j}(L)\neq 0\}.
\end{alignat*}
If $j_0$ is constant, then define $Kh^{*,j_0}(L) = \bigoplus_{i\in\mathbb{Z}}Kh^{i,j_0}(L)$. The statement $Kh^{*,j_0}(L)=Kh^{i_0,j_0}(L)$ means that the Khovanov homology of $L$ in polynomial grading $j_0$ is entirely supported in homological grading $i_0$. Our main theorem examines the Khovanov homology of a Turaev genus one link in its first or last two polynomial gradings.
\begin{theorem}
\label{theorem:main}
Let $L$ be a nonsplit link whose Turaev genus is one. If $L$ is $A$-Turaev genus one, then
\begin{enumerate}
    \item $Kh^{*,j_{\min}(L)}(L)\cong Kh^{i_{\min}(L),j_{\min}(L)}(L)\cong \mathbb{Z}$,
    \item $2i_{\min}(L)-j_{\min}(L)=\delta_{\min}(L)+2$, and
    \item $Kh^{i_{\min}(L)+2,j_{\min}(L)+2}(L)$ is trivial.
\end{enumerate}
If $L$ is $B$-Turaev genus one, then
\begin{enumerate}
    \item $Kh^{*,j_{\max}(L)}(L)\cong Kh^{i_{\max}(L),j_{\max}(L)}(L)\cong\mathbb{Z}$,
    \item $2i_{\max}(L) - j_{\max}(L) = \delta_{\max}(L)-2$, and
    \item $Kh^{i_{\max}(L)-2,j_{\max}(L)-2}(L)$ is trivial.
\end{enumerate}
\end{theorem}
In both the $A$- and $B$-Turaev genus one cases of Theorem \ref{theorem:main}, the first two statements were proved by Dasbach and Lowrance \cite{DasLow:Extremal}; our contribution is to prove the third statement. In Example \ref{example:14}, we show that the Theorem \ref{theorem:main} implies the $14$-crossing knot in Figure \ref{figure:14ex} has Turaev genus of at least two.


The Seifert genus, or $3$-genus, $g_3(K)$ of a knot $K$ is the minimum genus of any oriented surface embedded in $S^3$ whose boundary is $K$. The smooth $4$-genus $g_4(K)$ of $K$ is the minimum genus of any oriented, smooth surface in the $4$-ball whose boundary is $K$. Rasmussen \cite{Rasmussen:Slice} used Lee's deformation of Khovanov homology \cite{Lee:Endomorphism} to define the $s$ invariant of $K$. The Rasmussen $s$ invariant is a concordance invariant and yields a lower bound on the smooth $4$-genus $g_4(K)$ of a knot: $\frac{1}{2}|s(K)|\leq g_4(K)$. If certain conditions on the number of negative crossings in an $A$- or $B$-Turaev genus one diagram are met, then Theorem \ref{theorem:main} allows us to compute $s(K)$.

For any knot or link diagram $D$, define $s_A(D)$ and $s_B(D)$ to be the number of components in the all-$A$ and all-$B$ Kauffman states of $D$ respectively. Also, let $c(D)$ be the number of crossings in $D$, and let $c_+(D)$ and $c_-(D)$ denote the number of positive $~ \tikz[baseline=.6ex, scale = .4]{
\draw[->] (0,0) -- (1,1);
\draw[->] (.3,.7) -- (0,1);
\draw (.7,.3) -- (1,0);
}~$ and negative $~ \tikz[baseline=.6ex, scale = .4]{
\draw[->] (.7,.7) -- (1,1);
\draw[->] (1,0) -- (0,1);
\draw (0,0) -- (.3,.3);
}~$ crossings in $D$ respectively.

\begin{theorem}
\label{theorem:ras}
Let $D$ be a diagram of the knot $K$.
\begin{enumerate}
    \item If $D$ is $A$-adequate, has Turaev genus one, and $c_-(D)=2$, then 
    \[s(K)=c(D)-s_A(D)-1.\]
    \item If $D$ is $A$-almost alternating and $c_-(D)=3$, then 
    \[s(K)=c(D)-s_A(D)-2.\]
    \item If $D$ is $B$-adequate, has Turaev genus one, and $c_+(D)=2$, then 
    \[s(K)=-c(D)+s_B(D)+1.\]
    \item If $D$ is $B$-almost alternating and $c_+(D)=3$, then 
    \[s(K) = -c(D)+s_B(D)+2.\]
\end{enumerate}

\end{theorem}

The next result states that for all knots considered in Theorem \ref{theorem:ras}, the smooth $4$-genus is bounded above by $\frac{1}{2}|s(K)|+1$.
\begin{theorem}
\label{theorem:genus}
Let $K$ be a knot. Suppose that $K$ has a diagram $D$ satisfying one of the following four conditions.
\begin{enumerate}
    \item $D$ is $A$-adequate, has Turaev genus one, and $c_-(D)=2$,
    \item $D$ is $B$-adequate, has Turaev genus one, and $c_+(D)=2$,
    \item $D$ is $A$-almost alternating and $c_-(D)=3$, or
    \item $D$ is $B$-almost alternating and $c_+(D)=3.$
\end{enumerate}
Then
\[\frac{|s(K)|}{2} \leq g_4(K) \leq \frac{|s(K)|}{2}+1.\]
\end{theorem}

This paper is organized as follows. In Section \ref{section:Background} we review the construction of Khovanov homology. In Section \ref{section:Near}, we study the Khovanov homology of Turaev genus one links in the first and last two polynomial gradings and prove Theorem \ref{theorem:main}.  In Section \ref{section:genus}, we prove Theorems \ref{theorem:ras} and \ref{theorem:genus}.

\section{Background}
\label{section:Background}

In this section, we give background information on the Turaev genus of a link and Khovanov homology.

\subsection{The Turaev genus of a link} Let $L$ be a nonsplit link in $S^3$, and let $D$ be a diagram of $L$ on a two-sphere $S^2$ inside $S^3$. The \textit{all-$A$ state} (respectively \textit{all-$B$ state}) of $D$ is the set of curves in $S^2$ obtained by replacing each crossing  $\tikz[baseline=.6ex, scale = .4]{
\draw (0,0) -- (.3,.3);
\draw (.7,.7) -- (1,1);
\draw (0,1) -- (1,0);
}
~$ in $D$ with its $A$-resolution $~ \tikz[baseline=.6ex, scale = .4]{
\draw[rounded corners = 1mm] (0,0) -- (.5,.45) -- (1,0);
\draw[rounded corners = 1mm] (0,1) -- (.5,.55) -- (1,1);
}~$ (respectively with its $B$-resolution $~ \tikz[baseline=.6ex, scale = .4]{
\draw[rounded corners = 1mm] (0,0) -- (.45,.5) -- (0,1);
\draw[rounded corners = 1mm] (1,0) -- (.55,.5) -- (1,1);
}~$). The Turaev surface $\Sigma_D$ of $D$ is constructed as follows. Push the all-$A$ resolution slightly to one side of $S^2$ and the all-$B$ resolution slightly to the other side. Build a cobordism between the all-$A$ state and the all-$B$ state that consists of bands away from the crossings and saddles near the crossings, as in Figure \ref{figure:saddle}. The Turaev surface $\Sigma_D$ is obtained by capping off the boundary components of this cobordism with disks.
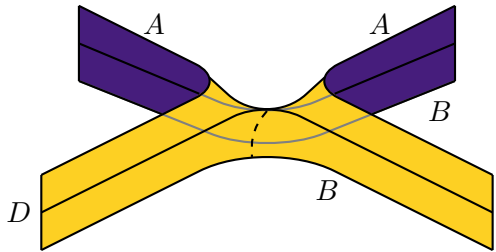
\begin{figure}[h]
$$\begin{tikzpicture}
\begin{scope}[thick]
\draw [rounded corners = 10mm] (0,0) -- (3,1.5) -- (6,0);
\draw (0,0) -- (0,1);
\draw (6,0) -- (6,1);
\draw [rounded corners = 5mm] (0,1) -- (2.5, 2.25) -- (0.5, 3.25);
\draw [rounded corners = 5mm] (6,1) -- (3.5, 2.25) -- (5.5,3.25);
\draw [rounded corners = 5mm] (0,.5) -- (3,2) -- (6,.5);
\draw [rounded corners = 7mm] (2.23, 2.3) -- (3,1.6) -- (3.77,2.3);
\draw (0.5,3.25) -- (0.5, 2.25);
\draw (5.5,3.25) -- (5.5, 2.25);
\end{scope}

\begin{pgfonlayer}{background2}
\fill [lsugold]  [rounded corners = 10 mm] (0,0) -- (3,1.5) -- (6,0) -- (6,1) -- (3,2) -- (0,1); 
\fill [lsugold] (6,0) -- (6,1) -- (3.9,2.05) -- (4,1);
\fill [lsugold] (0,0) -- (0,1) -- (2.1,2.05) -- (2,1);
\fill [lsugold] (2.23,2.28) --(3.77,2.28) -- (3.77,1.5) -- (2.23,1.5);

\fill [white, rounded corners = 7mm] (2.23,2.3) -- (3,1.6) -- (3.77,2.3);
\fill [lsugold] (2,2) -- (2.3,2.21) -- (2.2, 1.5) -- (2,1.5);
\fill [lsugold] (4,2) -- (3.7, 2.21) -- (3.8,1.5) -- (4,1.5);
\end{pgfonlayer}

\begin{pgfonlayer}{background4}
\fill [lsupurple] (.5,3.25) -- (.5,2.25) -- (3,1.25) -- (2.4,2.2);
\fill [rounded corners = 5mm, lsupurple] (0.5,3.25) -- (2.5,2.25) -- (2,2);
\fill [lsupurple] (5.5,3.25) -- (5.5,2.25) -- (3,1.25) -- (3.6,2.2);
\fill [rounded corners = 5mm, lsupurple] (5.5, 3.25) -- (3.5,2.25) -- (4,2);
\end{pgfonlayer}

\draw [thick] (0.5,2.25) -- (1.6,1.81);
\draw [thick] (5.5,2.25) -- (4.4,1.81);
\draw [thick] (0.5,2.75) -- (2.1,2.08);
\draw [thick] (5.5,2.75) -- (3.9,2.08);

\begin{pgfonlayer}{background}
\draw [black!50!white, rounded corners = 8mm, thick] (0.5, 2.25) -- (3,1.25) -- (5.5,2.25);
\draw [black!50!white, rounded corners = 7mm, thick] (2.13,2.07) -- (3,1.7)  -- (3.87,2.07);
\end{pgfonlayer}
\draw [thick, dashed, rounded corners = 2mm] (3,1.85) -- (2.8,1.6) -- (2.8,1.24);
\draw (0,0.5) node[left]{$D$};
\draw (1.5,3) node{$A$};
\draw (4.5,3) node{$A$};
\draw (3.8,.8) node{$B$};
\draw (5.3, 1.85) node{$B$};
\end{tikzpicture}$$
\caption{A saddle transitions between the all-$A$ and all-$B$ states in a neighborhood of each crossing of $D$.}
\label{figure:saddle}
\end{figure}

The genus $g_T(D)$ of the Turaev surface of $D$ is 
\[g(\Sigma_D) = \frac{1}{2}\left(2+c(D)-s_A(D)-s_B(D)\right),\]
where $c(D)$ is the number of crossings in $D$ and $s_A(D)$ and $s_B(D)$ are the number of components in the all-$A$ and all-$B$ states of $D$ respectively.
The \textit{Turaev genus} $g_T(L)$ of a nonsplit link $L$ is 
\[g_T(L) = \min\{g_T(D)~|~D~\text{is a diagram of}~L\}.\]

Turaev \cite{Turaev} proved that $g_T(D)=0$ if and only if $D$ is a connected sum of alternating diagrams, and consequently, $g_T(L)=0$ if and only if $L$ is an alternating link. Dasbach, Futer, Kalfagianni, Lin, and Stoltzfus \cite{DFKLS1,DFKLS2} showed that the Turaev surface is a Heegaard surface in $S^3$ on which the link has an alternating projection and gave Turaev surface models for computing the Jones polynomial and determinant of the link. Dasbach and Lowrance \cite{DasLow:Approach} generalized the Turaev surface model for the Jones polynomial to Khovanov homology. Lower bounds for the Turaev genus of a link arise from the Jones polynomial \cite{DasLow:Invariants, LowSp:Jones}, from Khovanov homology \cite{CKS:Graphs, DasLow:Extremal}, from knot Floer homology \cite{Lowrance:HFK}, and from the differences between certain concordance invariants \cite{DasLow:Concordance, Kim:Concordance}. Computations of Turaev genus include the $(3,q)$-torus links \cite{AbeKishimoto}, many other closed $3$-braids \cite{Lowrance:Width}, adequate knots \cite{Abe:Adequate}, many torus links on $4$, $5$, or $6$ strands \cite{Lowrance:Torus}, and connected sums of certain pretzel links \cite{Kim:Concordance}.

Kim \cite{Kim:Turaev} and Armond and Lowrance \cite{ArmLow} proved that if $L$ is a nonsplit Turaev genus one link, then it has a diagram as in Figure \ref{figure:tg1}. Links with Turaev genus one include nonalternating pretzel links on arbitrarily many strands, nonalternating Montesinos links, and almost alternating links. Although there are many diagrams with $g_T(D)=1$ that are not almost alternating, it is an open question whether every link with Turaev genus one has some almost alternating diagram; see \cite{Low:AltDist, Lowrance:Encyclopedia} for an in-depth discussion.

For more complete surveys of the Turaev surface and the Turaev genus of links see Champanerkar and Kofman \cite{CK:Survey} and Kim and Kofman \cite{KimKofman}.

\subsection{Khovanov homology}
In this subsection, we construct the Khovanov homology of a link and recall some results that will be useful for our purposes. Khovanov homology was originally defined by Khovanov \cite{Khovanov:Categorification}. The construction that follows mixes elements of Viro \cite{Viro:Khovanov} and Bar-Natan \cite{BN:Khovanov}. 

A \textit{Kauffman state} of the link diagram $D$  is the set of simple closed curves in $S^2$ obtained by replacing each crossing  $\tikz[baseline=.6ex, scale = .4]{
\draw (0,0) -- (.3,.3);
\draw (.7,.7) -- (1,1);
\draw (0,1) -- (1,0);
}
~$ in $D$ by either an $A$-resolution $~ \tikz[baseline=.6ex, scale = .4]{
\draw[rounded corners = 1mm] (0,0) -- (.5,.45) -- (1,0);
\draw[rounded corners = 1mm] (0,1) -- (.5,.55) -- (1,1);
}~$ or a $B$-resolution $~ \tikz[baseline=.6ex, scale = .4]{
\draw[rounded corners = 1mm] (0,0) -- (.45,.5) -- (0,1);
\draw[rounded corners = 1mm] (1,0) -- (.55,.5) -- (1,1);
}~$. An \textit{enhanced state} $S$ of $D$ is a Kauffman state of $D$ together with a labeling of ``$+$" or ``$-$" on each component of the Kauffman state. Define $a(S)$ and $b(S)$ to be the number of $A$- and $B$-resolutions in the enhanced state $S$ respectively, and define $\theta(S)$ to be the difference between the number of $+$ and $-$ labels of $S$.

Let $\mathcal{S}(D)$ be the set of enhanced states of $D$. Define the homological grading $i:\mathcal{S}(D)\to\mathbb{Z}$ and the polynomial grading $j:\mathcal{S}(D)\to\mathbb{Z}$ by $i(S)=b(S)$ and $j(S) = b(S) + \theta(S)$ where $S$ is an enhanced state of $D$. If $R$ is a commutative ring with identity, then define $CKh(D;R)$ to be the free $R$-module with basis $\mathcal{S}(D)$. Let $CKh^{i,j}(D;R)$ be the submodule of $CKh(D;R)$ generated by enhanced states in homological grading $i$ and polynomial grading $j$. If the ring $R$ is the integers $\mathbb{Z}$, then we drop it from the notation so that $CKh(D)=CKh(D;\mathbb{Z})$.

Changing an $A$-resolution to a $B$-resolution in a Kauffman state either merges two components together or splits one component into two. The incidence $\langle S:S'\rangle$ between two enhanced states $S$ and $S'$ is $\pm 1$ if $S'$ can be obtained from $S$ by changing one $A$-resolution to a $B$-resolution (ignoring labels) and if the labels on the two components being merged or the component being split are as in Figure \ref{figure:mergesplit}. Otherwise the incidence number $\langle S : S'\rangle$ is zero. In the case, where the incidence number is nonzero, its sign is determined by an arbitrary numbering of the crossings $1,\dots, c$. Suppose that crossing that is changed from an $A$-resolution to a $B$-resolution is crossing $k$. Then $\langle S:S'\rangle = -1$ if the number of $B$-resolutions in $S$ before crossing $k$ is odd and $\langle S: S'\rangle = 1$ if the number of $B$-resolutions in $S$ before crossing $k$ is even.
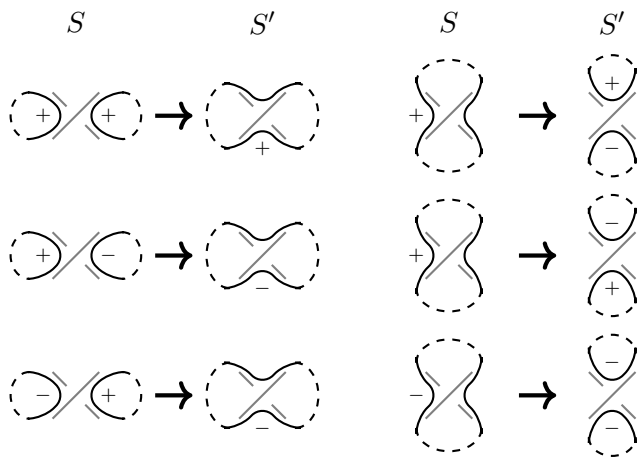
\begin{figure}[ht]
\[\begin{tikzpicture}[thick, scale = .62]

\draw (1.5,2.5) node{$S$};
\draw (5.5,2.5) node{$S'$};
\draw (9.5,2.5) node{$S$};
\draw (13,2.5) node{$S'$};


\draw[white!50!black] (1,0) -- (2,1);
\draw[white!50!black] (2,0) -- (1.7,.3);
\draw[white!50!black] (1,1) -- (1.3,.7);

\draw[rounded corners = 2mm] (.5,1) -- (.8,1) -- (1.3,.5) -- (.8,0) -- (.5,0);
\draw[rounded corners = 2mm] (2.5,1) -- (2.2,1) -- (1.7,.5) -- (2.2,0) -- (2.5,0);
\draw[dashed, thick] (.5,1) arc (90:270:.4cm and .5cm);
\draw[dashed, thick] (2.5,1) arc (90:-90:.4cm and .5cm);
\draw(.8,.5) node{\tiny{$+$}};
\draw(2.2,.5) node{\tiny{$+$}};

\begin{scope}[xshift = 4cm]
\draw[white!50!black] (1,0) -- (2,1);
\draw[white!50!black] (2,0) -- (1.7,.3);
\draw[white!50!black] (1,1) -- (1.3,.7);

\draw[rounded corners = 2mm] (.8,1.2) -- (1,1.2) -- (1.5,.7) -- (2,1.2) -- (2.2,1.2);
\draw[rounded corners = 2mm] (.8,-.2) -- (1,-.2) -- (1.5,.3) -- (2,-.2) -- (2.2,-.2);
\draw[dashed, thick] (.8,1.2) arc (90:270:.5cm and .7cm);
\draw[dashed, thick] (2.2,1.2) arc (90:-90:.5cm and .7cm); 
\draw (1.5,-.2) node{\tiny{$+$}};
\end{scope}

\draw[ultra thick, ->] (3.2,.5) -- (4,.5);

\begin{scope}[yshift = -3cm]

\draw[white!50!black] (1,0) -- (2,1);
\draw[white!50!black] (2,0) -- (1.7,.3);
\draw[white!50!black] (1,1) -- (1.3,.7);

\draw[rounded corners = 2mm] (.5,1) -- (.8,1) -- (1.3,.5) -- (.8,0) -- (.5,0);
\draw[rounded corners = 2mm] (2.5,1) -- (2.2,1) -- (1.7,.5) -- (2.2,0) -- (2.5,0);
\draw[dashed, thick] (.5,1) arc (90:270:.4cm and .5cm);
\draw[dashed, thick] (2.5,1) arc (90:-90:.4cm and .5cm);
\draw(.8,.5) node{\tiny{$+$}};
\draw(2.2,.5) node{\tiny{$-$}};

\begin{scope}[xshift = 4cm]
\draw[white!50!black] (1,0) -- (2,1);
\draw[white!50!black] (2,0) -- (1.7,.3);
\draw[white!50!black] (1,1) -- (1.3,.7);

\draw[rounded corners = 2mm] (.8,1.2) -- (1,1.2) -- (1.5,.7) -- (2,1.2) -- (2.2,1.2);
\draw[rounded corners = 2mm] (.8,-.2) -- (1,-.2) -- (1.5,.3) -- (2,-.2) -- (2.2,-.2);
\draw[dashed, thick] (.8,1.2) arc (90:270:.5cm and .7cm);
\draw[dashed, thick] (2.2,1.2) arc (90:-90:.5cm and .7cm); 
\draw (1.5,-.2) node{\tiny{$-$}};
\end{scope}

\draw[ultra thick, ->] (3.2,.5) -- (4,.5);
\end{scope}

\begin{scope}[yshift = -6cm]

\draw[white!50!black] (1,0) -- (2,1);
\draw[white!50!black] (2,0) -- (1.7,.3);
\draw[white!50!black] (1,1) -- (1.3,.7);

\draw[rounded corners = 2mm] (.5,1) -- (.8,1) -- (1.3,.5) -- (.8,0) -- (.5,0);
\draw[rounded corners = 2mm] (2.5,1) -- (2.2,1) -- (1.7,.5) -- (2.2,0) -- (2.5,0);
\draw[dashed, thick] (.5,1) arc (90:270:.4cm and .5cm);
\draw[dashed, thick] (2.5,1) arc (90:-90:.4cm and .5cm);
\draw(.8,.5) node{\tiny{$-$}};
\draw(2.2,.5) node{\tiny{$+$}};

\begin{scope}[xshift = 4cm]
\draw[white!50!black] (1,0) -- (2,1);
\draw[white!50!black] (2,0) -- (1.7,.3);
\draw[white!50!black] (1,1) -- (1.3,.7);

\draw[rounded corners = 2mm] (.8,1.2) -- (1,1.2) -- (1.5,.7) -- (2,1.2) -- (2.2,1.2);
\draw[rounded corners = 2mm] (.8,-.2) -- (1,-.2) -- (1.5,.3) -- (2,-.2) -- (2.2,-.2);
\draw[dashed, thick] (.8,1.2) arc (90:270:.5cm and .7cm);
\draw[dashed, thick] (2.2,1.2) arc (90:-90:.5cm and .7cm); 
\draw (1.5,-.2) node{\tiny{$-$}};
\end{scope}

\draw[ultra thick, ->] (3.2,.5) -- (4,.5);
\end{scope}


\begin{scope}[xshift = 8cm]

\draw[white!50!black] (1,0) -- (2,1);
\draw[white!50!black] (2,0) -- (1.7,.3);
\draw[white!50!black] (1,1) -- (1.3,.7);

\draw[rounded corners = 2mm] (.8,1.2) -- (.8,1) -- (1.3,.5) -- (.8,0) -- (.8,-.2);
\draw[rounded corners = 2mm] (2.2,1.2) -- (2.2,1) --(1.7,.5) -- (2.2,0) -- (2.2,-.2);
\draw[dashed, thick] (.8,1.2) arc (180:0:.7cm and .5cm);
\draw[dashed, thick] (.8,-.2) arc (-180:0:.7cm and .5cm);
\draw (.8,.5) node{\tiny{$+$}};

\begin{scope}[xshift = 3.5cm]
\draw[white!50!black] (1,0) -- (2,1);
\draw[white!50!black] (2,0) -- (1.7,.3);
\draw[white!50!black] (1,1) -- (1.3,.7);

\draw[rounded corners = 2mm] (1,1.4) -- (1,1.2) -- (1.5,.7)  -- (2,1.2) -- (2,1.4);
\draw[rounded corners = 2mm] (1,-.4) -- (1,-.2) -- (1.5,.3) -- (2,-.2) -- (2,-.4);
\draw [dashed, thick] (1,1.4) arc (180:0:.5cm and .4cm);
\draw [dashed, thick] (1,-.4) arc (-180:0:.5cm and .4cm);

\draw (1.5,1.2) node{\tiny{$+$}};
\draw (1.5,-.2) node{\tiny{$-$}};

\end{scope}

\draw[ultra thick, ->] (3,.5) -- (3.8,.5);

\end{scope}


\begin{scope}[xshift = 8cm, yshift = -3cm]

\draw[white!50!black] (1,0) -- (2,1);
\draw[white!50!black] (2,0) -- (1.7,.3);
\draw[white!50!black] (1,1) -- (1.3,.7);

\draw[rounded corners = 2mm] (.8,1.2) -- (.8,1) -- (1.3,.5) -- (.8,0) -- (.8,-.2);
\draw[rounded corners = 2mm] (2.2,1.2) -- (2.2,1) --(1.7,.5) -- (2.2,0) -- (2.2,-.2);
\draw[dashed, thick] (.8,1.2) arc (180:0:.7cm and .5cm);
\draw[dashed, thick] (.8,-.2) arc (-180:0:.7cm and .5cm);
\draw (.8,.5) node{\tiny{$+$}};

\begin{scope}[xshift = 3.5cm]
\draw[white!50!black] (1,0) -- (2,1);
\draw[white!50!black] (2,0) -- (1.7,.3);
\draw[white!50!black] (1,1) -- (1.3,.7);

\draw[rounded corners = 2mm] (1,1.4) -- (1,1.2) -- (1.5,.7)  -- (2,1.2) -- (2,1.4);
\draw[rounded corners = 2mm] (1,-.4) -- (1,-.2) -- (1.5,.3) -- (2,-.2) -- (2,-.4);
\draw [dashed, thick] (1,1.4) arc (180:0:.5cm and .4cm);
\draw [dashed, thick] (1,-.4) arc (-180:0:.5cm and .4cm);

\draw (1.5,1.2) node{\tiny{$-$}};
\draw (1.5,-.2) node{\tiny{$+$}};

\end{scope}

\draw[ultra thick, ->] (3,.5) -- (3.8,.5);

\end{scope}


\begin{scope}[xshift = 8cm, yshift = -6cm]

\draw[white!50!black] (1,0) -- (2,1);
\draw[white!50!black] (2,0) -- (1.7,.3);
\draw[white!50!black] (1,1) -- (1.3,.7);

\draw[rounded corners = 2mm] (.8,1.2) -- (.8,1) -- (1.3,.5) -- (.8,0) -- (.8,-.2);
\draw[rounded corners = 2mm] (2.2,1.2) -- (2.2,1) --(1.7,.5) -- (2.2,0) -- (2.2,-.2);
\draw[dashed, thick] (.8,1.2) arc (180:0:.7cm and .5cm);
\draw[dashed, thick] (.8,-.2) arc (-180:0:.7cm and .5cm);
\draw (.8,.5) node{\tiny{$-$}};

\begin{scope}[xshift = 3.5cm]
\draw[white!50!black] (1,0) -- (2,1);
\draw[white!50!black] (2,0) -- (1.7,.3);
\draw[white!50!black] (1,1) -- (1.3,.7);

\draw[rounded corners = 2mm] (1,1.4) -- (1,1.2) -- (1.5,.7)  -- (2,1.2) -- (2,1.4);
\draw[rounded corners = 2mm] (1,-.4) -- (1,-.2) -- (1.5,.3) -- (2,-.2) -- (2,-.4);
\draw [dashed, thick] (1,1.4) arc (180:0:.5cm and .4cm);
\draw [dashed, thick] (1,-.4) arc (-180:0:.5cm and .4cm);

\draw (1.5,1.2) node{\tiny{$-$}};
\draw (1.5,-.2) node{\tiny{$-$}};

\end{scope}

\draw[ultra thick, ->] (3,.5) -- (3.8,.5);

\end{scope}
\end{tikzpicture}\]
        \caption{\textbf{Left.} Merges with incidence number $\pm 1$. \textbf{Right.} Splits with incidence number $\pm 1$.}
        \label{figure:mergesplit}
\end{figure}

The differential $d^{i,j}:CKh^{i,j}(D;R)\to CKh^{i+1,j}(D;R)$ is defined by 
\[d^{i,j}(S)=\sum_{S'\in \mathcal{S}(D)} \langle S: S'\rangle S,\]
and extending linearly. The unshifted Khovanov homology $\ukh^{i,j}(D;R)$ is defined as 
\[\ukh^{i,j}(D;R) = \operatorname{ker} d^{i,j} / \operatorname{im} d^{i-1,j}.\]

Suppose that $D$ is a diagram of a link $L$ with $c_+$ positive crossings and $c_-$ negative crossings. The \textit{Khovanov homology} of $D$, denoted by $Kh(D;R)$ or equivalently by $Kh(L;R)$, is defined as $Kh(D;R)=\bigoplus_{i,j\in\mathbb{Z}} Kh^{i,j}(D;R)$ where $Kh^{i,j}(D;R) = \ukh^{i+c_-,j+c_+-2c_-}(D;R).$ 

The Khovanov homology of a link is a link invariant, but the unshifted Khovanov homology $\ukh(D)$ is not. Reidemeister moves can induce grading shifts in $\ukh(D;R)$. We record those shifts in the following lemma.
\begin{lemma}
\label{lemma:Reidemeister}
The grading shifts of unshifted Khovanov homology $\ukh(D;R)$ induced by Reidemeister moves and flypes are as follows.
  \begin{itemize}
        \item Positive Reidemeister move I: $\Kh^{i,j}\left(~
\tikz[baseline=.6ex, scale = .6]{
\draw[thick] (0,0) -- (0,1);
}~\right) \cong \Kh^{i,j-1}\left(~
\tikz[baseline = .6ex, scale = .6]{
\draw[thick, rounded corners = .8mm] (0,0) -- (0,.5) -- (.2,.7) -- (.4,.5) -- (.2,.3) -- (.1,.4);
\draw[thick, rounded corners = .8mm] (-.03,.55) -- (-.1,.7) -- (0,1);
}\right).$
\item Negative Reidemeister move I: $\Kh^{i,j}\left(~
\tikz[baseline=.6ex, scale = .6]{
\draw[thick] (0,0) -- (0,1);
}~\right) \cong \Kh^{i+1,j+2}\left(~
\tikz[baseline = .6ex, scale = .6]{
\draw[thick, rounded corners = .8mm] (0,1) -- (0,.5) -- (.2,.3) -- (.4,.5) -- (.2,.7) -- (.1,.6);
\draw[thick, rounded corners = .8mm] (-.03,.45) -- (-.1,.3) -- (0,0);
}\right).$ 
\item Reidemeister move II: $\Kh^{i,j}\left(~
\tikz[baseline=.6ex, scale = .6]{
\draw[thick] (0,0) -- (0,1);
\draw[thick] (0.5,0) -- (0.5,1);
}~\right) \cong \Kh^{i+1,j+1}\left(~
\tikz[baseline=.6ex, scale = .6]{
\draw[thick, rounded corners = 2mm] (0,0) -- (.5,.5) -- (0,1);
\draw[thick] (.5,1) -- (.3,.8);
\draw[thick] (.3,.2) -- (.5,0);
\draw[thick, rounded corners = 1mm] (.2,.7) -- (0,.5) -- (.2,.3);
}~\right).$

\item Reidemeister move III or a flype does not change unshifted Khovanov homology $\ukh(D)$.
    \end{itemize}
\end{lemma}

The unshifted Khovanov homology of the mirror image $\overline{D}$ of $D$ satisfies
\begin{align}
\begin{split}
\label{equation:Mirror}
\underline{Kh}^{i,j}(\overline{D};\mathbb{Q}) \cong & \; \underline{Kh}^{c-i,c-j}(D;\mathbb{Q})~\text{and}\\
\operatorname{Tor}\underline{Kh}^{i,j}(\overline{D}) \cong & \; \operatorname{Tor} \underline{Kh}^{c-i+1,c-j}(D),
\end{split}
\end{align}
where $c$ is the number of crossings in $D$.

Let $D_A$ and $D_B$ be the diagrams obtained from $D$ by choosing an $A$- and $B$-resolution respectively at a specified crossing. There are natural inclusions $\iota_A:\mathcal{S}(D_A)\hookrightarrow\mathcal{S}(D)$ and $\iota_B:\mathcal{S}(D_B)\hookrightarrow\mathcal{S}(D)$ such that $\iota:\mathcal{S}(D_A)\sqcup \mathcal{S}(D_B) \to \mathcal{S}(D)$ defined by $\iota(S) =\iota_A(S)$ if $S\in \mathcal{S}(D_A)$ and $\iota(S)=\iota_B(S)$ if $S\in\mathcal{S}(D_B)$ is a bijection. Define functions $f:CKh(D_B)\to CKh(D)$ and $g:CKh(D)\to CKh(D_A)$ by $f(S)=\iota(S)$ and $g(S)=\iota^{-1}(S)$ when $\iota^{-1}(S)\in S_A$ and $g(S)=0$ otherwise, and extending linearly. There is a short exact sequence of chain complexes
\[0 \rightarrow CKh(D_B) \rightarrow CKh(D) \rightarrow CKh(D_A) \rightarrow 0\]
yielding the following long exact sequence in Khovanov homology \cite{Khovanov:Categorification}.
\begin{theorem}[Khovanov]
\label{theorem:LES}
Let $D$ be a link diagram with $D_A$ and $D_B$ being the $A$- and $B$-resolutions respectively of $D$ at a chosen crossing. For each $j\in\mathbb{Z}$, there is a long exact sequence of unshifted Khovanov homology
\begin{equation}
\cdots \to \underline{Kh}^{i-1,j-1}(D_B) \xrightarrow{f_*} \underline{Kh}^{i,j}(D) \xrightarrow{g_*} \underline{Kh}^{i,j}(D_A) \xrightarrow{\partial} \underline{Kh}^{i,j-1}(D_B)\to\cdots
\label{equation:LES}
\end{equation}
where $\partial$ is the boundary map in the snake lemma.
\end{theorem}

Lee \cite{Lee:Endomorphism} proved that the Khovanov homology of a nonsplit alternating link is supported in two adjacent $\delta = 2i-j$ gradings. More specifically, she showed that if $L$ is a nonsplit alternating link, then $Kh^{i,j}(L)$ is trivial unless $2i-j = \sigma(L)\pm 1$ where $\sigma(L)$ is the signature of $L$.  Asaeda and Przytycki \cite{AP} proved that the Khovanov homology of a nonsplit almost alternating link is supported in at most three adjacent $\delta=2i-j$ gradings; Champanerkar and Kofman \cite{CK:Span} gave a different approach to this result using a spanning tree model for Khovanov homology. Champanerkar, Kofman, and Stoltzfus \cite{CKS:Graphs} proved that the Khovanov homology of a nonsplit Turaev genus one link  is supported in at most three adjacent $\delta=2i-j$ gradings. Dasbach and Lowrance \cite{DasLow:Concordance} showed that if $L$ is nonsplit, has Turaev genus one, and $Kh^{i,j}(L)$ is nontrivial, then
\begin{equation}
\label{equation:diagonal}
s_A(D) - 2 \leq 2i -j \leq s_A(D)+2.
\end{equation}

\section{Near extremal Khovanov homology}
\label{section:Near}
In this section, we prove Theorem \ref{theorem:main}. Because the mirror image of an $A$-Turaev genus one diagram is a $B$-Turaev genus one diagram, we focus our attention on $A$-Turaev genus one diagrams. Our proof of Theorem \ref{theorem:main} breaks down into two cases; the link diagram is $A$-almost alternating or $A$-adequate and Turaev genus one.

\subsection{$A$-almost alternating diagrams}

Theorem \ref{theorem:almostalternating} below implies Theorem \ref{theorem:main} for almost alternating links. Its proof is broken up over several lemmas; the strategy of the proof is as follows. Let $D$ be an $A$-almost alternating diagram labeled as in Figure \ref{figure:aadiagram}. We examine the long exact sequence obtained by choosing a crossing in the tangle $R$ and taking the $A$- and $B$-resolutions of $D$ at that crossing, called $D_A$ and $D_B$ respectively. There are four cases to consider.
\begin{enumerate}
\item For every crossing in $R$, both $D_A$ and $D_B$ are not $A$-almost alternating.

\item The tangle $R$ contains a crossing such that $D_A$ is $A$-almost alternating, but $D_B$ is not $A$-almost alternating.

\item The tangle $R$ contains a crossing such that $D_B$ is $A$-almost alternating, but $D_A$ is not $A$-almost alternating.

\item The tangle $R$ contains a crossing such that both $D_A$ and $D_B$ are $A$-almost alternating.

\end{enumerate}

Lemma \ref{lemma:base} proves the desired result when $D$ satisfies case (1). For the remaining three cases, we consider the long exact sequence in Khovanov homology for $D$, $D_A$, and $D_B$. The relevant portion of the long exact sequence is 
\[ \cdots\to\ukh^{2,2-s_A(D_B)}(D_B)\to\ukh^{3,4-s_A(D)}(D)\to\ukh^{3,4-s_A(D_A)}(D_A)\to\cdots. \]
Lemma \ref{lemma:path2} proves that $\ukh^{2,2-s_A(D_B)}(D_B)$ is trivial. The remainder of the proof argues on a case-by-case basis why $\ukh^{3,4-s_A(D_A)}(D_A)$ is also trivial, and so consequently $\ukh^{3,4-s_A(D)}(D)$ is trivial as well. In certain instances, it is easier to assume that the crossing where the resolution takes place is not part of a twist region of length greater than one. Lemma \ref{lemma:twist} allows us to lift the result from a resolution where the twist region is length one to a resolution where the twist region has arbitrary length.

Several of the proofs below use checkerboard graphs of link diagrams. The complementary regions of a knot diagram can be colored black and white so that at every crossing, exactly two of the four regions are white and so that no two regions that share an edge are the same color. A \textit{checkerboard graph} of a knot diagram $D$ is the graph whose vertices correspond to either the white or the black regions, whose edges correspond to the crossings, and where two vertices are connected by an edge corresponding to a crossing if the two regions corresponding to those vertices have the crossing in their border.

Lemma \ref{lemma:base} shows that if $D$ is $A$-almost alternating, then $\ukh^{3,4-s_A(D)}(D)$ is trivial when neither $D_A$ nor $D_B$ are $A$-almost alternating for every crossing in the alternating tangle $R$. It is one of the cases in the proof of Theorem \ref{theorem:almostalternating}. 
\begin{lemma}
\label{lemma:base}
Let $D$ be an $A$-almost alternating diagram such that for every crossing in the alternating tangle $R$, neither of the resolutions $D_A$ nor $D_B$ of that crossing are $A$-almost alternating. Then $\ukh^{3,4-s_A(D)}(D)$ is trivial.
\end{lemma}
\begin{proof}
Let $G$ be the checkerboard graph that contains vertices $u_1$ and $u_2$, and let $\overline{G}$ be the checkerboard graph of $D$ that contains vertices $v_1$ and $v_2$. Fix a crossing $x$ in the alternating tangle $R$ with resolutions $D_A$ and $D_B$ of $D$. By assumption, neither $D_A$ nor $D_B$ is $A$-almost alternating. Therefore the edge in $G$ associated with the crossing $x$ lies on a path of length three between $u_1$ and $u_2$, and the edge in $\overline{G}$ associated with the crossing $x$ lies on a path of length two between $v_1$ and $v_2$.  Let $G'$ and $\overline{G}'$ be the simplifications of $G$ and $\overline{G}$ respectively. The simplification $G'$ of $G$ is the graph that has the same vertex set as $G$ but multiple edges in $G$ are replaced with a single edge in $G'$. If $G'$ contains more than two paths of length three between $u_1$ and $u_2$, then $\overline{G}'$, and hence $\overline{G}$, does not contain a path of length two between $v_1$ and $v_2$. Therefore, $G'$ is one of the four graphs in the first column of Figure \ref{figure:BaseCase}.
\begin{figure}
\[\begin{tikzpicture}
\fill (0,0) circle (.1cm);
\fill (1,0) circle (.1cm);
\fill (2,0) circle (.1cm);
\fill (3,0) circle (.1cm);

\draw (0,0) -- (3,0);

\draw (0,0) arc (270:90:0.4);
\draw (3,0) arc (-90:90:0.4);
\draw (0,.8) -- (3,.8);

\draw (-.3,-.3) node{$u_2$};
\draw (3.3,-.3) node{$u_1$};

\begin{scope}[xshift = 5cm]
\fill (0,0) circle (.1cm);
\fill (1,0) circle (.1cm);
\fill (2,0) circle (.1cm);
\fill (3,0) circle (.1cm);

\draw (0,0) to [out = 60, in = 120] (1,0);
\draw (0,0) to [out = -60, in = -120] (1,0);

\draw (1,0) to [out = 60, in = 120] (2,0);
\draw (1,0) to [out = -60, in = -120] (2,0);

\draw (2,0) to [out = 60, in = 120] (3,0);
\draw (2,0) to [out = -60, in = -120] (3,0);

\draw (0,0) arc (270:90:0.4);
\draw (3,0) arc (-90:90:0.4);
\draw (0,.8) -- (3,.8);

\draw[blue] (1.5,.5) -- (1.5,-.5);
\draw[blue] (1.5,.5) to [out = 180, in = 90] (0.5,0) to [out = 270, in = 180] (1.5,-.5) to [out = 0, in = 270] (2.5,0) to [out = 90, in = 0] (1.5,.5);

\draw[blue] (1.5,.5) to [out = 90, in = 0] (0,1) to [out = 180, in = 90] (-.7,0) to [out = 270, in = 180] (0,-1) to [out = 0, in = 270] (1.5,-.5);

\fill[white] (1.5,.5) circle (.1cm);
\draw (1.5,.5) circle (.1cm);

\fill[white] (1.5,-.5) circle (.1cm);
\draw (1.5,-.5) circle (.1cm);

\fill[white] (.5,0) circle (.1cm);
\draw (.5,0) circle (.1cm);

\fill[white] (1.5,0) circle (.1cm);
\draw (1.5,0) circle (.1cm);

\fill[white] (2.5,0) circle (.1cm);
\draw (2.5,0) circle (.1cm);

\draw (-.3,-.3) node{$u_2$};
\draw (3.3,-.3) node{$u_1$};

\draw (1.8,.64) node{$v_1$};
\draw (1.8, -.67) node{$v_2$};

\end{scope}

\begin{scope}[xshift = 10cm]

\begin{knot}[
    consider self intersections,
    clip width = 4,
    ignore endpoint intersections = true,
    end tolerance = 2pt
    ]
    \flipcrossings{3, 7, 4}
    \strand (1.5,1.5) to [out = 180, in = 90]
    (-.5,.5) to [out = 270, in = 180]
    (.5,-.5) to [out = 0, in = 270]
    (1.25,0) to [out = 90, in = 0]
    (.5,.5) to [out = 180, in = 90]
    (-.5,-.5) to [out = 270, in = 180]
    (1.5,-1.5) to [out = 0, in = 270]
    (3.5,-.75) to [out = 90, in = 270]
    (2.75,0) to [out = 90, in = 270]
    (3.5,.75) to [out = 90, in = 0]
    (1.5,1.5);
    \strand (1.5,.6) to [out = 180, in = 90]
    (.75,0) to [out =270, in = 180]
    (1.5,-.6) to [out = 0, in = 270]
    (2.25,0) to [out = 90, in = 0]
    (1.5,.6);
        \strand (2.5,.6) to [out = 180, in = 90]
    (1.75,0) to [out =270, in = 180]
    (2.5,-.6) to [out = 0, in = 270]
    (3.25,0) to [out = 90, in = 0]
    (2.5,.6);
    
    \end{knot}
    
    \draw (-.8,0) node {$u_1$};
    \draw (.3,0) node {$u_2$};
    \draw (.3,.8) node {$v_1$};
    \draw (.3,-.8) node {$v_2$};

\end{scope}

\begin{scope}[yshift = -4cm]
\fill (0,0) circle (.1cm);
\fill (1,.5) circle (.1cm);
\fill (2,.5) circle (.1cm);
\fill (1,-.5) circle (.1cm);
\fill (2,-.5) circle (.1cm);
\fill (3,0) circle (.1cm);

\draw (0,0) -- (1,.5) -- (2,.5) -- (3,0) -- (2,-.5) -- (1,-.5) -- (0,0);

\draw (0,0) to [out = 180, in = 270]
(-.5,.5) to [out = 90, in = 180]
(1.5,1.5) to [out = 0, in = 90]
(3.5,.5) to [out = 270, in = 0]
(3,0);

\draw (-.3,-.3) node{$u_2$};
\draw (3.3,-.3) node{$u_1$};

\begin{scope}[xshift = 5cm]
\fill (0,0) circle (.1cm);
\fill (1,.5) circle (.1cm);
\fill (2,.5) circle (.1cm);
\fill (1,-.5) circle (.1cm);
\fill (2,-.5) circle (.1cm);
\fill (3,0) circle (.1cm);

\draw (0,0) -- (1,.5) -- (2,.5) -- (3,0) -- (2,-.5) -- (1,-.5) -- (0,0);

\draw (0,0) to [out = 180, in = 270]
(-.5,.5) to [out = 90, in = 180]
(1.5,1.5) to [out = 0, in = 90]
(3.5,.5) to [out = 270, in = 0]
(3,0);

\draw (-.3,-.3) node{$u_2$};
\draw (3.3,-.3) node{$u_1$};

\draw[blue] (1.5,1) -- (1.5,-1);
\draw[blue] (1.5,1) to [out = 180, in = 90] (0.5,0) to [out = 270, in = 180] (1.5,-1) to [out = 0, in = 270] (2.5,0) to [out = 90, in = 0] (1.5,1);

\draw[blue] (1.5,1) to [out = 90, in = 90] (-.7,0) to [out = 270, in = 270]  (1.5,-1);

\fill[white] (1.5,1) circle (.1cm);
\draw (1.5,1) circle (.1cm);

\fill[white] (1.5,-1) circle (.1cm);
\draw (1.5,-1) circle (.1cm);

\fill[white] (.5,0) circle (.1cm);
\draw (.5,0) circle (.1cm);

\fill[white] (1.5,0) circle (.1cm);
\draw (1.5,0) circle (.1cm);

\fill[white] (2.5,0) circle (.1cm);
\draw (2.5,0) circle (.1cm);

\draw (1.8,1.2) node{$v_1$};
\draw (1.8, -1.2) node{$v_2$};

\end{scope}

\begin{scope}[xshift = 10.5cm]

\begin{knot}[
    consider self intersections,
    clip width = 4,
    ignore endpoint intersections = true,
    end tolerance = 2pt,
    ]
    \flipcrossings{3, 7, 5}
    \strand (1.5,1.5) to [out = 180, in = 90]
    (-.5,.5) to [out = 270, in = 180]
    (.5,-.7) to [out = 0, in = 180]
    (1.5,-.25) to [out = 0, in = 180]
    (2.5,-.75) to [out = 0, in = 270]
    (3,0) to [out = 90, in = 0]
    (2.5,.75) to [out = 180, in = 0]
    (1.5,.25) to [out = 180, in = 0]
    (.5,.75) to [out = 180, in = 90]
    (-.5,-.5) to [out = 270, in = 180]
    (1.5,-1.5) to [out = 0, in =270]
    (3,-.5) to [out = 90, in = 0]
    (2.5,-.25) to [out = 180, in =0]
    (1.5,-.7) to [out = 180, in = 270]
    (1,0) to [out = 90, in = 180]
    (1.5,.75) to [out = 0, in = 180]
    (2.5,.25) to [out = 0, in = 270]
    (3,.5) to [out = 90, in = 0]
    (1.5,1.5);
\end{knot}
    
    \draw (-.8,0) node {$u_1$};
    \draw (.3,0) node {$u_2$};
    \draw (.3,.9) node {$v_1$};
    \draw (.3,-.9) node {$v_2$};

\end{scope}
\end{scope}

\begin{scope}[yshift= - 8cm]
\fill (0,0) circle (.1cm);
\fill (1,0) circle (.1cm);
\fill (2,.5) circle (.1cm);
\fill (2,-.5) circle (.1cm);
\fill (3,0) circle (.1cm);

\draw (0,0) -- (1,0) -- (2,.5) -- (3,0) -- (2,-.5) -- (1,0);

\draw (0,0) to [out = 180, in = 270]
(-.5,.5) to [out = 90, in = 180]
(1.5,1.5) to [out = 0, in = 90]
(3.5,.5) to [out = 270, in = 0]
(3,0);

\draw (-.3,-.3) node{$u_2$};
\draw (3.3,-.3) node{$u_1$};

\begin{scope}[xshift = 5 cm]

\fill (0,0) circle (.1cm);
\fill (1,0) circle (.1cm);
\fill (2,.5) circle (.1cm);
\fill (2,-.5) circle (.1cm);
\fill (3,0) circle (.1cm);

\draw (1,0) -- (2,.5) -- (3,0) -- (2,-.5) -- (1,0);
\draw (0,0) to [out = 60, in = 120] (1,0);
\draw (0,0) to [out = -60, in = -120] (1,0);

\draw (0,0) to [out = 180, in = 270]
(-.5,.5) to [out = 90, in = 180]
(1.5,1.5) to [out = 0, in = 90]
(3.5,.5) to [out = 270, in = 0]
(3,0);

\draw[blue] (1.5,1) to [out = 270, in = 135] (1.5,.25) to [out = -45, in = 135] (2,0);
\draw[blue] (1.5,1) to [out = 0, in =45] (2.5,.25) to [out = 225,in=45] (2,0);
\draw[blue] (1.5,-1) to [out = 90, in = 225] (1.5,-.25) to [out = 45, in =225] (2,0);
\draw[blue] (1.5,-1) to [out = 0, in =-45] (2.5,-.25) to [out = 135,in = -45] (2,0);
\draw[blue] (1.5,1) to [out = 180, in = 90] (0.5,0) to [out = 270, in = 180] (1.5,-1);

\draw[blue] (1.5,1) to [out = 90, in = 90] (-.7,0) to [out = 270, in = 270]  (1.5,-1);

\fill[white] (1.5,1) circle (.1cm);
\draw (1.5,1) circle (.1cm);

\fill[white] (1.5,-1) circle (.1cm);
\draw (1.5,-1) circle (.1cm);

\draw (-.3,-.3) node{$u_2$};
\draw (3.3,-.3) node{$u_1$};
\draw (1.8,1.2) node{$v_1$};
\draw (1.8,-1.2) node{$v_2$};

\fill[white] (0.5,0) circle (.1cm);
\draw (0.5,0) circle (.1cm);
\fill[white] (2,0) circle (.1cm);
\draw (2,0) circle (.1cm);

\end{scope}

\begin{scope}[xshift = 10.5 cm]

\begin{knot}[
    consider self intersections,
    clip width = 4,
    ignore endpoint intersections = true,
    end tolerance = 2pt
    ]
    \flipcrossings{3, 6, 4}
    \strand (1.5,1.5) to [out = 180, in = 90]
    (-.5,.5) to [out = 270, in = 180]
    (.5,-.5) to [out = 0, in = 270]
    (1.25,0) to [out = 90, in = 0]
    (.5,.5) to [out = 180, in = 90]
    (-.5,-.5) to [out = 270, in = 180]
    (1.5,-1.5) to [out = 0, in = 270]
    (3,-.75) to [out = 90, in = 0]
    (2.5,-.25) to [out = 180, in = 0]
    (1.5,-.75) to [out = 180, in = 270]
    (.75,0) to [out = 90, in = 180]
    (1.5,.75) to [out = 0, in = 180]
    (2.5,.25) to [out = 0, in = 270]
    (3,.75) to [out = 90, in = 0]
    (1.5,1.5);
    \strand (2.5,.6) to [out = 180, in = 90]
    (1.75,0) to [out =270, in = 180]
    (2.5,-.6) to [out = 0, in = 270]
    (3.25,0) to [out = 90, in = 0]
    (2.5,.6);
    
    \end{knot}
    
    \draw (-.8,0) node {$u_1$};
    \draw (.3,0) node {$u_2$};
    \draw (.3,.8) node {$v_1$};
    \draw (.3,-.8) node {$v_2$};

\end{scope}

\end{scope}

\begin{scope}[yshift= - 12cm]
\fill (0,0) circle (.1cm);
\fill (1,.5) circle (.1cm);
\fill (1,-.5) circle (.1cm);
\fill (2,0) circle (.1cm);
\fill (3,0) circle (.1cm);

\draw (3,0) -- (2,0) -- (1,.5) -- (0,0) -- (1,-.5) -- (2,0);

\draw (0,0) to [out = 180, in = 270]
(-.5,.5) to [out = 90, in = 180]
(1.5,1.5) to [out = 0, in = 90]
(3.5,.5) to [out = 270, in = 0]
(3,0);

\draw (-.3,-.3) node{$u_2$};
\draw (3.3,-.3) node{$u_1$};

\begin{scope}[xshift = 5 cm]

\fill (0,0) circle (.1cm);
\fill (1,.5) circle (.1cm);
\fill (1,-.5) circle (.1cm);
\fill (2,0) circle (.1cm);
\fill (3,0) circle (.1cm);

\draw (3,0) -- (2,0) -- (1,.5) -- (0,0) -- (1,-.5) -- (2,0);
\draw (2,0) to [out = 60, in = 120] (3,0);
\draw (2,0) to [out = -60, in = -120] (3,0);

\draw (0,0) to [out = 180, in = 270]
(-.5,.5) to [out = 90, in = 180]
(1.5,1.5) to [out = 0, in = 90]
(3.5,.5) to [out = 270, in = 0]
(3,0);

\draw[blue] (1.5,1) to [out = 270, in = 45] (1.5,.25) to [out = -135, in = 45] (1,0);
\draw[blue] (1.5,1) to [out = 180, in =135] (0.5,.25) to [out = -45,in=135] (1,0);
\draw[blue] (1.5,-1) to [out = 90, in = -45] (1.5,-.25) to [out = 135, in =-45] (1,0);
\draw[blue] (1.5,-1) to [out = 180, in =-135] (0.5,-.25) to [out = 45,in = -135] (1,0);
\draw[blue] (1.5,1) to [out = 0, in = 90] (2.5,0) to [out = 270, in = 0] (1.5,-1);

\draw[blue] (1.5,1) to [out = 90, in = 90] (-.7,0) to [out = 270, in = 270]  (1.5,-1);

\fill[white] (1.5,1) circle (.1cm);
\draw (1.5,1) circle (.1cm);

\fill[white] (1.5,-1) circle (.1cm);
\draw (1.5,-1) circle (.1cm);

\fill[white] (1,0) circle (.1cm);
\draw (1,0) circle (.1cm);
\fill[white] (2.5,0) circle (.1cm);
\draw (2.5,0) circle (.1cm);

\draw (-.3,-.3) node{$u_2$};
\draw (3.3,-.3) node{$u_1$};
\draw (1.8,1.2) node{$v_1$};
\draw (1.8,-1.2) node{$v_2$};

\end{scope}

\begin{scope}[xshift = 10.5 cm]

    \begin{knot}[
    consider self intersections,
    clip width = 4,
    ignore endpoint intersections = true,
    end tolerance = 2pt
    ]
    \flipcrossings{3, 5, 7}
    \strand (1.5,1.5) to [out = 180, in = 90]
    (-.5,.5) to [out = 270, in = 180]
    (.5,-.5) to [out = 0, in = 180]
    (1.5,-.25) to [out = 0, in = 180]
    (2.5,-.6) to [out = 0, in = 270]
    (3.5,0) to [out = 90, in = 0]
    (2.5,.6) to [out = 180, in = 0]
    (1.5,.25) to [out = 180, in = 0]
    (.5,.5) to [out = 180, in = 90]
    (-.5,-.5) to [out = 270, in = 180]
    (1.5,-1.5) to [out = 0, in = 270]
    (3.25, -.5) to [out = 90, in = 270]
    (3,0) to [out = 90, in =270]
    (3.25,.5) to [out = 90, in = 0]
    (1.5,1.5);
    \strand (1.5,.6) to [out = 180, in = 90]
    (.75,0) to [out =270, in = 180]
    (1.5,-.6) to [out = 0, in = 270]
    (2.25,0) to [out = 90, in = 0]
    (1.5,.6);
    
    \end{knot}
    
    \draw (-.8,0) node {$u_1$};
    \draw (.3,0) node {$u_2$};
    \draw (.3,.8) node {$v_1$};
    \draw (.3,-.8) node {$v_2$};

\end{scope}

\end{scope}

\end{tikzpicture}\]
\caption{The possibilities for $G'$, $G$ and $\overline{G}$, and $D$ under the assumptions of Lemma \ref{lemma:base}.}
\label{figure:BaseCase}
\end{figure}
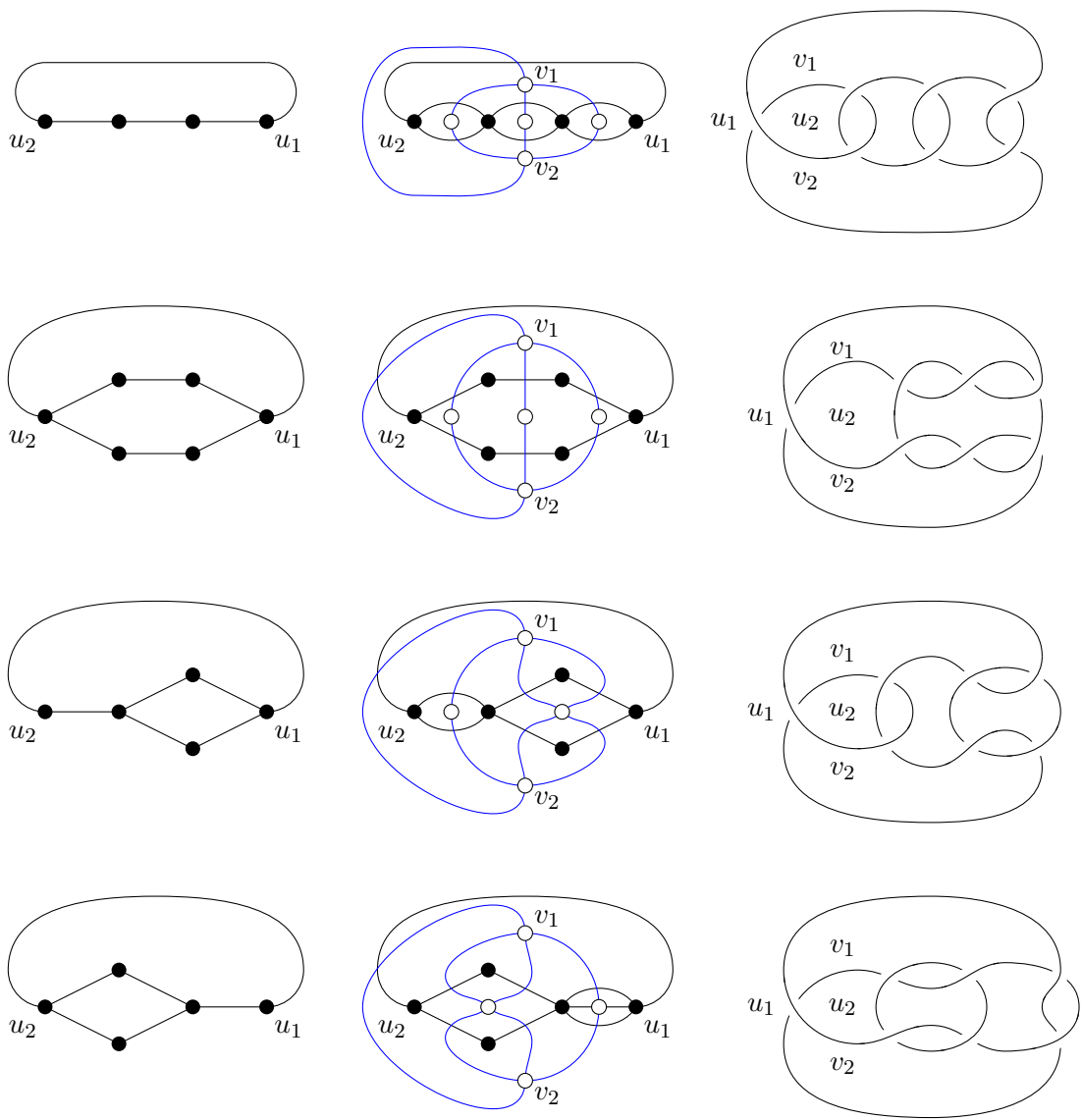

Suppose that $G'$ is the first graph in the first row of  Figure \ref{figure:BaseCase}. In order for every edge of $\overline{G}$ associated with a crossing in $R$ to be on a path of length two between $v_1$ and $v_2$, every edge on the path of length three from $u_1$ to $u_2$ in $G'$ has multiplicity two in $G$. The resulting graphs $G$ and $\overline{G}$ are in the second column of Figure \ref{figure:BaseCase} and the associated link diagram $D_1$ is in the third column of Figure \ref{figure:BaseCase}. The diagram $D_1$ is a diagram of the mirror of the link $L_{6n1}$ in Thisltethwaite's table. Its unshifted Khovanov homology is given in Table \ref{table:L6n1} and $\ukh^{3,4-s_A(D_1)}(D_1)$ is trivial.

\begin{table}
\begin{tabular}{| r || c | c | c | c | c |}
\hline
$j \backslash i$& 1 & 2 & 3 & 4 & 5\\
\hline
\hline
9 & & & & & 2\\
\hline
7 & & & & 1 & 3\\
\hline
5 & & &  & $1_2$ & 1\\
\hline
3 & & & 1 & & \\
\hline
1 & 1 & &\cellcolor{red} & &\\
\hline
-1 & 1 & & & &\\
\hline
\end{tabular}
\caption{The shifted Khovanov homology $\ukh(D_1)=\ukh(L_{6n1})$. An entry of $k$ indicates a $\mathbb{Z}^k$ summand, and an entry of $k_2$ indicates a $\mathbb{Z}_2^k$ summand.}
\label{table:L6n1}
\end{table}

Suppose that $G'$ is the first graph in the second row of Figure \ref{figure:BaseCase}. In order for every edge of $\overline{G}$ associated with a crossing in $R$ to be on a path of length two between $v_1$ and $v_2$, all of the edges in $G'$ have multiplicity one in $G$. The resulting graphs $G$ and $\overline{G}$ are in the second column of Figure \ref{figure:BaseCase}, and the associated link diagram $D_2$ is in the third column of Figure \ref{figure:BaseCase}. The diagram $D_2$ is a diagram of the right-handed trefoil, an alternating knot. Since $D_2$ is an $A$-almost alternating diagram of an alternating knot, its unshifted Khovanov homology is supported in bigradings $(i,j)$ where $2i-j = s_A(D) -2$ or $s_A(D)$. Therefore $\ukh^{3,4-s_A(D_2)}(D_2)$ is trivial.

Suppose that $G'$ is the first graph in the third row of Figure \ref{figure:BaseCase}. In order for every edge of $\overline{G}$ associated with a crossing in $R$ to be on a path of length two between $v_1$ and $v_2$, only the edge incident to $u_2$ on a path of length three between $u_1$ and $u_2$ has multiplicity two in $G$. All other edges have multiplicity one in $G$. The resulting graphs $G$ and $\overline{G}$ are in the second column of Figure \ref{figure:BaseCase} and the associated link diagram $D_3$ is in the third column of Figure \ref{figure:BaseCase}. The diagram $D_3$ is an $A$-almost alternating diagram of the $(2,4)$ torus link, an alternating link. As noted previously, this implies that $\ukh^{3,4-s_A(D_2)}(D_2)$ is trivial. The case where $G'$ is the first graph in fourth row is similar to this case and also results in an $A$-almost alternating diagram of the $(2,4)$-torus link.
\end{proof}

Lemma \ref{lemma:path2} shows that one of the three terms in the long exact sequence needed to compute $\ukh^{3,4-s_A(D)}(D)$ is trivial. See Figure \ref{figure:proofpath2} for an example of the enhanced states described in the proof of Lemma \ref{lemma:path2}.

\begin{lemma}
\label{lemma:path2}
Let $D$ be an almost alternating diagram where the only crossing in the boundary of $u_1$ and $u_2$ is the dealternator and where there is exactly one face in the alternating tangle $R$ that shares crossing(s) in its boundary with both $u_1$ and $u_2$. Then $\ukh^{2,2-s_A(D)}(D)$ is trivial.
\end{lemma}
\begin{proof}
Let $G$ be the checkerboard graph of $D$ containing $u_1$ and $u_2$. First, suppose that $G$ is a simple graph, i.e. it contains no multiple edges. Let $S_A^-$ be the enhanced state where every crossing is an $A$-resolution and each component is labeled with a $-$. Then $i(S_A^-) =0$ and $j(s_A^-)=-s_A(D)$. Suppose that $S$ is an enhanced state such that $i(S)=2$ and $j(S) = 2-s_A(D)$. Since $i(S)=2$, it follows that $S$ has exactly two $B$-resolutions. The states of $D$ with exactly two $B$-resolutions and the most number of components are those states where the dealternator and one other crossing are $B$-resolutions. Such states have $s_A(D)$ components. Because $j(S)=2-s_A(D)$, it follows that $S$ is an enhanced state where the dealternator and exactly one other crossing are $B$-resolutions and every component of $S$ is labeled with a $-$.

Our strategy is to show that the Khovanov boundary map $d^{1,2-s_A(D)}:CKh^{1,2-s_A(D)}(D)\to CKh^{2,2-s_A(D)}(D)$ is onto, and therefore the homology $\ukh^{2,2-s_A(D)}(D)$ is trivial.  Number the crossings of $D$ (and thus the edges of $G$) so that the dealternator is assigned $c=c(D)$, the two edges on the path of length two between $u_1$ and $u_2$ are assigned 1 and 2, and the remaining crossings are assigned 3 through $c-1$. For $1\leq k \leq c-1$, define $S_k$ to be the enhanced state with $B$-resolutions at crossings $k$ and $c$, $A$-resolutions at all other crossings, and where every component is labeled with a $-$. The enhanced states $\{S_k~|~k=1,\dots,c-1\}$ form a basis of $CKh^{2,2-s_A(D)}(D)$.

For $k=1$ or $2$, define $T_k$ to be the enhanced state where the dealternator is assigned a $B$-resolution, all other crossing are assigned $A$-resolutions, the component incident to the dealternator and crossing $k$ is labeled with a $+$, and all other components are labeled with a $-$. For $3\leq k \leq c-1$, define $T_k$ to be the enhanced state where crossing $k$ is assigned a $B$-resolution, all other crossings are assigned $A$-resolutions, and every component is labeled with a $-$. For all $1\leq k \leq c-1$, the enhanced state $T_k$ satisfies $i(T_k)=1$ and $j(T_k)=2-s_A(D)$. 

When $3\leq k \leq c-1$, the only crossing where switching from an $A$-resolution to a $B$-resolution splits one component into two is the dealternator. All other switches from an $A$-resolution to a $B$-resolution merge two components into one. Since every component of $T_k$ for $3\leq k \leq c-1$ is labeled with a $-$, it follows that $d(T_k) = - S_k$ for all $3\leq k \leq c-1$. When $k=1$ or $2$, each crossing change from an $A$-resolution to a $B$-resolution merges two components into one. Since there is only one component labeled with a $+$, it follows that an enhanced state $S_\ell$ will appear with nonzero coefficient in the sum $d(T_k)$ if and only if the component labeled with a $+$ in $T_k$ is incident to crossing $\ell$. Therefore
\[ d(T_1) = S_1 + \sum_{k=3}^{c-1} a_k S_k~\text{and}~d(T_2) = S_2 + \sum_{k=3}^{c-1}b_kS_k\]
where each $a_k$ and $b_k$ is either $0$, $1$, or $-1$, and also, $S_k=d(-T_k)$ for $3\leq k \leq c-1$. Therefore $S_1= d(T_1) + \sum_{k=3}^{c-1}a_k d(T_k)$ and $S_2=d(T_2)+\sum_{k=3}^{c-1}b_k d(T_k)$. Since each $S_k$ is in the image of the Khovanov differential $d$ for $1\leq k \leq c-1$, it follows that the map $d^{1,2-s_A(D)}:CKh^{1,2-s_A(D)}(D)\to CKh^{2,2-s_A(D)}(D)$ is onto, and hence $\ukh^{2,2-s_A(D)}(D)$ is trivial.

Now suppose that $G$ potentially has multiple edges. Let $G'$ be the simplification of $G$, that is $G'$ has the same vertex set as $G$ but multiple edges in $G$ are replaced with a single edge in $G'$. Let $m$ be the difference in the number of edges between $G$ and $G'$, and proceed by induction on $m$. If $m=0$, then $G$ is simple and that case is proved above.

By way of induction, assume that if the difference in the number of edges between $G$ and $G'$ is less than $m$, then $\ukh^{2,2-s_A(D)}(D)$ is trivial. Choose an edge of $G$ that is a part of a set of $n$ multiple edges for some $n>1$. Call this link diagram $D_n$. Let $D$ be the link diagram obtained from $G$ by replacing the $n$ multiple edges with a single edge. Then $s_A(D)=s_A(D_n)$, and by the inductive hypothesis, $\ukh^{2,2-s_A(D)}(D)$ is trivial. Because $s_A(D_B) = s_A(D)-1$, it follows that $2-s_A(D) - 2k+1 = 4-s_A(D_B)-2k$. Thus if $k>1$, then $2-s_A(D)-2k+1\leq -s_A(D_B)$ and hence $\ukh^{2-k,2-s_A(D)-2k+1}(D_B)$ is trivial. Lemma \ref{lemma:twist} implies that $\ukh^{2,2-s_A(D_n)}(D_n)$ is trivial, as desired.
\end{proof}

\begin{example}
\label{example:14}
Let $D$ be the diagram of the unknot in Figure \ref{figure:proofpath2}. Then $D$ satisfies the assumptions of Lemma \ref{lemma:path2}, and the enhanced states $S_i$ and $T_j$ for $1\leq i,j\leq 5$ are also depicted in Figure \ref{figure:proofpath2}. The differential applied to $T_j$ yields $d(T_1) = S_1+S_3$, $d(T_2)=S_2+S_5$, and $d(T_j)=-S_j$ for $j=1,2,$ and $3$. 
\end{example}

\begin{figure}[ht]
\[\begin{tikzpicture}
    \begin{knot}[
    consider self intersections,
    clip width = 4,
    ignore endpoint intersections = true,
    end tolerance = .25pt
    ]
    \flipcrossings{3, 4, 6}
    \strand (1.5,1.5) to [out = 180, in = 90]
    (-.5,.5) to [out = 270, in = 180]
    (.75,-.75) to [out = 0, in = 180]
    (2,-.25) to [out = 0, in = 90]
    (3.25,-1) to [out = 270, in = 0]
    (1.5,-1.5) to [out = 180, in = 270]
    (-.5,-.5) to [out = 90, in = 180]
    (.6,.75) to [out = 0, in = 180]
    (1.5,.25) to [out = 0, in = 180]
    (2.5,.75) to [out = 0, in = 90]
    (3,0) to [out = 270, in = 0]
    (2.25,-.75) to [out= 180, in = 270]
    (0.75,0) to [out = 90, in = 180]
    (1.5,.75) to [out = 0, in = 180]
    (2.5,.25) to [out = 0, in = 270]
    (3.25,.75) to [out = 90, in = 0]
    (1.5,1.5);
    \end{knot}
    
\draw (-.7,0) node{6};
\draw (1,1) node{3};
\draw (1.2,-1) node{1};
\draw (2.9,-1) node{2};
\draw (2,1) node {4};
\draw (2.9,1) node{5};

\begin{scope}[scale = .4, xshift = 9cm, yshift = -2.7cm]

\draw (2,3) to [out = 45, in = 180] (3,4.5) to [out = 0, in = 135] (4,4);
\draw (2,2) to [out = -45, in = 180] (3,.5) to [out = 0, in = -135] (4,1);
\draw (4,2) to [out = 135, in = 225] (4,3);
\draw (5,4) to [out = 45, in = 135] (6,4);
\draw (5,3) to [out = -45, in = 225] (6,3);
\draw (7,4) to [out = 45, in = 135] (8,4);
\draw (7,3) to [out = -45, in = 225] (8,3);
\draw (9,3) to [out = -45, in = 45] (9,2);
\draw (8,2) to [out = 135, in = 45, looseness=.5] (5,2);
\draw (8,1) to [out = 225, in = -45, looseness=.5](5,1);
\draw (1,2) to [out = 225, in = 90] (.5,1);
\draw (.5,1) arc (180:270:1);
\draw (1.5,0) -- (9,0);
\draw (9,0) arc (-90:0:.5);
\draw (9.5,.5) to [out = 90, in = -45] (9,1);
\draw (1,3) to [out = 135, in = 270] (.5,4);
\draw (.5,4) arc (180:90:1);
\draw (1.5,5) -- (9,5);
\draw (9,5) arc (90:0:.5);
\draw (9.5,4.5) to [out = 270, in = 45] (9,4);

\begin{scope}[xshift =1cm, yshift = 2cm]
\draw[very thick, blue] (.3,.5) -- (.7,.5);
\draw (0,0) to [out = 45, in = 270] (.3,.5) to [out = 90, in = -45] (0,1);
\draw (1,0) to [out = 135, in = 270] (.7,.5) to [out = 90, in = 225] (1,1);
\end{scope}

\begin{scope}[xshift =4cm, yshift = 1cm]
\draw[very thick, blue] (.5,.3) -- (.5,.7);
\draw (0,0) to [out = 45, in = 180] (.5,.3) to [out = 0, in = 135] (1,0);
\draw (0,1) to [out = -45, in = 180] (.5,.7) to [out = 0, in = 225] (1,1);
\end{scope}

\begin{scope}[xshift =6cm, yshift = 3cm]
\draw[very thick, red] (.3,.5) -- (.7,.5);
\draw (0,0) to [out = 45, in = 270] (.3,.5) to [out = 90, in = -45] (0,1);
\draw (1,0) to [out = 135, in = 270] (.7,.5) to [out = 90, in = 225] (1,1);
\end{scope}

\begin{scope}[xshift =8cm, yshift = 3cm]
\draw[very thick, red] (.3,.5) -- (.7,.5);
\draw (0,0) to [out = 45, in = 270] (.3,.5) to [out = 90, in = -45] (0,1);
\draw (1,0) to [out = 135, in = 270] (.7,.5) to [out = 90, in = 225] (1,1);
\end{scope}

\begin{scope}[xshift =8cm, yshift = 1cm]
\draw[very thick, red] (.3,.5) -- (.7,.5);
\draw (0,0) to [out = 45, in = 270] (.3,.5) to [out = 90, in = -45] (0,1);
\draw (1,0) to [out = 135, in = 270] (.7,.5) to [out = 90, in = 225] (1,1);
\end{scope}

\begin{scope}[xshift =4cm, yshift = 3cm]
\draw[very thick, red] (.3,.5) -- (.7,.5);
\draw (0,0) to [out = 45, in = 270] (.3,.5) to [out = 90, in = -45] (0,1);
\draw (1,0) to [out = 135, in = 270] (.7,.5) to [out = 90, in = 225] (1,1);
\end{scope}

\draw (1,4) node{$-$};
\draw (3,4) node{$-$};
\draw (5.5,3.5) node{$-$};
\draw (7.5,3.5) node{$-$};

\draw (5,-1) node {$S_1$};

\end{scope}

\begin{scope}[scale = .4, xshift = 19cm, yshift = -2.7cm]

\draw (2,3) to [out = 45, in = 180] (3,4.5) to [out = 0, in = 135] (4,4);
\draw (2,2) to [out = -45, in = 180] (3,.5) to [out = 0, in = -135] (4,1);
\draw (4,2) to [out = 135, in = 225] (4,3);
\draw (5,4) to [out = 45, in = 135] (6,4);
\draw (5,3) to [out = -45, in = 225] (6,3);
\draw (7,4) to [out = 45, in = 135] (8,4);
\draw (7,3) to [out = -45, in = 225] (8,3);
\draw (9,3) to [out = -45, in = 45] (9,2);
\draw (8,2) to [out = 135, in = 45, looseness=.5] (5,2);
\draw (8,1) to [out = 225, in = -45, looseness=.5](5,1);
\draw (1,2) to [out = 225, in = 90] (.5,1);
\draw (.5,1) arc (180:270:1);
\draw (1.5,0) -- (9,0);
\draw (9,0) arc (-90:0:.5);
\draw (9.5,.5) to [out = 90, in = -45] (9,1);
\draw (1,3) to [out = 135, in = 270] (.5,4);
\draw (.5,4) arc (180:90:1);
\draw (1.5,5) -- (9,5);
\draw (9,5) arc (90:0:.5);
\draw (9.5,4.5) to [out = 270, in = 45] (9,4);

\begin{scope}[xshift =1cm, yshift = 2cm]
\draw[very thick, blue] (.3,.5) -- (.7,.5);
\draw (0,0) to [out = 45, in = 270] (.3,.5) to [out = 90, in = -45] (0,1);
\draw (1,0) to [out = 135, in = 270] (.7,.5) to [out = 90, in = 225] (1,1);
\end{scope}

\begin{scope}[xshift =8cm, yshift = 1cm]
\draw[very thick, blue] (.5,.3) -- (.5,.7);
\draw (0,0) to [out = 45, in = 180] (.5,.3) to [out = 0, in = 135] (1,0);
\draw (0,1) to [out = -45, in = 180] (.5,.7) to [out = 0, in = 225] (1,1);
\end{scope}

\begin{scope}[xshift =6cm, yshift = 3cm]
\draw[very thick, red] (.3,.5) -- (.7,.5);
\draw (0,0) to [out = 45, in = 270] (.3,.5) to [out = 90, in = -45] (0,1);
\draw (1,0) to [out = 135, in = 270] (.7,.5) to [out = 90, in = 225] (1,1);
\end{scope}

\begin{scope}[xshift =8cm, yshift = 3cm]
\draw[very thick, red] (.3,.5) -- (.7,.5);
\draw (0,0) to [out = 45, in = 270] (.3,.5) to [out = 90, in = -45] (0,1);
\draw (1,0) to [out = 135, in = 270] (.7,.5) to [out = 90, in = 225] (1,1);
\end{scope}

\begin{scope}[xshift =4cm, yshift = 3cm]
\draw[very thick, red] (.3,.5) -- (.7,.5);
\draw (0,0) to [out = 45, in = 270] (.3,.5) to [out = 90, in = -45] (0,1);
\draw (1,0) to [out = 135, in = 270] (.7,.5) to [out = 90, in = 225] (1,1);
\end{scope}

\begin{scope}[xshift =4cm, yshift = 1cm]
\draw[very thick, red] (.3,.5) -- (.7,.5);
\draw (0,0) to [out = 45, in = 270] (.3,.5) to [out = 90, in = -45] (0,1);
\draw (1,0) to [out = 135, in = 270] (.7,.5) to [out = 90, in = 225] (1,1);
\end{scope}

\draw (1,4) node{$-$};
\draw (3,4) node{$-$};
\draw (5.5,3.5) node{$-$};
\draw (7.5,3.5) node{$-$};

\draw (5,-1) node {$S_2$};

\end{scope}

\begin{scope}[scale = .4, xshift = 29cm, yshift = -2.7cm]

\draw (2,3) to [out = 45, in = 180] (3,4.5) to [out = 0, in = 135] (4,4);
\draw (2,2) to [out = -45, in = 180] (3,.5) to [out = 0, in = -135] (4,1);
\draw (4,2) to [out = 135, in = 225] (4,3);
\draw (5,4) to [out = 45, in = 135] (6,4);
\draw (5,3) to [out = -45, in = 225] (6,3);
\draw (7,4) to [out = 45, in = 135] (8,4);
\draw (7,3) to [out = -45, in = 225] (8,3);
\draw (9,3) to [out = -45, in = 45] (9,2);
\draw (8,2) to [out = 135, in = 45, looseness=.5] (5,2);
\draw (8,1) to [out = 225, in = -45, looseness=.5](5,1);
\draw (1,2) to [out = 225, in = 90] (.5,1);
\draw (.5,1) arc (180:270:1);
\draw (1.5,0) -- (9,0);
\draw (9,0) arc (-90:0:.5);
\draw (9.5,.5) to [out = 90, in = -45] (9,1);
\draw (1,3) to [out = 135, in = 270] (.5,4);
\draw (.5,4) arc (180:90:1);
\draw (1.5,5) -- (9,5);
\draw (9,5) arc (90:0:.5);
\draw (9.5,4.5) to [out = 270, in = 45] (9,4);

\begin{scope}[xshift =1cm, yshift = 2cm]
\draw[very thick, blue] (.3,.5) -- (.7,.5);
\draw (0,0) to [out = 45, in = 270] (.3,.5) to [out = 90, in = -45] (0,1);
\draw (1,0) to [out = 135, in = 270] (.7,.5) to [out = 90, in = 225] (1,1);
\end{scope}

\begin{scope}[xshift =4cm, yshift = 3cm]
\draw[very thick, blue] (.5,.3) -- (.5,.7);
\draw (0,0) to [out = 45, in = 180] (.5,.3) to [out = 0, in = 135] (1,0);
\draw (0,1) to [out = -45, in = 180] (.5,.7) to [out = 0, in = 225] (1,1);
\end{scope}

\begin{scope}[xshift =6cm, yshift = 3cm]
\draw[very thick, red] (.3,.5) -- (.7,.5);
\draw (0,0) to [out = 45, in = 270] (.3,.5) to [out = 90, in = -45] (0,1);
\draw (1,0) to [out = 135, in = 270] (.7,.5) to [out = 90, in = 225] (1,1);
\end{scope}

\begin{scope}[xshift =8cm, yshift = 3cm]
\draw[very thick, red] (.3,.5) -- (.7,.5);
\draw (0,0) to [out = 45, in = 270] (.3,.5) to [out = 90, in = -45] (0,1);
\draw (1,0) to [out = 135, in = 270] (.7,.5) to [out = 90, in = 225] (1,1);
\end{scope}

\begin{scope}[xshift =8cm, yshift = 1cm]
\draw[very thick, red] (.3,.5) -- (.7,.5);
\draw (0,0) to [out = 45, in = 270] (.3,.5) to [out = 90, in = -45] (0,1);
\draw (1,0) to [out = 135, in = 270] (.7,.5) to [out = 90, in = 225] (1,1);
\end{scope}

\begin{scope}[xshift =4cm, yshift = 1cm]
\draw[very thick, red] (.3,.5) -- (.7,.5);
\draw (0,0) to [out = 45, in = 270] (.3,.5) to [out = 90, in = -45] (0,1);
\draw (1,0) to [out = 135, in = 270] (.7,.5) to [out = 90, in = 225] (1,1);
\end{scope}

\draw (1,4) node{$-$};
\draw (3,4) node{$-$};
\draw (6.5,1.5) node{$-$};
\draw (7.5,3.5) node{$-$};

\draw (5,-1) node {$S_3$};

\end{scope}

\begin{scope}[yshift = - 3.5cm]

\begin{scope}[scale = .4, xshift = -1cm, yshift = -2.7cm]

\draw (2,3) to [out = 45, in = 180] (3,4.5) to [out = 0, in = 135] (4,4);
\draw (2,2) to [out = -45, in = 180] (3,.5) to [out = 0, in = -135] (4,1);
\draw (4,2) to [out = 135, in = 225] (4,3);
\draw (5,4) to [out = 45, in = 135] (6,4);
\draw (5,3) to [out = -45, in = 225] (6,3);
\draw (7,4) to [out = 45, in = 135] (8,4);
\draw (7,3) to [out = -45, in = 225] (8,3);
\draw (9,3) to [out = -45, in = 45] (9,2);
\draw (8,2) to [out = 135, in = 45, looseness=.5] (5,2);
\draw (8,1) to [out = 225, in = -45, looseness=.5](5,1);
\draw (1,2) to [out = 225, in = 90] (.5,1);
\draw (.5,1) arc (180:270:1);
\draw (1.5,0) -- (9,0);
\draw (9,0) arc (-90:0:.5);
\draw (9.5,.5) to [out = 90, in = -45] (9,1);
\draw (1,3) to [out = 135, in = 270] (.5,4);
\draw (.5,4) arc (180:90:1);
\draw (1.5,5) -- (9,5);
\draw (9,5) arc (90:0:.5);
\draw (9.5,4.5) to [out = 270, in = 45] (9,4);

\begin{scope}[xshift =1cm, yshift = 2cm]
\draw[very thick, blue] (.3,.5) -- (.7,.5);
\draw (0,0) to [out = 45, in = 270] (.3,.5) to [out = 90, in = -45] (0,1);
\draw (1,0) to [out = 135, in = 270] (.7,.5) to [out = 90, in = 225] (1,1);
\end{scope}

\begin{scope}[xshift =6cm, yshift = 3cm]
\draw[very thick, blue] (.5,.3) -- (.5,.7);
\draw (0,0) to [out = 45, in = 180] (.5,.3) to [out = 0, in = 135] (1,0);
\draw (0,1) to [out = -45, in = 180] (.5,.7) to [out = 0, in = 225] (1,1);
\end{scope}

\begin{scope}[xshift =4cm, yshift = 3cm]
\draw[very thick, red] (.3,.5) -- (.7,.5);
\draw (0,0) to [out = 45, in = 270] (.3,.5) to [out = 90, in = -45] (0,1);
\draw (1,0) to [out = 135, in = 270] (.7,.5) to [out = 90, in = 225] (1,1);
\end{scope}

\begin{scope}[xshift =8cm, yshift = 3cm]
\draw[very thick, red] (.3,.5) -- (.7,.5);
\draw (0,0) to [out = 45, in = 270] (.3,.5) to [out = 90, in = -45] (0,1);
\draw (1,0) to [out = 135, in = 270] (.7,.5) to [out = 90, in = 225] (1,1);
\end{scope}

\begin{scope}[xshift =8cm, yshift = 1cm]
\draw[very thick, red] (.3,.5) -- (.7,.5);
\draw (0,0) to [out = 45, in = 270] (.3,.5) to [out = 90, in = -45] (0,1);
\draw (1,0) to [out = 135, in = 270] (.7,.5) to [out = 90, in = 225] (1,1);
\end{scope}

\begin{scope}[xshift =4cm, yshift = 1cm]
\draw[very thick, red] (.3,.5) -- (.7,.5);
\draw (0,0) to [out = 45, in = 270] (.3,.5) to [out = 90, in = -45] (0,1);
\draw (1,0) to [out = 135, in = 270] (.7,.5) to [out = 90, in = 225] (1,1);
\end{scope}

\draw (1,4) node{$-$};
\draw (3,4) node{$-$};
\draw (6.5,1.5) node{$-$};
\draw (7.5,3.5) node{$-$};

\draw (5,-1) node {$S_4$};

\end{scope}
    
\begin{scope}[scale = .4, xshift = 9cm, yshift = -2.7cm]

\draw (2,3) to [out = 45, in = 180] (3,4.5) to [out = 0, in = 135] (4,4);
\draw (2,2) to [out = -45, in = 180] (3,.5) to [out = 0, in = -135] (4,1);
\draw (4,2) to [out = 135, in = 225] (4,3);
\draw (5,4) to [out = 45, in = 135] (6,4);
\draw (5,3) to [out = -45, in = 225] (6,3);
\draw (7,4) to [out = 45, in = 135] (8,4);
\draw (7,3) to [out = -45, in = 225] (8,3);
\draw (9,3) to [out = -45, in = 45] (9,2);
\draw (8,2) to [out = 135, in = 45, looseness=.5] (5,2);
\draw (8,1) to [out = 225, in = -45, looseness=.5](5,1);
\draw (1,2) to [out = 225, in = 90] (.5,1);
\draw (.5,1) arc (180:270:1);
\draw (1.5,0) -- (9,0);
\draw (9,0) arc (-90:0:.5);
\draw (9.5,.5) to [out = 90, in = -45] (9,1);
\draw (1,3) to [out = 135, in = 270] (.5,4);
\draw (.5,4) arc (180:90:1);
\draw (1.5,5) -- (9,5);
\draw (9,5) arc (90:0:.5);
\draw (9.5,4.5) to [out = 270, in = 45] (9,4);

\begin{scope}[xshift =1cm, yshift = 2cm]
\draw[very thick, blue] (.3,.5) -- (.7,.5);
\draw (0,0) to [out = 45, in = 270] (.3,.5) to [out = 90, in = -45] (0,1);
\draw (1,0) to [out = 135, in = 270] (.7,.5) to [out = 90, in = 225] (1,1);
\end{scope}

\begin{scope}[xshift =8cm, yshift = 3cm]
\draw[very thick, blue] (.5,.3) -- (.5,.7);
\draw (0,0) to [out = 45, in = 180] (.5,.3) to [out = 0, in = 135] (1,0);
\draw (0,1) to [out = -45, in = 180] (.5,.7) to [out = 0, in = 225] (1,1);
\end{scope}

\begin{scope}[xshift =4cm, yshift = 3cm]
\draw[very thick, red] (.3,.5) -- (.7,.5);
\draw (0,0) to [out = 45, in = 270] (.3,.5) to [out = 90, in = -45] (0,1);
\draw (1,0) to [out = 135, in = 270] (.7,.5) to [out = 90, in = 225] (1,1);
\end{scope}

\begin{scope}[xshift =6cm, yshift = 3cm]
\draw[very thick, red] (.3,.5) -- (.7,.5);
\draw (0,0) to [out = 45, in = 270] (.3,.5) to [out = 90, in = -45] (0,1);
\draw (1,0) to [out = 135, in = 270] (.7,.5) to [out = 90, in = 225] (1,1);
\end{scope}

\begin{scope}[xshift =8cm, yshift = 1cm]
\draw[very thick, red] (.3,.5) -- (.7,.5);
\draw (0,0) to [out = 45, in = 270] (.3,.5) to [out = 90, in = -45] (0,1);
\draw (1,0) to [out = 135, in = 270] (.7,.5) to [out = 90, in = 225] (1,1);
\end{scope}

\begin{scope}[xshift =4cm, yshift = 1cm]
\draw[very thick, red] (.3,.5) -- (.7,.5);
\draw (0,0) to [out = 45, in = 270] (.3,.5) to [out = 90, in = -45] (0,1);
\draw (1,0) to [out = 135, in = 270] (.7,.5) to [out = 90, in = 225] (1,1);
\end{scope}

\draw (1,4) node{$-$};
\draw (3,4) node{$-$};
\draw (6.5,1.5) node{$-$};
\draw (5.5,3.5) node{$-$};

\draw (5,-1) node {$S_5$};

\end{scope}

\begin{scope}[scale = .4, xshift = 19cm, yshift = -2.7cm]

\draw (2,3) to [out = 45, in = 180] (3,4.5) to [out = 0, in = 135] (4,4);
\draw (2,2) to [out = -45, in = 180] (3,.5) to [out = 0, in = -135] (4,1);
\draw (4,2) to [out = 135, in = 225] (4,3);
\draw (5,4) to [out = 45, in = 135] (6,4);
\draw (5,3) to [out = -45, in = 225] (6,3);
\draw (7,4) to [out = 45, in = 135] (8,4);
\draw (7,3) to [out = -45, in = 225] (8,3);
\draw (9,3) to [out = -45, in = 45] (9,2);
\draw (8,2) to [out = 135, in = 45, looseness=.5] (5,2);
\draw (8,1) to [out = 225, in = -45, looseness=.5](5,1);
\draw (1,2) to [out = 225, in = 90] (.5,1);
\draw (.5,1) arc (180:270:1);
\draw (1.5,0) -- (9,0);
\draw (9,0) arc (-90:0:.5);
\draw (9.5,.5) to [out = 90, in = -45] (9,1);
\draw (1,3) to [out = 135, in = 270] (.5,4);
\draw (.5,4) arc (180:90:1);
\draw (1.5,5) -- (9,5);
\draw (9,5) arc (90:0:.5);
\draw (9.5,4.5) to [out = 270, in = 45] (9,4);

\begin{scope}[xshift =1cm, yshift = 2cm]
\draw[very thick, blue] (.3,.5) -- (.7,.5);
\draw (0,0) to [out = 45, in = 270] (.3,.5) to [out = 90, in = -45] (0,1);
\draw (1,0) to [out = 135, in = 270] (.7,.5) to [out = 90, in = 225] (1,1);
\end{scope}

\begin{scope}[xshift =6cm, yshift = 3cm]
\draw[very thick, red] (.3,.5) -- (.7,.5);
\draw (0,0) to [out = 45, in = 270] (.3,.5) to [out = 90, in = -45] (0,1);
\draw (1,0) to [out = 135, in = 270] (.7,.5) to [out = 90, in = 225] (1,1);
\end{scope}

\begin{scope}[xshift =4cm, yshift = 3cm]
\draw[very thick, red] (.3,.5) -- (.7,.5);
\draw (0,0) to [out = 45, in = 270] (.3,.5) to [out = 90, in = -45] (0,1);
\draw (1,0) to [out = 135, in = 270] (.7,.5) to [out = 90, in = 225] (1,1);
\end{scope}

\begin{scope}[xshift =8cm, yshift = 3cm]
\draw[very thick, red] (.3,.5) -- (.7,.5);
\draw (0,0) to [out = 45, in = 270] (.3,.5) to [out = 90, in = -45] (0,1);
\draw (1,0) to [out = 135, in = 270] (.7,.5) to [out = 90, in = 225] (1,1);
\end{scope}

\begin{scope}[xshift =8cm, yshift = 1cm]
\draw[very thick, red] (.3,.5) -- (.7,.5);
\draw (0,0) to [out = 45, in = 270] (.3,.5) to [out = 90, in = -45] (0,1);
\draw (1,0) to [out = 135, in = 270] (.7,.5) to [out = 90, in = 225] (1,1);
\end{scope}

\begin{scope}[xshift =4cm, yshift = 1cm]
\draw[very thick, red] (.3,.5) -- (.7,.5);
\draw (0,0) to [out = 45, in = 270] (.3,.5) to [out = 90, in = -45] (0,1);
\draw (1,0) to [out = 135, in = 270] (.7,.5) to [out = 90, in = 225] (1,1);
\end{scope}

\draw (1,4) node{$-$};
\draw (3,4) node{$+$};
\draw (5.5,3.5) node{$-$};
\draw (6.5,1.5) node{$-$};
\draw (7.5,3.5) node{$-$};

\draw (5,-1) node {$T_1$};

\end{scope}

\begin{scope}[scale = .4, xshift = 29cm, yshift = -2.7cm]

\draw (2,3) to [out = 45, in = 180] (3,4.5) to [out = 0, in = 135] (4,4);
\draw (2,2) to [out = -45, in = 180] (3,.5) to [out = 0, in = -135] (4,1);
\draw (4,2) to [out = 135, in = 225] (4,3);
\draw (5,4) to [out = 45, in = 135] (6,4);
\draw (5,3) to [out = -45, in = 225] (6,3);
\draw (7,4) to [out = 45, in = 135] (8,4);
\draw (7,3) to [out = -45, in = 225] (8,3);
\draw (9,3) to [out = -45, in = 45] (9,2);
\draw (8,2) to [out = 135, in = 45, looseness=.5] (5,2);
\draw (8,1) to [out = 225, in = -45, looseness=.5](5,1);
\draw (1,2) to [out = 225, in = 90] (.5,1);
\draw (.5,1) arc (180:270:1);
\draw (1.5,0) -- (9,0);
\draw (9,0) arc (-90:0:.5);
\draw (9.5,.5) to [out = 90, in = -45] (9,1);
\draw (1,3) to [out = 135, in = 270] (.5,4);
\draw (.5,4) arc (180:90:1);
\draw (1.5,5) -- (9,5);
\draw (9,5) arc (90:0:.5);
\draw (9.5,4.5) to [out = 270, in = 45] (9,4);

\begin{scope}[xshift =1cm, yshift = 2cm]
\draw[very thick, blue] (.3,.5) -- (.7,.5);
\draw (0,0) to [out = 45, in = 270] (.3,.5) to [out = 90, in = -45] (0,1);
\draw (1,0) to [out = 135, in = 270] (.7,.5) to [out = 90, in = 225] (1,1);
\end{scope}

\begin{scope}[xshift =6cm, yshift = 3cm]
\draw[very thick, red] (.3,.5) -- (.7,.5);
\draw (0,0) to [out = 45, in = 270] (.3,.5) to [out = 90, in = -45] (0,1);
\draw (1,0) to [out = 135, in = 270] (.7,.5) to [out = 90, in = 225] (1,1);
\end{scope}

\begin{scope}[xshift =4cm, yshift = 3cm]
\draw[very thick, red] (.3,.5) -- (.7,.5);
\draw (0,0) to [out = 45, in = 270] (.3,.5) to [out = 90, in = -45] (0,1);
\draw (1,0) to [out = 135, in = 270] (.7,.5) to [out = 90, in = 225] (1,1);
\end{scope}

\begin{scope}[xshift =8cm, yshift = 3cm]
\draw[very thick, red] (.3,.5) -- (.7,.5);
\draw (0,0) to [out = 45, in = 270] (.3,.5) to [out = 90, in = -45] (0,1);
\draw (1,0) to [out = 135, in = 270] (.7,.5) to [out = 90, in = 225] (1,1);
\end{scope}

\begin{scope}[xshift =8cm, yshift = 1cm]
\draw[very thick, red] (.3,.5) -- (.7,.5);
\draw (0,0) to [out = 45, in = 270] (.3,.5) to [out = 90, in = -45] (0,1);
\draw (1,0) to [out = 135, in = 270] (.7,.5) to [out = 90, in = 225] (1,1);
\end{scope}

\begin{scope}[xshift =4cm, yshift = 1cm]
\draw[very thick, red] (.3,.5) -- (.7,.5);
\draw (0,0) to [out = 45, in = 270] (.3,.5) to [out = 90, in = -45] (0,1);
\draw (1,0) to [out = 135, in = 270] (.7,.5) to [out = 90, in = 225] (1,1);
\end{scope}

\draw (1,4) node{$+$};
\draw (3,4) node{$-$};
\draw (5.5,3.5) node{$-$};
\draw (6.5,1.5) node{$-$};
\draw (7.5,3.5) node{$-$};

\draw (5,-1) node {$T_2$};

\end{scope}

\end{scope}

\begin{scope}[yshift = - 7cm]

\begin{scope}[scale = .4, xshift = -1cm, yshift = -2.7cm]

\draw (2,3) to [out = 45, in = 180] (3,4.5) to [out = 0, in = 135] (4,4);
\draw (2,2) to [out = -45, in = 180] (3,.5) to [out = 0, in = -135] (4,1);
\draw (4,2) to [out = 135, in = 225] (4,3);
\draw (5,4) to [out = 45, in = 135] (6,4);
\draw (5,3) to [out = -45, in = 225] (6,3);
\draw (7,4) to [out = 45, in = 135] (8,4);
\draw (7,3) to [out = -45, in = 225] (8,3);
\draw (9,3) to [out = -45, in = 45] (9,2);
\draw (8,2) to [out = 135, in = 45, looseness=.5] (5,2);
\draw (8,1) to [out = 225, in = -45, looseness=.5](5,1);
\draw (1,2) to [out = 225, in = 90] (.5,1);
\draw (.5,1) arc (180:270:1);
\draw (1.5,0) -- (9,0);
\draw (9,0) arc (-90:0:.5);
\draw (9.5,.5) to [out = 90, in = -45] (9,1);
\draw (1,3) to [out = 135, in = 270] (.5,4);
\draw (.5,4) arc (180:90:1);
\draw (1.5,5) -- (9,5);
\draw (9,5) arc (90:0:.5);
\draw (9.5,4.5) to [out = 270, in = 45] (9,4);

\begin{scope}[xshift =1cm, yshift = 2cm]
\draw[very thick, red] (.5,.3) -- (.5,.7);
\draw (0,0) to [out = 45, in = 180] (.5,.3) to [out = 0, in = 135] (1,0);
\draw (0,1) to [out = -45, in = 180] (.5,.7) to [out = 0, in = 225] (1,1);
\end{scope}

\begin{scope}[xshift =4cm, yshift = 3cm]
\draw[very thick, blue] (.5,.3) -- (.5,.7);
\draw (0,0) to [out = 45, in = 180] (.5,.3) to [out = 0, in = 135] (1,0);
\draw (0,1) to [out = -45, in = 180] (.5,.7) to [out = 0, in = 225] (1,1);
\end{scope}

\begin{scope}[xshift =6cm, yshift = 3cm]
\draw[very thick, red] (.3,.5) -- (.7,.5);
\draw (0,0) to [out = 45, in = 270] (.3,.5) to [out = 90, in = -45] (0,1);
\draw (1,0) to [out = 135, in = 270] (.7,.5) to [out = 90, in = 225] (1,1);
\end{scope}

\begin{scope}[xshift =8cm, yshift = 3cm]
\draw[very thick, red] (.3,.5) -- (.7,.5);
\draw (0,0) to [out = 45, in = 270] (.3,.5) to [out = 90, in = -45] (0,1);
\draw (1,0) to [out = 135, in = 270] (.7,.5) to [out = 90, in = 225] (1,1);
\end{scope}

\begin{scope}[xshift =8cm, yshift = 1cm]
\draw[very thick, red] (.3,.5) -- (.7,.5);
\draw (0,0) to [out = 45, in = 270] (.3,.5) to [out = 90, in = -45] (0,1);
\draw (1,0) to [out = 135, in = 270] (.7,.5) to [out = 90, in = 225] (1,1);
\end{scope}

\begin{scope}[xshift =4cm, yshift = 1cm]
\draw[very thick, red] (.3,.5) -- (.7,.5);
\draw (0,0) to [out = 45, in = 270] (.3,.5) to [out = 90, in = -45] (0,1);
\draw (1,0) to [out = 135, in = 270] (.7,.5) to [out = 90, in = 225] (1,1);
\end{scope}

\draw (5.5,3.5) node{$-$};
\draw (6.5,1.5) node{$-$};
\draw (7.5,3.5) node{$-$};

\draw (5,-1) node {$T_3$};

\end{scope}
    
\begin{scope}[scale = .4, xshift = 9cm, yshift = -2.7cm]

\draw (2,3) to [out = 45, in = 180] (3,4.5) to [out = 0, in = 135] (4,4);
\draw (2,2) to [out = -45, in = 180] (3,.5) to [out = 0, in = -135] (4,1);
\draw (4,2) to [out = 135, in = 225] (4,3);
\draw (5,4) to [out = 45, in = 135] (6,4);
\draw (5,3) to [out = -45, in = 225] (6,3);
\draw (7,4) to [out = 45, in = 135] (8,4);
\draw (7,3) to [out = -45, in = 225] (8,3);
\draw (9,3) to [out = -45, in = 45] (9,2);
\draw (8,2) to [out = 135, in = 45, looseness=.5] (5,2);
\draw (8,1) to [out = 225, in = -45, looseness=.5](5,1);
\draw (1,2) to [out = 225, in = 90] (.5,1);
\draw (.5,1) arc (180:270:1);
\draw (1.5,0) -- (9,0);
\draw (9,0) arc (-90:0:.5);
\draw (9.5,.5) to [out = 90, in = -45] (9,1);
\draw (1,3) to [out = 135, in = 270] (.5,4);
\draw (.5,4) arc (180:90:1);
\draw (1.5,5) -- (9,5);
\draw (9,5) arc (90:0:.5);
\draw (9.5,4.5) to [out = 270, in = 45] (9,4);

\begin{scope}[xshift =1cm, yshift = 2cm]
\draw[very thick, red] (.5,.3) -- (.5,.7);
\draw (0,0) to [out = 45, in = 180] (.5,.3) to [out = 0, in = 135] (1,0);
\draw (0,1) to [out = -45, in = 180] (.5,.7) to [out = 0, in = 225] (1,1);
\end{scope}

\begin{scope}[xshift =6cm, yshift = 3cm]
\draw[very thick, blue] (.5,.3) -- (.5,.7);
\draw (0,0) to [out = 45, in = 180] (.5,.3) to [out = 0, in = 135] (1,0);
\draw (0,1) to [out = -45, in = 180] (.5,.7) to [out = 0, in = 225] (1,1);
\end{scope}

\begin{scope}[xshift =4cm, yshift = 3cm]
\draw[very thick, red] (.3,.5) -- (.7,.5);
\draw (0,0) to [out = 45, in = 270] (.3,.5) to [out = 90, in = -45] (0,1);
\draw (1,0) to [out = 135, in = 270] (.7,.5) to [out = 90, in = 225] (1,1);
\end{scope}

\begin{scope}[xshift =8cm, yshift = 3cm]
\draw[very thick, red] (.3,.5) -- (.7,.5);
\draw (0,0) to [out = 45, in = 270] (.3,.5) to [out = 90, in = -45] (0,1);
\draw (1,0) to [out = 135, in = 270] (.7,.5) to [out = 90, in = 225] (1,1);
\end{scope}

\begin{scope}[xshift =8cm, yshift = 1cm]
\draw[very thick, red] (.3,.5) -- (.7,.5);
\draw (0,0) to [out = 45, in = 270] (.3,.5) to [out = 90, in = -45] (0,1);
\draw (1,0) to [out = 135, in = 270] (.7,.5) to [out = 90, in = 225] (1,1);
\end{scope}

\begin{scope}[xshift =4cm, yshift = 1cm]
\draw[very thick, red] (.3,.5) -- (.7,.5);
\draw (0,0) to [out = 45, in = 270] (.3,.5) to [out = 90, in = -45] (0,1);
\draw (1,0) to [out = 135, in = 270] (.7,.5) to [out = 90, in = 225] (1,1);
\end{scope}

\draw (3.5,3.5) node{$-$};
\draw (6.5,1.5) node{$-$};
\draw (7.5,3.5) node{$-$};

\draw (5,-1) node {$T_4$};

\end{scope}

\begin{scope}[scale = .4, xshift = 19cm, yshift = -2.7cm]

\draw (2,3) to [out = 45, in = 180] (3,4.5) to [out = 0, in = 135] (4,4);
\draw (2,2) to [out = -45, in = 180] (3,.5) to [out = 0, in = -135] (4,1);
\draw (4,2) to [out = 135, in = 225] (4,3);
\draw (5,4) to [out = 45, in = 135] (6,4);
\draw (5,3) to [out = -45, in = 225] (6,3);
\draw (7,4) to [out = 45, in = 135] (8,4);
\draw (7,3) to [out = -45, in = 225] (8,3);
\draw (9,3) to [out = -45, in = 45] (9,2);
\draw (8,2) to [out = 135, in = 45, looseness=.5] (5,2);
\draw (8,1) to [out = 225, in = -45, looseness=.5](5,1);
\draw (1,2) to [out = 225, in = 90] (.5,1);
\draw (.5,1) arc (180:270:1);
\draw (1.5,0) -- (9,0);
\draw (9,0) arc (-90:0:.5);
\draw (9.5,.5) to [out = 90, in = -45] (9,1);
\draw (1,3) to [out = 135, in = 270] (.5,4);
\draw (.5,4) arc (180:90:1);
\draw (1.5,5) -- (9,5);
\draw (9,5) arc (90:0:.5);
\draw (9.5,4.5) to [out = 270, in = 45] (9,4);

\begin{scope}[xshift =1cm, yshift = 2cm]
\draw[very thick, red] (.5,.3) -- (.5,.7);
\draw (0,0) to [out = 45, in = 180] (.5,.3) to [out = 0, in = 135] (1,0);
\draw (0,1) to [out = -45, in = 180] (.5,.7) to [out = 0, in = 225] (1,1);
\end{scope}

\begin{scope}[xshift =8cm, yshift = 3cm]
\draw[very thick, blue] (.5,.3) -- (.5,.7);
\draw (0,0) to [out = 45, in = 180] (.5,.3) to [out = 0, in = 135] (1,0);
\draw (0,1) to [out = -45, in = 180] (.5,.7) to [out = 0, in = 225] (1,1);
\end{scope}

\begin{scope}[xshift =6cm, yshift = 3cm]
\draw[very thick, red] (.3,.5) -- (.7,.5);
\draw (0,0) to [out = 45, in = 270] (.3,.5) to [out = 90, in = -45] (0,1);
\draw (1,0) to [out = 135, in = 270] (.7,.5) to [out = 90, in = 225] (1,1);
\end{scope}

\begin{scope}[xshift =4cm, yshift = 3cm]
\draw[very thick, red] (.3,.5) -- (.7,.5);
\draw (0,0) to [out = 45, in = 270] (.3,.5) to [out = 90, in = -45] (0,1);
\draw (1,0) to [out = 135, in = 270] (.7,.5) to [out = 90, in = 225] (1,1);
\end{scope}

\begin{scope}[xshift =8cm, yshift = 1cm]
\draw[very thick, red] (.3,.5) -- (.7,.5);
\draw (0,0) to [out = 45, in = 270] (.3,.5) to [out = 90, in = -45] (0,1);
\draw (1,0) to [out = 135, in = 270] (.7,.5) to [out = 90, in = 225] (1,1);
\end{scope}

\begin{scope}[xshift =4cm, yshift = 1cm]
\draw[very thick, red] (.3,.5) -- (.7,.5);
\draw (0,0) to [out = 45, in = 270] (.3,.5) to [out = 90, in = -45] (0,1);
\draw (1,0) to [out = 135, in = 270] (.7,.5) to [out = 90, in = 225] (1,1);
\end{scope}

\draw (5.5,3.5) node{$-$};
\draw (6.5,1.5) node{$-$};
\draw (3.5,3.5) node{$-$};

\draw (5,-1) node {$T_5$};

\end{scope}

\end{scope}

\end{tikzpicture}\]
\caption{A six-crossing unknot that satisfies the assumptions of Lemma \ref{lemma:path2} together with the enhanced states $S_i$ and $T_j$ for $1\leq i,j\leq 5$.}
\label{figure:proofpath2}
\end{figure}
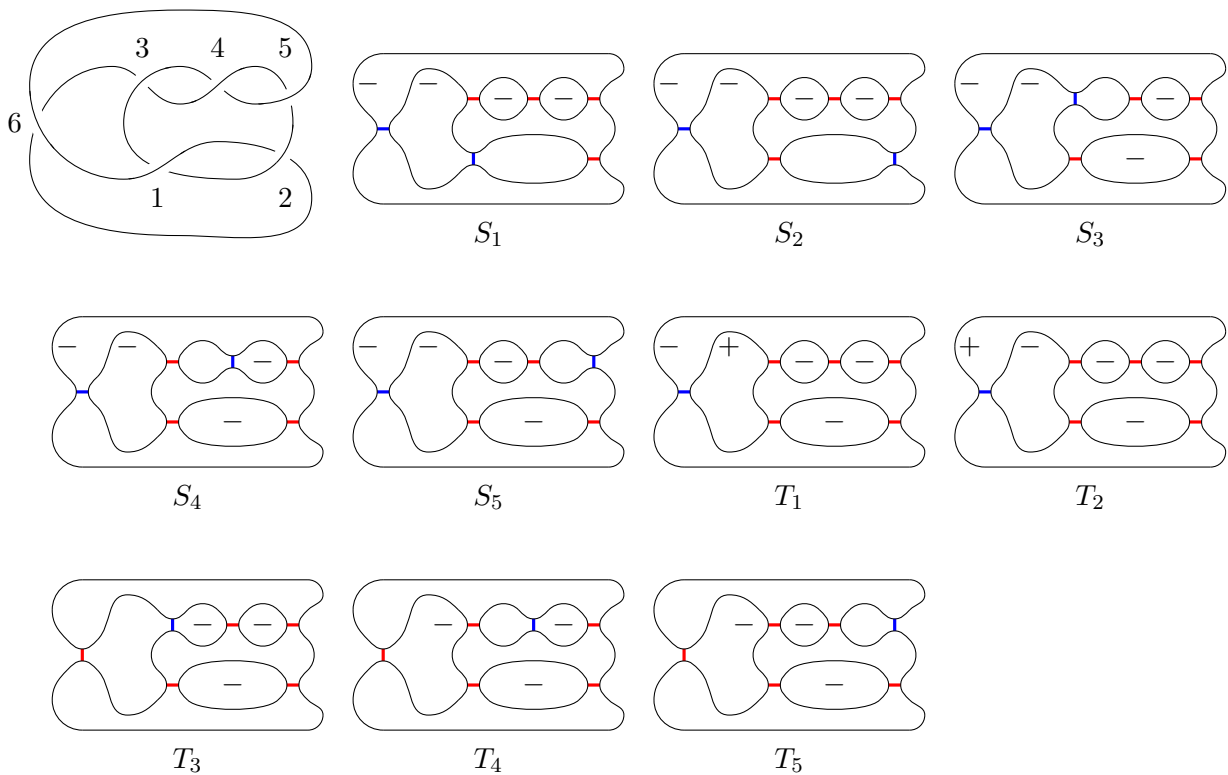

In the proof of Theorem \ref{theorem:almostalternating}, we first assume that the crossing where a resolution occurs is in a twist region of length one. Lemma \ref{lemma:twist} allows us to generalize to twist regions of arbitrary length.

\begin{figure}[ht]
\[\begin{tikzpicture}[scale = .7]
\begin{scope}[rounded corners = 2mm]
\draw (0,0) -- (1,1) -- (.7,1.3);
\draw (1,0) -- (.7,.3);
\draw (.3,.7) -- (0,1) -- (1,2);
\draw (.3,1.7) -- (0,2);
\fill (.5,2.3) circle (.03cm);
\fill (.5,2.5) circle (.03cm);
\fill (.5,2.7) circle (.03cm);
\draw (0,3) -- (1,4);
\draw (1,3) -- (.7,3.3);
\draw (.3,3.7) -- (0,4);

\draw (6,0) -- (6.3,.5) -- (6,1);
\draw (7,0) -- (6.7,.5) -- (7,1);

\draw (9,0) -- (9.5,.3) -- (10,0);
\draw (9,1) -- (9.5,.7) -- (10,1);

\draw (6.5,-.5) node{$D_A$};
\draw (9.5,-.5) node{$D_B$};

\end{scope}

\draw [decorate,
    decoration = {brace,
        raise=5pt,
        amplitude=5pt}] (-.5,0) --  (-.5,4);
        
\draw (.5,-.5) node{$D_n$};
\draw (-1.4,2) node{$n$};

\draw (3,0) -- (4,1);
\draw (4,0) -- (3.7,.3);
\draw (3,1) -- (3.3,.7);
\draw (3.5,-.5) node{$D$};

\end{tikzpicture}\]
\caption{The diagrams $D_n$, $D$, $D_A$ and $D_B$ in Lemma \ref{lemma:twist}.}
\label{figure:twist}
\end{figure}
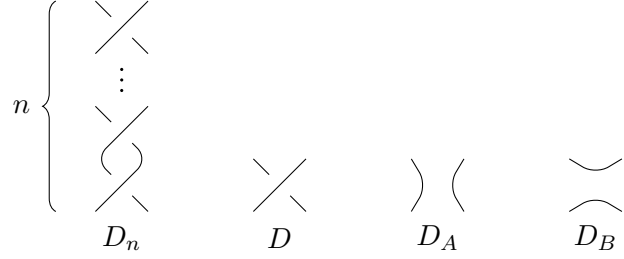

\begin{lemma}
\label{lemma:twist}
Let $D$, $D_A$, $D_B$, and $D_n$ be as in Figure \ref{figure:twist}. If $\ukh^{i,j}(D)$ and $\ukh^{i-k,j-2k+1}(D_B)$ are trivial for all $k> 1$, then $\ukh^{i,j}(D_n)$ is trivial for all $n\geq 1$.
\end{lemma}
\begin{proof}
We proceed by induction on $n$. Since $D_1=D$ and $\ukh^{i,j}(D)$ is trivial, the base case is complete. Suppose that $\ukh^{i,j}(D_{n-1})$ is trivial. An $A$-resolution of a crossing in the twist region of $D_n$ results in the diagram $D_{n-1}$, and a $B$-resolution of crossing in the twist region of $D_n$ results in a diagram that can be transformed into $D_B$ via $n-1$ negative Reidemeister 1 moves. Therefore the long exact sequence \eqref{equation:LES} becomes
\[\cdots \to \ukh^{i-n,j-2n+1}(D_B) \to \ukh^{i,j}(D_n) \to \ukh^{i,j}(D_{n-1})\to\cdots.\]
Since both $\ukh^{i-n,j-2n+1}(D_B)$ and $\ukh^{i,j}(D_{n-1})$ are trivial, it follows that $\ukh^{i,j}(D_n)$ is also trivial.
\end{proof}

Let $D$ be an $A$-almost alternating diagram of the nonsplit link $L$. Theorem \ref{theorem:almostalternating} proves that  $\ukh^{3,4-s_A(D)}(D)$ is trivial, which will then imply that $Kh(D)^{i_{\min}(L)+2,j_{\min}(L)+2}(L)$ is also trivial.
\begin{theorem}
\label{theorem:almostalternating}
If $D$ is $A$-almost alternating, then $\ukh^{3,4-s_A(D)}(D)$ is trivial.
\end{theorem}
\begin{proof}
Suppose that $D$ is $A$-almost alternating and that the tangle containing all of the crossings of $D$ other than the dealternator is $R$. We proceed by induction on $c$, the number of crossings of $D$. There are four cases to consider.
\begin{enumerate}
\item For every crossing in $R$, both $D_A$ and $D_B$ are not $A$-almost alternating.
\item $R$ has a crossing such that both $D_A$ and $D_B$ are $A$-almost alternating,
\item $R$ has a crossing such that $D_A$ is $A$-almost alternating, but $D_B$ is not $A$-almost alternating, and
\item $R$ has a crossing such that $D_B$ is $A$-almost alternating, but $D_A$ is not $A$-almost alternating.
\end{enumerate}

\noindent\textbf{Case 1.} Suppose that both $D_A$ and $D_B$ are not $A$-almost alternating for every crossing in the tangle $R$. This is the base case of the induction, and Lemma \ref{lemma:base} establishes this result.

\noindent\textbf{Case 2.} Suppose that $R$ has a crossing such that both $D_A$ and $D_B$ are $A$-almost alternating. Suppose that $D_A$ has $k\geq 0$ nugatory crossings. Let $D_A^{\text{red}}$ be the diagram $D_A$ with all nugatory crossings removed. Because $D_A$ is $A$-almost alternating, so is $D_A^{\text{red}}$. All nugatory crossings in $D_A$ are positive. Lemma \ref{lemma:Reidemeister} and the fact that $s_A(D_A)=s_A(D_A^{\text{red}})+k$ imply that $\ukh^{3,4-s_A(D_A)}(D_A) \cong \ukh^{3,4-s_A(D_A^{\text{red}})}(D_A^{\text{red}})$.
The long exact sequence in this case is 
\[ \cdots\to\ukh^{2,2-s_A(D_B)}(D_B)\to\ukh^{3,4-s_A(D)}(D)\to\ukh^{3,4-s_A(D_A^{\text{red}})}(D_A^{\text{red}})\to\cdots. \]
Lemma \ref{lemma:path2} implies $\ukh^{2,2-s_A(D_B)}(D_B)$ is trivial, and $\ukh^{3,4-s_A(D_A^{\text{red}})}(D_A^{\text{red}})$ is trivial by the inductive hypothesis. Therefore $\ukh^{3,4-s_A(D)}(D)$ is also trivial.

Now suppose that $D_B$ has $\ell\geq 0$ nugatory crossings. After potentially flyping, the nugatory crossings of $D_B$ can be concentrated into a twist region of $D$. Let $D_1$ be the diagram where the twist region of $D$ is a single crossing. The case where $\ell=0$ is the same as the case above where $k=0$ and implies that $\ukh^{3,4-s_A(D_1)}(D_1)$ is trivial. Because $\ukh^{3-\ell, 2-s_A(D_B)-2\ell}(D_B)$ is trivial, Lemma \ref{lemma:twist} implies that $\ukh^{3,4-s_A(D)}(D)$ is also trivial.

\noindent\textbf{Case 3.} Suppose that $R$ has a crossing such that $D_A$ is $A$-almost alternating, but $D_B$ is not $A$-almost alternating. Let $e$ and $\overline{e}$ be the edges in $G$ and $\overline{G}$ respectively associated to the crossing that is resolved to obtain $D_A$ and $D_B$. Since $D_A$ is $A$-almost alternating, but $D_B$ is not, $e$ is on a path of length three between $u_1$ and $u_2$, but $\overline{e}$ is not on a path of length two between $v_1$ and $v_2$. Thus the diagrams $D$, $D_A$, and $D_B$ are as in Figure \ref{figure:case3}.
\begin{figure}[ht]
\[\begin{tikzpicture}[scale = .7]
\draw (0,0) rectangle (5.25,2);
\draw (0,3) rectangle (5.25,5);
\draw (2.625,1) node{$R_2$};
\draw (2.625,4) node{$R_1$};
\draw (.35,1.8) node{\tiny{$-$}};
\draw (1.2,1.8) node{\tiny{$+$}};
\draw (1.55,1.8) node{\tiny{$-$}};
\draw (2.45,1.8) node{\tiny{$+$}};
\draw (2.8,1.8) node{\tiny{$-$}};
\draw (3.7,1.8) node{\tiny{$+$}};
\draw (4.05,1.8) node{\tiny{$-$}};
\draw (4.95,1.8) node{\tiny{$+$}};

\draw (.35,3.2) node{\tiny{$+$}};
\draw (1.2,3.2) node{\tiny{$-$}};
\draw (1.55,3.2) node{\tiny{$+$}};
\draw (2.45,3.2) node{\tiny{$-$}};
\draw (2.8,3.2) node{\tiny{$+$}};
\draw (3.7,3.2) node{\tiny{$-$}};
\draw (4.05,3.2) node{\tiny{$+$}};
\draw (4.95,3.2) node{\tiny{$-$}};

\draw (1.4,2.5) node{$u_2$};
\draw (0,2.5) node{$u_1$};

\draw (0.25,3) -- (1.25,2);
\draw (.25,2 ) -- (.5,2.25);
\draw (1,2.75) -- (1.25,3);

\draw (1.5,2) -- (1.75,2.25);
\draw (2.25, 2.75) -- (2.5,3);
\draw (1.5,3) -- (2.5,2);

\begin{scope}[xshift = 1.25cm]
\draw (1.5,2) -- (1.75,2.25);
\draw (2.25, 2.75) -- (2.5,3);
\draw (1.5,3) -- (2.5,2);
\end{scope}

\begin{scope}[xshift = 2.5cm]
\draw (1.5,2) -- (1.75,2.25);
\draw (2.25, 2.75) -- (2.5,3);
\draw (1.5,3) -- (2.5,2);
\end{scope}

\draw (2.625,-1) node{$D$};

\begin{scope}[xshift = 7 cm]
\draw (0,0) rectangle (5.25,2);
\draw (0,3) rectangle (5.25,5);
\draw (2.625,1) node{$R_2$};
\draw (2.625,4) node{$R_1$};
\draw (.35,1.8) node{\tiny{$-$}};
\draw (1.2,1.8) node{\tiny{$+$}};
\draw (1.55,1.8) node{\tiny{$-$}};
\draw (2.45,1.8) node{\tiny{$+$}};
\draw (2.8,1.8) node{\tiny{$-$}};
\draw (3.7,1.8) node{\tiny{$+$}};
\draw (4.05,1.8) node{\tiny{$-$}};
\draw (4.95,1.8) node{\tiny{$+$}};

\draw (.35,3.2) node{\tiny{$+$}};
\draw (1.2,3.2) node{\tiny{$-$}};
\draw (1.55,3.2) node{\tiny{$+$}};
\draw (2.45,3.2) node{\tiny{$-$}};
\draw (2.8,3.2) node{\tiny{$+$}};
\draw (3.7,3.2) node{\tiny{$-$}};
\draw (4.05,3.2) node{\tiny{$+$}};
\draw (4.95,3.2) node{\tiny{$-$}};

\draw (1.4,2.5) node{$u_2$};
\draw (0,2.5) node{$u_1$};

\draw (0.25,3) -- (1.25,2);
\draw (.25,2 ) -- (.5,2.25);
\draw (1,2.75) -- (1.25,3);

\draw[rounded corners = 2mm] (1.5,2) -- (1.9,2.5) -- (1.5,3);
\draw[rounded corners = 2mm] (2.5,2) -- (2.1,2.5) -- (2.5,3);

\begin{scope}[xshift = 1.25cm]
\draw (1.5,2) -- (1.75,2.25);
\draw (2.25, 2.75) -- (2.5,3);
\draw (1.5,3) -- (2.5,2);
\end{scope}

\begin{scope}[xshift = 2.5cm]
\draw (1.5,2) -- (1.75,2.25);
\draw (2.25, 2.75) -- (2.5,3);
\draw (1.5,3) -- (2.5,2);
\end{scope}

\draw (2.625,-1) node{$D_A$};

\end{scope}

\begin{scope}[xshift = 14 cm]
\draw (0,0) rectangle (5.25,2);
\draw (0,3) rectangle (5.25,5);
\draw (2.625,1) node{$R_2$};
\draw (2.625,4) node{$R_1$};
\draw (.35,1.8) node{\tiny{$-$}};
\draw (1.2,1.8) node{\tiny{$+$}};
\draw (1.55,1.8) node{\tiny{$-$}};
\draw (2.45,1.8) node{\tiny{$+$}};
\draw (2.8,1.8) node{\tiny{$-$}};
\draw (3.7,1.8) node{\tiny{$+$}};
\draw (4.05,1.8) node{\tiny{$-$}};
\draw (4.95,1.8) node{\tiny{$+$}};

\draw (.35,3.2) node{\tiny{$+$}};
\draw (1.2,3.2) node{\tiny{$-$}};
\draw (1.55,3.2) node{\tiny{$+$}};
\draw (2.45,3.2) node{\tiny{$-$}};
\draw (2.8,3.2) node{\tiny{$+$}};
\draw (3.7,3.2) node{\tiny{$-$}};
\draw (4.05,3.2) node{\tiny{$+$}};
\draw (4.95,3.2) node{\tiny{$-$}};

\draw (1.4,2.5) node{$u_2$};
\draw (0,2.5) node{$u_1$};

\draw (0.25,3) -- (1.25,2);
\draw (.25,2 ) -- (.5,2.25);
\draw (1,2.75) -- (1.25,3);

\draw[rounded corners = 2mm] (1.5,2) -- (2,2.4) -- (2.5,2);
\draw[rounded corners = 2mm] (1.5,3) -- (2,2.7) -- (2.5,3);

\begin{scope}[xshift = 1.25cm]
\draw (1.5,2) -- (1.75,2.25);
\draw (2.25, 2.75) -- (2.5,3);
\draw (1.5,3) -- (2.5,2);
\end{scope}

\begin{scope}[xshift = 2.5cm]
\draw (1.5,2) -- (1.75,2.25);
\draw (2.25, 2.75) -- (2.5,3);
\draw (1.5,3) -- (2.5,2);
\end{scope}

\draw (2.625,-1) node{$D_B$};

\end{scope}

\end{tikzpicture}\]

\caption{The diagrams $D$, $D_A$, and $D_B$ in case 3.}
\label{figure:case3}
\end{figure}
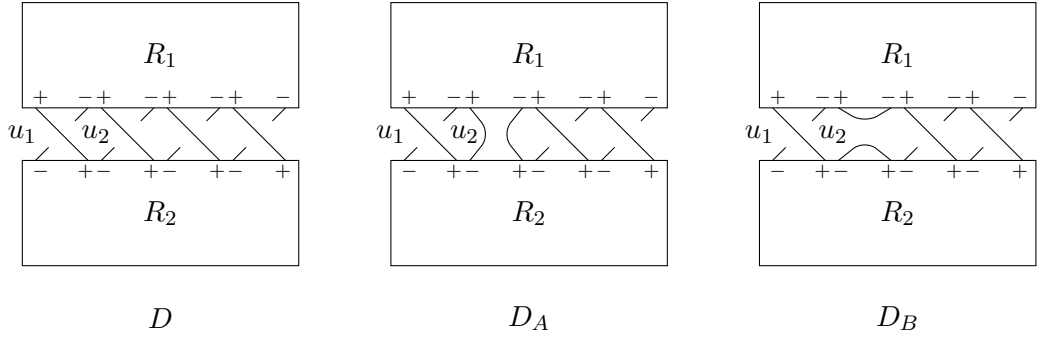

Suppose that $D_A$ has $k\geq 0$ nugatory crossings. Let $D_A^{\text{red}}$ be the diagram $D_A$ with all nugatory crossings removed. Because $D_A$ is $A$-almost alternating, so is $D_A^{\text{red}}$, and all nugatory crossings in $D_A$ are positive. Lemma \ref{lemma:Reidemeister} and the fact that $s_A(D_A)=s_A(D_A^{\text{red}})+k$ imply that $\ukh^{3,4-s_A(D_A)}(D_A) \cong \ukh^{3,4-s_A(D_A^{\text{red}})}(D_A^{\text{red}})$. The long exact sequence in this case is 
\[ \cdots\to\ukh^{2,2-s_A(D_B)}(D_B)\to\ukh^{3,4-s_A(D)}(D)\to\ukh^{3,4-s_A(D_A^{\text{red}})}(D_A^{\text{red}})\to\cdots. \]
Since $D_A^{\text{red}}$ is a reduced, $A$-almost alternating diagram, the inductive hypothesis implies that $\ukh^{3,4-s_A(D_A^{\text{red}})}(D_A^{\text{red}})$ is trivial. Since $D_B$ is an almost alternating diagram where the only crossing in the boundary of $u_1$ and $u_2$ is the deatlernator and where $\adj(u_1,u_2)=1$, Lemma \ref{lemma:path2} implies that $\ukh^{2,2-s_A(D_B)}(D_B)$ is trivial. Therefore $\ukh^{3,4-s_A(D)}(D)$ is also trivial. 

If $D_B$ has $\ell\geq 0$ nugatory crossings, then Lemma \ref{lemma:twist} implies the result just like it did in case 1.

\noindent\textbf{Case 4.} Suppose that $R$ has a crossing such that $D_B$ is $A$-almost alternating, but $D_A$ is not $A$-almost alternating. Let $e$ and $\overline{e}$ be the edges in $G$ and $\overline{G}$ respectively associated to the crossing that is resolved to obtain $D_A$ and $D_B$. Since $D_B$ is $A$-almost alternating, but $D_A$ is not, $e$ is not on a path of length three between $u_1$ and $u_2$, but $\overline{e}$ is on a path of length two between $v_1$ and $v_2$. Thus the diagrams $D$, $D_A$ and $D_B$ are as in Figure \ref{figure:case4}.
\begin{figure}[ht]
\[\begin{tikzpicture}[scale=.8]

\draw (0,0) rectangle (4.5,2);
\draw (0,3) rectangle (4.5,5);

\draw  (2.25,1) node{$R_2$};
\draw (2.25,4) node{$R_1$};

\draw (.25,2) -- (1.25,3);
\draw (.25,3) -- (.5,2.75);
\draw (1,2.25) -- (1.25,2);
\draw (1.75,3) -- (2.75,2);
\draw (1.75,2) -- (2,2.25);
\draw (2.5,2.75) -- (2.75,3);
\draw (3.25,3) -- (4.25,2);
\draw (3.25,2) -- (3.5,2.25);
\draw (4,2.75) -- (4.25,3);

\draw (0.25,3.2) node{$-$};
\draw (1.25,3.2) node{$+$};
\draw (1.75,3.2) node{$-$};
\draw (2.75,3.2) node{$+$};
\draw (3.25,3.2) node{$-$};
\draw (4,3.2) node{$+$};

\draw (0.25,1.8) node{$+$};
\draw (1.25,1.8) node{$-$};
\draw (1.75,1.8) node{$+$};
\draw (2.75,1.8) node{$-$};
\draw (3.25,1.8) node{$+$};
\draw (4,1.8) node{$-$};

\draw (0,2.5) node{$v_1$};
\draw (1.5,2.5) node{$v_2$};

\draw (2.25,-1) node{$D$};

\begin{scope}[xshift = 6cm]

\draw (0,0) rectangle (4.5,2);
\draw (0,3) rectangle (4.5,5);

\draw  (2.25,1) node{$R_2$};
\draw (2.25,4) node{$R_1$};

\draw (.25,2) -- (1.25,3);
\draw (.25,3) -- (.5,2.75);
\draw (1,2.25) -- (1.25,2);

\draw (1.75,3) to[out = -30, in = 180]
(2.25,2.75) to [out = 0, in = 210]
(2.75,3);
\draw (1.75,2) to [out = 30, in = 180]
(2.25,2.25) to [out = 0, in = 150]
(2.75,2);

\draw (3.25,3) -- (4.25,2);
\draw (3.25,2) -- (3.5,2.25);
\draw (4,2.75) -- (4.25,3);

\draw (0.25,3.2) node{$-$};
\draw (1.25,3.2) node{$+$};
\draw (1.75,3.2) node{$-$};
\draw (2.75,3.2) node{$+$};
\draw (3.25,3.2) node{$-$};
\draw (4,3.2) node{$+$};

\draw (0.25,1.8) node{$+$};
\draw (1.25,1.8) node{$-$};
\draw (1.75,1.8) node{$+$};
\draw (2.75,1.8) node{$-$};
\draw (3.25,1.8) node{$+$};
\draw (4,1.8) node{$-$};

\draw (0,2.5) node{$v_1$};
\draw (1.5,2.5) node{$v_2$};

\draw (2.25,-1) node{$D_A$};
\end{scope}

\begin{scope}[xshift = 12cm]

\draw (0,0) rectangle (4.5,2);
\draw (0,3) rectangle (4.5,5);

\draw  (2.25,1) node{$R_2$};
\draw (2.25,4) node{$R_1$};

\draw (.25,2) -- (1.25,3);
\draw (.25,3) -- (.5,2.75);
\draw (1,2.25) -- (1.25,2);

\draw (1.75,3) to[out = -60, in = 90]
(2,2.5) to [out = 270, in = 60]
(1.75,2);
\draw (2.75,3) to [out = 240, in = 90]
(2.5,2.5) to [out = 270, in = 120]
(2.75,2);

\draw (3.25,3) -- (4.25,2);
\draw (3.25,2) -- (3.5,2.25);
\draw (4,2.75) -- (4.25,3);

\draw (0.25,3.2) node{$-$};
\draw (1.25,3.2) node{$+$};
\draw (1.75,3.2) node{$-$};
\draw (2.75,3.2) node{$+$};
\draw (3.25,3.2) node{$-$};
\draw (4,3.2) node{$+$};

\draw (0.25,1.8) node{$+$};
\draw (1.25,1.8) node{$-$};
\draw (1.75,1.8) node{$+$};
\draw (2.75,1.8) node{$-$};
\draw (3.25,1.8) node{$+$};
\draw (4,1.8) node{$-$};

\draw (0,2.5) node{$v_1$};
\draw (1.5,2.5) node{$v_2$};

\draw (2.25,-1) node{$D_B$};
\end{scope}

\end{tikzpicture}\]
\caption{The diagrams $D$, $D_A$, and $D_B$ in case $4$.}
\label{figure:case4}
\end{figure}
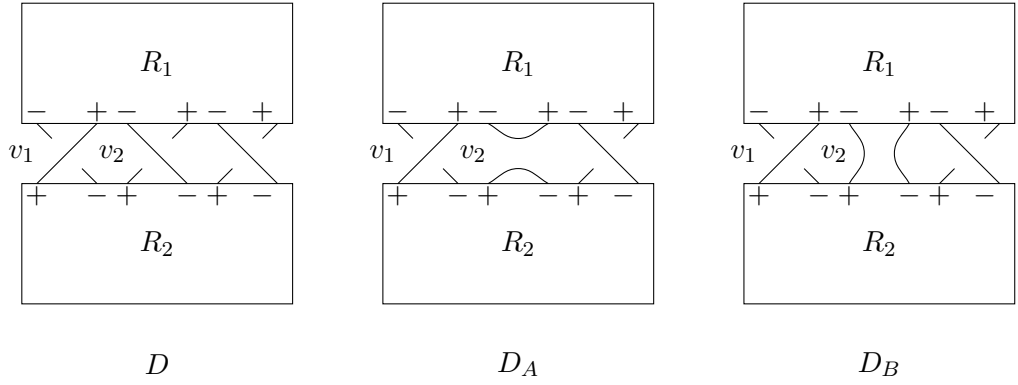

The diagram $D_B$ can have at most one nugatory crossing, and if it does, that crossing must be the rightmost crossing in $D_B$ above. If $D_B$ has this nugatory crossing, let $D_B^{\text{red}}$ be the diagram $D_B$ with the nugatory crossing removed, and if $D_B$ has no nugatory crossings, then let $D_B^{\text{red}}$ be $D_B$. In either case, $D_B^{\text{red}}$ is $A$-almost alternating. 

The diagram $D_A$ is a non-alternating diagram of an alternating link. It can be transformed into the alternating diagram $D_A^{\text{alt}}$ by a flype followed by a Reidemeister 2 move, as in Figure \ref{figure:case4flype}.
\begin{figure}[ht]
\[\begin{tikzpicture}[scale=.8]

\draw (0,0) rectangle (4.5,2);
\draw (0,3) rectangle (4.5,5);

\draw  (2.25,1) node{$R_2$};
\draw (2.25,4) node{$R_1$};

\draw (.25,2) -- (1.25,3);
\draw (.25,3) -- (.5,2.75);
\draw (1,2.25) -- (1.25,2);
\draw (1.75,3) to[out = -30, in = 180]
(2.25,2.75) to [out = 0, in = 210]
(2.75,3);
\draw (1.75,2) to [out = 30, in = 180]
(2.25,2.25) to [out = 0, in = 150]
(2.75,2);
\draw (3.25,3) -- (4.25,2);
\draw (3.25,2) -- (3.5,2.25);
\draw (4,2.75) -- (4.25,3);

\draw (0.25,3.2) node{$-$};
\draw (1.25,3.2) node{$+$};
\draw (1.75,3.2) node{$-$};
\draw (2.75,3.2) node{$+$};
\draw (3.25,3.2) node{$-$};
\draw (4,3.2) node{$+$};

\draw (0.25,1.8) node{$+$};
\draw (1.25,1.8) node{$-$};
\draw (1.75,1.8) node{$+$};
\draw (2.75,1.8) node{$-$};
\draw (3.25,1.8) node{$+$};
\draw (4,1.8) node{$-$};


\draw (2.25,-1) node{$D_A$};

\begin{scope}[xshift = 6cm]

\draw (1.5,0) rectangle (3,2);
\draw (1.5,3) rectangle (3,5);

\draw  (2.25,1) node{$R_2$};
\draw (2.25,4) node{$R_1$};

\begin{knot}[
    consider self intersections,
    clip width = 4,
    ignore endpoint intersections = true,
    end tolerance = 2pt
    ]
    \strand (1.5,.5) to [out = 135, in = 270]
    (.5,2) to [out = 90, in = 225]
    (1.5,3.5);
    \strand (1.5,1.5) to [out = 135, in = 270]
    (.5,3) to [out = 90, in = 225]
    (1.5,4.5);
    \strand (3,.5) to [out = 45, in = 270]
    (4,2) to [out = 90, in = 315]
    (3,3.5);
    \strand (3,1.5) to [out = 45, in = 270]
    (4,3) to [out = 90, in = 315]
    (3,4.5);

\end{knot}

\draw (1.75,3) to[out = -30, in = 180]
(2.25,2.75) to [out = 0, in = 210]
(2.75,3);
\draw (1.75,2) to [out = 30, in = 180]
(2.25,2.25) to [out = 0, in = 150]
(2.75,2);

\draw (1.75,1.8) node{$+$};
\draw (2.75,1.8) node{$-$};
\draw (1.7,1.5) node{$-$};
\draw (1.7,.5) node{$+$};
\draw (2.8,.5) node{$-$};
\draw (2.8,1.5) node{$+$};

\draw (1.75,3.2) node{$-$};
\draw (2.75,3.2) node{$+$};
\draw (1.7,3.5) node{$+$};
\draw (1.7,4.5) node{$-$};
\draw (2.8,4.5) node{$+$};
\draw (2.8,3.5) node{$-$};


\draw (2.25,-1) node{$D_A$};
\end{scope}

\begin{scope}[xshift = 12cm]

\draw (1.5,0) rectangle (3,2);
\draw (1.5,3) rectangle (3,5);

\draw  (2.25,1) node{$R_2$};
\draw (2.25,4) node[yscale=-1]{$R_1$};

    \draw (1.5,.5) to [out = 135, in = 270]
    (0,2.5) to [out = 90, in = 225]
    (1.5,4.5);
    \draw (1.5,1.5) to [out = 150, in = 270]
    (.5,2.5) to [out = 90, in = 210]
    (1.5,3.5);
    \draw (3,.5) to [out = 45, in = 270]
    (4.5,2.5) to [out = 90, in = 315]
    (3,4.5);
    \draw (3,1.5) to [out = 30, in = 270]
    (4,2.5) to [out = 90, in = 330]
    (3,3.5);

\draw (1.75,2) to [out = 30, in = 180]
(2.25,2.25) to [out = 0, in = 150]
(2.75,2);

\draw (1.75,5) to [out = 30, in = 180]
(2.25,5.25) to [out = 0, in = 150]
(2.75,5);

\draw (1.75,3.2) node{$-$};
\draw (2.75,3.2) node{$+$};
\draw (1.7,1.5) node{$-$};
\draw (1.7,.5) node{$+$};
\draw (2.8,.5) node{$-$};
\draw (2.8,1.5) node{$+$};

\draw (1.75,1.8) node{$-$};
\draw (2.75,1.8) node{$+$};
\draw (1.7,3.5) node{$+$};
\draw (1.7,4.5) node{$-$};
\draw (2.8,4.5) node{$+$};
\draw (2.8,3.5) node{$-$};


\draw (2.25,-1) node{$D_A^{\alt}$};
\end{scope}

\end{tikzpicture}\]
\caption{The diagram $D_A$ can be transformed into the alternating diagram $D_A^{\text{alt}}$.}
\label{figure:case4flype}
\end{figure}
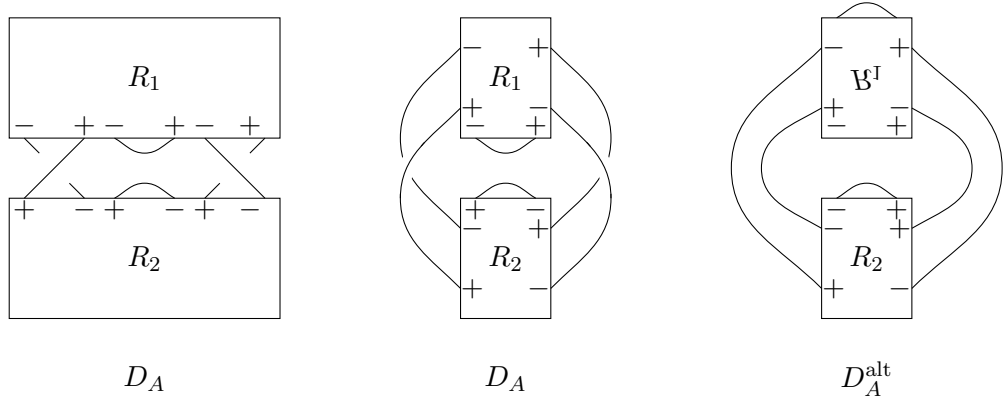

The long exact sequence in this case is
\[\cdots \ukh^{2,3-s_A(D)}(D_B) \to \ukh^{3,4-s_A(D)}(D)\to \ukh^{3,4-s_A(D)}(D_A)\to\cdots. \]
Since $s_A(D) = s_A(D_A) = s_A(D_B)+1$, the long exact sequence can be written as
\[\cdots \ukh^{2,2-s_A(D_B)}(D_B) \to \ukh^{3,4-s_A(D)}(D)\to \ukh^{3,4-s_A(D_A)}(D_A)\to\cdots. \]
If there is a nugatory crossing in $D_B$, it is negative. Let $k$ be the number of nugatory crossings in $D_B$. Then 
\[\ukh^{2,2-s_A(D_B)}(D_B) \cong \ukh^{2-k,2-s_A(D_B^\text{red})-2k}(D_B^{\text{red}})\] is trivial by part (1) in Theorem \ref{theorem:main} (which is also \cite[Theorem 1.3]{DasLow:Extremal}) since $D_B^{\text{red}}$ is $A$-almost alternating. Also, since $s_A(D_A)=s_A(D_A^{\text{alt}})+1$, it follows that $\ukh^{3,4-s_A(D_A)}(D_A)\cong \ukh^{2,2-s_A(D_A^{\text{alt}})}(D_A^{\text{alt}})$ is trivial because $D_A^{\text{alt}}$ is a reduced alternating diagram and hence $\ukh^{i,j}(D_A^{\text{alt}})$ is trivial unless $s_A(D_A^{\text{alt}}) -2 \leq 2i-j\leq s_A(D_A^{\text{alt}})$. Therefore $\ukh^{3,4-s_A(D)}(D)$ is trivial, as desired.

Now suppose that $D_A$ has $k$ nugatory crossings. After potentially flyping $D$ a number of times, we may consider $D$, $D_A$, and $D_B$ to be the diagrams in Figure \ref{figure:case4twist}.
\begin{figure}[ht]
\[\begin{tikzpicture}[scale=.8]

\draw (0,0) rectangle (5,2);
\draw (0,6) rectangle (5,8);

\fill (2.5,4.85) circle (.05cm);
\fill (2.5,4.55) circle (.05cm);
\fill (2.5,5.15) circle (.05cm);

\begin{knot}[
	consider self intersections,
 	clip width = 5,
 	ignore endpoint intersections = true,
 ]
 \flipcrossings{4,1};
 \strand (2,2) to [out = 90, in =270]
 (3,3) to [out = 90, in = 270]
 (2,4);
 \strand (3,2) to [out = 90, in = 270]
 (2,3) to [out = 90, in = 270]
 (3,4);
 
 \strand (.25,2) to [out = 90, in = 270]
 (.25,3.5) to [out = 90, in = 270]
 (1.25,4.5) to [out = 90, in = 270]
 (1.25,6);
 \strand (1.25,2) to [out = 90, in =270]
 (1.25,3.5) to [out = 90, in = 270] 
 (.25,4.5) to [out = 90, in = 270]
 (.25,6);

  \strand (3.75,2) to [out = 90, in = 270]
 (3.75,3.5) to [out = 90, in = 270]
 (4.75,4.5) to [out = 90, in = 270]
 (4.75,6);
 \strand (4.75,2) to [out = 90, in =270]
 (4.75,3.5) to [out = 90, in = 270] 
 (3.75,4.5) to [out = 90, in = 270]
 (3.75,6);

\end{knot}

\draw (3,6) -- (3,5.7);
\draw (2,6) -- (2,5.7);

\draw [decorate,decoration={brace,amplitude=3pt}] (3.1,5.8) -- (3.1,2.2) node [black,midway,xshift = 13pt] {\tiny $k{+}1$};

\draw (.25,6.2) node{$-$};
\draw (1.25,6.2) node{$+$};
\draw (2,6.2) node{$-$};
\draw (3,6.2) node{$+$};
\draw (3.75,6.2) node{$-$};
\draw (4.75,6.2) node{$+$};

\draw (.25,1.8) node{$+$};
\draw (1.25,1.8) node{$-$};
\draw (2,1.8) node{$+$};
\draw (3,1.8) node{$-$};
\draw (3.75,1.8) node{$+$};
\draw (4.75,1.8) node{$-$};

\draw (2.5,1) node{$R_2$};
\draw (2.5,7) node{$R_1$};
\draw (2.5,-1) node{$D$};

\begin{scope}[xshift = 6cm]

\draw (0,0) rectangle (5,2);
\draw (0,6) rectangle (5,8);

\fill (2.5,4.85) circle (.05cm);
\fill (2.5,4.55) circle (.05cm);
\fill (2.5,5.15) circle (.05cm);

\begin{knot}[
	consider self intersections,
 	clip width = 5,
 	ignore endpoint intersections = true,
 ]
 \flipcrossings{3};
 \strand (2,2) to [out = 90, in =90]
 (3,2);
 
 \strand (2,4) to [out = 270, in = 90]
 (3,3) to [out = 270, in = 0]
 (2.5,2.6) to [out = 180, in = 270]
 (2,3) to [out = 90, in = 270]
 (3,4);

 \strand (.25,2) to [out = 90, in = 270]
 (.25,3.5) to [out = 90, in = 270]
 (1.25,4.5) to [out = 90, in = 270]
 (1.25,6);
 \strand (1.25,2) to [out = 90, in =270]
 (1.25,3.5) to [out = 90, in = 270] 
 (.25,4.5) to [out = 90, in = 270]
 (.25,6);

  \strand (3.75,2) to [out = 90, in = 270]
 (3.75,3.5) to [out = 90, in = 270]
 (4.75,4.5) to [out = 90, in = 270]
 (4.75,6);
 \strand (4.75,2) to [out = 90, in =270]
 (4.75,3.5) to [out = 90, in = 270] 
 (3.75,4.5) to [out = 90, in = 270]
 (3.75,6);

\end{knot}

\draw (3,6) -- (3,5.7);
\draw (2,6) -- (2,5.7);

\draw [decorate,decoration={brace,amplitude=3pt}] (3.1,5.8) -- (3.1,2.8) node [black,midway,xshift = 10pt] {\tiny$k$};

\draw (.25,6.2) node{$-$};
\draw (1.25,6.2) node{$+$};
\draw (2,6.2) node{$-$};
\draw (3,6.2) node{$+$};
\draw (3.75,6.2) node{$-$};
\draw (4.75,6.2) node{$+$};

\draw (.25,1.8) node{$+$};
\draw (1.25,1.8) node{$-$};
\draw (2,1.8) node{$+$};
\draw (3,1.8) node{$-$};
\draw (3.75,1.8) node{$+$};
\draw (4.75,1.8) node{$-$};

\draw (2.5,1) node{$R_2$};
\draw (2.5,7) node{$R_1$};
\draw (2.5,-1) node{$D_A$};

\end{scope}

\begin{scope}[xshift = 12cm]

\draw (0,0) rectangle (5,2);
\draw (0,6) rectangle (5,8);

\fill (2.5,4.85) circle (.05cm);
\fill (2.5,4.55) circle (.05cm);
\fill (2.5,5.15) circle (.05cm);

\begin{knot}[
	consider self intersections,
 	clip width = 5,
 	ignore endpoint intersections = true,
 ]
 \flipcrossings{3,1};
 \strand (2,2) to [out = 90, in =270]
 (2.3,2.5) to [out = 90, in = 270]
 (2,3)to [out = 90, in = 270]
 (3,4);
 \strand (3,2) to [out = 90, in = 270]
 (2.7,2.5) to [out=90, in =270]
 (3,3) to [out = 90, in = 270]
 (2,4);
 
 \strand (.25,2) to [out = 90, in = 270]
 (.25,3.5) to [out = 90, in = 270]
 (1.25,4.5) to [out = 90, in = 270]
 (1.25,6);
 \strand (1.25,2) to [out = 90, in =270]
 (1.25,3.5) to [out = 90, in = 270] 
 (.25,4.5) to [out = 90, in = 270]
 (.25,6);

  \strand (3.75,2) to [out = 90, in = 270]
 (3.75,3.5) to [out = 90, in = 270]
 (4.75,4.5) to [out = 90, in = 270]
 (4.75,6);
 \strand (4.75,2) to [out = 90, in =270]
 (4.75,3.5) to [out = 90, in = 270] 
 (3.75,4.5) to [out = 90, in = 270]
 (3.75,6);

\end{knot}

\draw (3,6) -- (3,5.7);
\draw (2,6) -- (2,5.7);

\draw [decorate,decoration={brace,amplitude=3pt}] (3.1,5.8) -- (3.1,3) node [black,midway,xshift = 10pt] {\tiny $k$};

\draw (.25,6.2) node{$-$};
\draw (1.25,6.2) node{$+$};
\draw (2,6.2) node{$-$};
\draw (3,6.2) node{$+$};
\draw (3.75,6.2) node{$-$};
\draw (4.75,6.2) node{$+$};

\draw (.25,1.8) node{$+$};
\draw (1.25,1.8) node{$-$};
\draw (2,1.8) node{$+$};
\draw (3,1.8) node{$-$};
\draw (3.75,1.8) node{$+$};
\draw (4.75,1.8) node{$-$};

\draw (2.5,1) node{$R_2$};
\draw (2.5,7) node{$R_1$};
\draw (2.5,-1) node{$D_B$};

\end{scope}

\end{tikzpicture}\]

\caption{The diagrams $D$, $D_A$, and $D_B$ in case $4$ when the resolution crossing is part of a twist region.}
\label{figure:case4twist}
\end{figure}
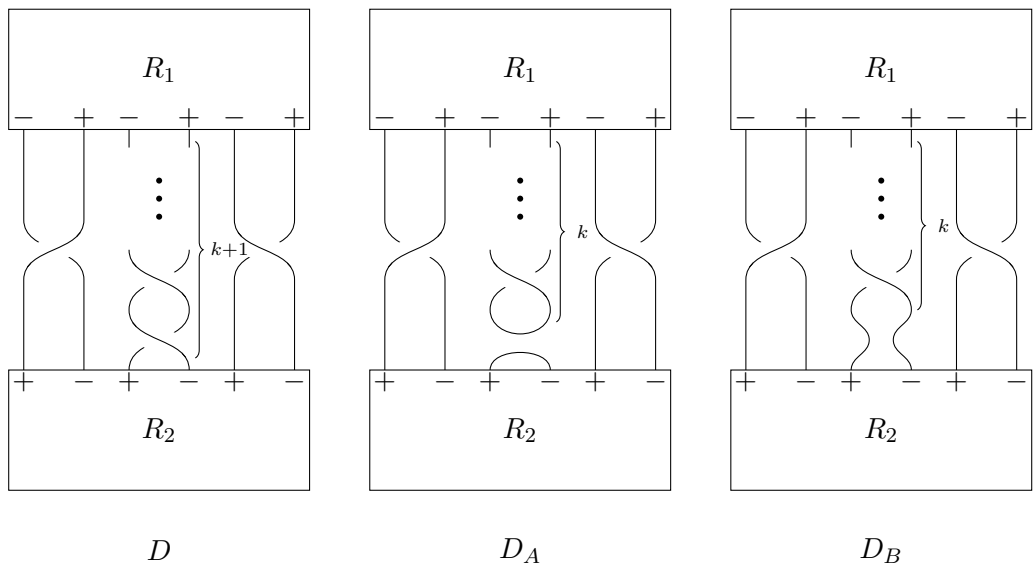
In this case, $D_B$ is a reduced $A$-almost alternating diagram, and $D_A$ is a non-alternating diagram of an alternating link. Let $D_A^{\text{alt}}$ be the alternating diagram obtained from $D_A$ by $k$ positive Reidemeister 1 moves, a flype, and a Reidemeister 2 move. Then $\ukh^{3,4-s_A(D_A)}(D_A) \cong \ukh^{2,3-s_A(D_A)+k}(D_A^{\text{alt}})\cong \ukh^{2,2-s_A(D_A^{\text{alt}})}(D_A^{\text{alt}})$. As in the previous case, $\ukh^{2,2-s_A(D_A^{\text{alt}})}(D_A^{\text{alt}})$ is trivial because $D_A^{\text{alt}}$ is a reduced alternating diagram. Since $D_B$ is a reduced $A$-almost alternating diagram $\ukh^{2,2-s_A(D_B)}(D_B)$ is trivial. Because the long exact sequence is 
\[\cdots \ukh^{2,2-s_A(D_B)}(D_B) \to \ukh^{3,4-s_A(D)}(D)\to \ukh^{3,4-s_A(D_A)}(D_A)\to\cdots \]
and both $\ukh^{2,2-s_A(D_B)}(D_B)$ and $\ukh^{3,4-s_A(D_A)}(D_A)$ are trivial, it follows that $\ukh^{3,4-s_A(D)}(D)$ is trivial, as desired.
\end{proof}

\subsection{$A$-adequate Turaev genus one links}

In this subsection, we prove in Theorem \ref{theorem:tg1} that if $D$ is an $A$-adequate Turaev genus one diagram of the nonsplit link $L$, then $\ukh^{2,2-s_A(D)}(D)\cong Kh^{i_{\min}(L)+2,j_{\min}(L)+2}(L)$ is trivial. Together Theorems \ref{theorem:almostalternating} and \ref{theorem:tg1} imply Theorem \ref{theorem:main}. 

The proof of Theorem \ref{theorem:tg1} broadly breaks down into two parts depending on the length of the shortest cycle in the state graph $G_A(D)$, called the \textit{girth} of $G_A(D)$. The state graph $G_A(D)$ is the graph whose vertices are in one-to-one correspondence with the components of the all-$A$ state of $D$ and whose edges are in one-to-one correspondence with the crossings of $D$; an edge $e$ in $G_A(D)$ is incident to two vertices $u$ and $v$ if and only if the components of the all-$A$ state corresponding to $u$ and $v$ respectively are incident to the crossing corresponding to $e$. Because $D$ is $A$-adequate, it follows that $\girth(G_A(D))\geq 2$. If the $\girth(G_A(D))\geq 3$, then the chain complex $CKh^{2,2-s_A(D)}(D)$ is trivial, and thus so is $\ukh^{2,2-s_A(D)}(D)$. Proving that $\ukh^{2,2-s_A(D)}(D)$ is trivial when $\girth(G_A(D))=2$ takes considerably more work and is the bulk of this subsection.

In the following lemma, we show that if $D$ is a diagram of a link $L$ such that the girth of $G_A(D)$ is at least three, then $\ukh^{2,2-s_A(D)}(D)$ is trivial (without the assumption that $D$ has Turaev genus one).
\begin{lemma}
\label{lemma:girth3}
Let $D$ be a diagram of a nonsplit link such that the girth of $G_A(D)$ is at least three. Then $\ukh^{2,2-s_A(D)}(D)$ is trivial.
\end{lemma}
\begin{proof}
For each homological grading $i$, define $m_i$ to be the minimum polynomial grading $j$ such that there is an enhanced state of $D$ with homological grading $i$ and polynomial grading $j$. Then $m_0=-s_A(D)$ because the minimum polynomial grading is obtained by labeling each component of the all-$A$ state with a $-$. If $i<\girth(G_A(D))$, then every Kauffman state in homological grading $i$ has $s_A(D) - i$ components. Therefore, the minimal polynomial grading is obtained by labeling each of these components with a $-$, and hence if $i<\girth(G_A(D))$, then $m_i = 2i - s_A(D)$. In particular, if $\girth(G_A(D))\geq 3$, then $m_2 = 4-s_A(D)$. Hence there are no enhanced states with homological grading $2$ and polynomial grading $2-s_A(D)$, and consequently $\ukh^{2,2-s_A(D)}(D)$ is trivial.
\end{proof}
 
For related work on the extremal Khovanov homology of links with high girth, see Sazdanovi\'c and Scofield \cite{SC:Girth}.

\begin{theorem}
\label{theorem:tg1}
Let $D$ be a Turaev genus one link diagram, as in Figure \ref{figure:tg1}. 
\begin{enumerate}
    \item If $D$ is $A$-adequate, then $\ukh^{2,2-s_A(D)}(D)$ is trivial.
    \item If $D$ is $B$-adequate, then $\ukh^{c(D)-2,c(D)+s_B(D)-2}(D)$ is trivial.
\end{enumerate}
\end{theorem}
\begin{proof}
We complete this proof by induction. The diagram $D$ with the fewest number of crossings satisfying the conditions of the theorem is the diagram of the trefoil with four crossings depicted in Figure \ref{figure:figure8tref}. Since this is an $A$-adequate diagram of the trefoil, it follows that $Kh^{2,2-s_A(D)}(D)$ is trivial.
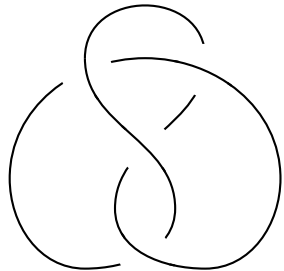
\begin{figure}[ht]
\[\begin{tikzpicture}[scale = .8, thick]
\begin{knot}[
    consider self intersections,
    clip width = 6,
    ignore endpoint intersections = true,
    end tolerance = 2pt
    ]
    \flipcrossings{1,4}
    \strand (.75,1.5) to [out = 270, in = 180]
    (2,0) to [out = 0, in = 270]
    (3.5,1) to [out = 90, in = 270]
    (2,3.5) to [out = 90, in = 90,looseness = 1.5]
    (4,3.5) to [out = 270, in = 90]
    (2.5,1) to [out = 270, in = 180]
    (4,0) to [out = 0, in = 270]
    (5.25,1.5) to [out = 90, in = 0]
    (3,3.5) to [out = 180, in = 90]
    (.75,1.5);
    
    \end{knot}
    \end{tikzpicture}\]
\caption{The base case to the proof of Theorem \ref{theorem:tg1} is this four crossing diagram of the trefoil.}
\label{figure:figure8tref}
\end{figure}

Let $G_A(D)$ have $n$ crossings for some $n>4$. Suppose by way of induction that if $D'$ is an $A$-adequate diagram with Turaev genus one and fewer than $n$ crossings, then $\ukh^{2,2-s_A(D')}(D')$ is trivial. Assume that the diagram $D$ is twist reduced. 

If $\girth(G_A(D))=2$, then there are two components of the all-$A$ state that share two different crossings. Partition the components of the all-$A$ state into two sets: the components that are completely within an alternating tangle $R_i$ (called interior components) and the components that are not completely contained inside a single alternating tangle $R_i$ (called exterior components). Suppose that there is a cycle of length two in $G_A(D)$ and that the crossings corresponding to the edges of the cycle are not in the same twist region. Because each tangle $R_i$ is alternating and twist reduced, at least one of the components corresponding to vertices in the cycle must be an external component. If the cycle contains vertices corresponding to an internal and an external component, then $D$ has the format as depicted in Figure \ref{figure:tg1case3}. If the cycle contains vertices corresponding to two external components, then $D$ has the format as depicted in Figure \ref{figure:tg1case4}.

\begin{figure}[ht]
\[\begin{tikzpicture}[scale=.6]

\draw (0,0) rectangle (1,3);
\draw (2,0) rectangle (3,3);
\begin{knot}[
    consider self intersections,
    clip width = 5,
    ignore endpoint intersections = true,
    end tolerance = 2pt
    ]
    \strand (1,2.75) to [out = 0, in = 180] (2,1.75);
    \strand (2,2.75) to [out = 180, in =0] (1,1.75);
    
    \strand (1,1.25) to [out = 0, in =180] (2,0.25);
    \strand (1,0.25) to [out = 0, in = 180] (2,1.25);
    
    \end{knot}
    
    \draw (1.5, 1.5) circle (2.5cm);
      \draw (7.5, 1.5) circle (2.5cm);
     \draw (-4.5, 1.5) circle (2.5cm);
     \draw (0,2.75) to [out = 180 , in = 45] (-2.73,3.27);
     \draw (0,.25) to [out = 180, in = -45] (-2.73, -.27);
     \draw (3,2.75) to [out = 0, in = 135] (5.73,3.27);
     \draw (3,.25) to [out = 0, in = 225] (5.73,-.27);
     \draw (-6.27,3.27) to [out = 135, in = 0] (-7.5,3.2);
     \draw (-6.27,-.27) to [out = 225, in = 0] (-7.5,-.2);
     \draw (9.27,3.27) to [out = 45, in = 180] (10.5,3.2);
     \draw (9.27,-.27) to [out = -45, in = 180] (10.5,-.2);
     
\draw (-4.5,1.5) node{$R_{2i-1}$};
\draw (7.5,1.5) node{$R_{2i+1}$};
\draw (-3,3) node{\tiny{$+$}};
\draw (-3,0) node{\tiny{$-$}};
\draw (-6,3) node{\tiny{$-$}};
\draw (-6,0) node{\tiny{$+$}};
\begin{scope}[xshift = 12cm]
\draw (-3,3) node{\tiny{$+$}};
\draw (-3,0) node{\tiny{$-$}};
\draw (-6,3) node{\tiny{$-$}};
\draw (-6,0) node{\tiny{$+$}};
\end{scope}

\draw (.25,.25) node{\tiny{$-$}};
\draw (.25,2.75) node{\tiny{$+$}};
\draw (.75,.25) node{\tiny{$+$}};
\draw (.75,1.25) node{\tiny{$-$}};
\draw (.75,1.75) node{\tiny{$+$}};
\draw (.75,2.75) node{\tiny{$-$}};

\draw (2.25,.25) node{\tiny{$-$}};
\draw (2.25,1.25) node{\tiny{$+$}};
\draw (2.25,1.75) node{\tiny{$-$}};
\draw (2.25,2.75) node{\tiny{$+$}};
\draw (2.75,.25) node{\tiny{$+$}};
\draw (2.75,2.75) node{\tiny{$-$}};


\end{tikzpicture}\]
    \caption{An interior component of the all-$A$ state shares two crossings not in a twist region with an exterior component. }
    \label{figure:tg1case3}
\end{figure}
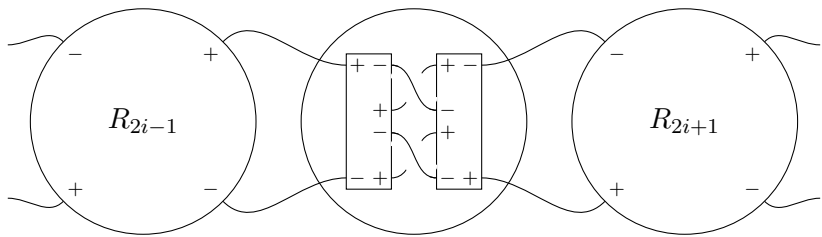

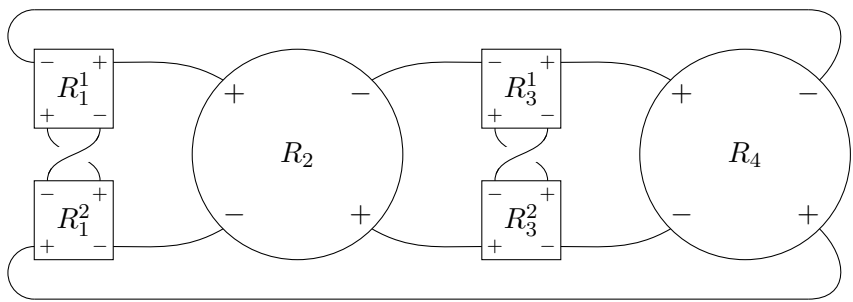
\begin{figure}[ht]
\[\begin{tikzpicture}[scale = .7]

\draw (0,0) rectangle (1.5,1.5);
\draw (0,2.5) rectangle (1.5,4);

\draw (5,2) circle (2cm);

\draw (8.5,0) rectangle (10,1.5);
\draw (8.5,2.5) rectangle (10,4);

\draw (13.5,2) circle (2cm);

\draw (.75,.75) node{$R_1^2$};
\draw (.75,3.25) node{$R_1^1$};

\draw (.25,.25) node{\tiny{$+$}};
\draw (1.25,.25) node{\tiny{$-$}};
\draw (.25,1.25) node{\tiny{$-$}};
\draw (1.25,1.25) node{\tiny{$+$}};

\draw (.25,2.75) node{\tiny{$+$}};
\draw (1.25,2.75) node{\tiny{$-$}};
\draw (.25,3.75) node{\tiny{$-$}};
\draw (1.25,3.75) node{\tiny{$+$}};

\begin{scope}[xshift = 8.5cm]
\draw (.25,.25) node{\tiny{$+$}};
\draw (1.25,.25) node{\tiny{$-$}};
\draw (.25,1.25) node{\tiny{$-$}};
\draw (1.25,1.25) node{\tiny{$+$}};

\draw (.25,2.75) node{\tiny{$+$}};
\draw (1.25,2.75) node{\tiny{$-$}};
\draw (.25,3.75) node{\tiny{$-$}};
\draw (1.25,3.75) node{\tiny{$+$}};

\draw (3.8,3.15) node{$+$};
\draw (3.8, .85)node{$-$};
\draw (6.2,3.15) node{$-$};
\draw (6.2,.85) node{$+$};
\end{scope}

\draw (5,2) node{$R_2$};

\draw (3.8,3.15) node{$+$};
\draw (3.8, .85)node{$-$};
\draw (6.2,3.15) node{$-$};
\draw (6.2,.85) node{$+$};

\draw (3.586,3.414) to [out = 150, in = 0] (1.5,3.75);
\draw (3.586,.586) to [out = 210, in = 0] (1.5,.25);
\draw (6.414,3.414) to [out = 30, in = 180] (8.5,3.75);
\draw (6.414,.586) to [out = -30, in = 180] (8.5,.25);

\begin{scope}[xshift = 8.5cm]
\draw (3.586,3.414) to [out = 150, in = 0] (1.5,3.75);
\draw (3.586,.586) to [out = 210, in = 0] (1.5,.25);
\end{scope}

\draw (9.25,.75) node{$R_3^2$};
\draw (9.25,3.25) node{$R_3^1$};
\draw (13.5,2) node{$R_4$};

\draw (0,.25) arc (90:270:.5cm);
\draw (0,-.75) -- (14.7,-.75);

\draw (0,3.75) arc (270:90:.5cm);
\draw (0,4.75) -- (14.7,4.75);

\draw (14.914, .586) to [out = -45, in = 0, looseness = 1.5] (14.7,-.75);
\draw (14.914,3.414) to [out = 45, in = 0, looseness=1.5] (14.7,4.75);

\begin{knot}[
    consider self intersections,
    clip width = 5,
    ignore endpoint intersections = true,
    end tolerance = 2pt
    ]
    \strand (.25,1.5) to [out = 90, in = 270] (1.25,2.5);
    \strand (.25,2.5) to [out = 270, in = 90] (1.25,1.5);
    
        \strand (8.75,1.5) to [out = 90, in = 270] (9.75,2.5);
    \strand (8.75,2.5) to [out = 270, in = 90] (9.75,1.5);

    \end{knot}

\end{tikzpicture}\]
    \caption{Two exterior components share two crossings not in a twist region.}
    \label{figure:tg1case4}
\end{figure}

The proof breaks down into the following four cases:
\begin{enumerate}
    \item $\girth(G_A(D))\geq 3$,
    \item $\girth(G_A(D))=2$ and $G_A(D)$ has a cycle of length two whose edges correspond to crossings in an $A$-twist region,
    \item $\girth(G_A(D))=2$ and $D$ is a diagram as in Figure \ref{figure:tg1case3}, or
    \item $\girth(G_A(D))=2$, $D$ is a diagram as in Figure \ref{figure:tg1case4}.
\end{enumerate}

\noindent\textbf{Case 1:} If $\girth(G_A(D))\geq 3$, then Lemma \ref{lemma:girth3} implies the result.\smallskip

\noindent\textbf{Case 2:} Suppose that $\girth(G_A(D))=2$ and that $G_A(D)$ contains $k\geq 2$ edges in an $A$-twist region. If a crossing in the $A$-twist region is resolved, then the long exact sequence is
\[\cdots \ukh^{1,-s_A(D_B)}(D_B)\to \ukh^{2,2-s_A(D)}(D)\to \ukh^{2,2-s_A(D_A)}(D_A)\to\cdots\]
where the diagram $D_B$ has $k-1$ negative Reidemeister one twists. Let $D_B^{\text{red}}$ be the diagram $D_B$ after removing $k-1$ negative Reidemeister one twists. Then $\ukh^{1,-s_A(D_B)}(D_B)\cong \ukh^{2-k,2-s_A(D_B^{\text{red}})}(D_B^{\text{red}})$. The diagram $D_A$ is Turaev genus one, $A$-adequate, and has one fewer crossing than $D$. Thus the inductive hypothesis implies that $\ukh^{2,2-s_A(D_A)}(D_A)$ is trivial. For all $k\geq 2$, the inequality $2-s_A(D_B^{\text{red}})-2k < -s_A(D_B^{\text{red}})$ holds. Since there are no enhanced states in polynomial grading $j$ where $j<-s_A(D_B^{\text{red}})$, it follows that $\ukh^{2-k,2-s_A(D_B^{\text{red}})}(D_B^{\text{red}})\cong \ukh^{1,-s_A(D_B)}(D_B)$ is trivial. Therefore, the long exact sequence implies that $\ukh^{2,2-s_A(D)}(D)$ is also trivial. \smallskip

\noindent\textbf{Case 3:} Suppose that $\girth(G_A(D))=2$ and $D$ is a diagram as in Figure \ref{figure:tg1case3}. Furthermore, suppose that the two crossings depicted in Figure \ref{figure:tg1case3} are the only two crossings shared by the depicted exterior and interior components of $s_A(D)$. If they were not the only two crossings, then there would be an $A$-twist region and case 2 would suffice to prove the desired result.

Let $D_A$ and $D_B$ be the resolutions of $D$ at the top crossing. If either of the two regions joined by the $A$-resolution are interior faces of $R_{2i}$ then $D_A$ is a Turaev genus one diagram in the format of Figure \ref{figure:case3DA}. In this case, the bottom crossing could be nugatory. If it is, one can remove it to get a reduced, $A$-adequate, Turaev genus one diagram. In either case, the inductive hypothesis implies that $\ukh^{2,2-s_A(D_A)}(D_A)$ is trivial. If both of the regions to the left and to the right of the top crossing are faces meeting the boundary of $R_{2i}$, then $D_A$ is the connected sum of a Turaev genus one diagram in the format of Figure \ref{figure:tg1} with an alternating link. Let $\widetilde{D}_A$ be the diagram where the reverse of the alternating summand is attached (see Figure \ref{figure:case3DA2}). Then $\ukh^{i,j}(D_A)\cong\ukh^{i,j}(\widetilde{D}_A)$ for all $i$ and $j$. Again the bottom crossing could be nugatory. If so, it can be removed to obtain a reduced, $A$-adequate, Turaev genus one diagram. In either case, the inductive hypothesis implies that $\ukh^{2,2-s_A(\widetilde{D}_A)}(\widetilde{D}_A)\cong\ukh^{2,2-s_A(D_A)}(D_A)$ is trivial.

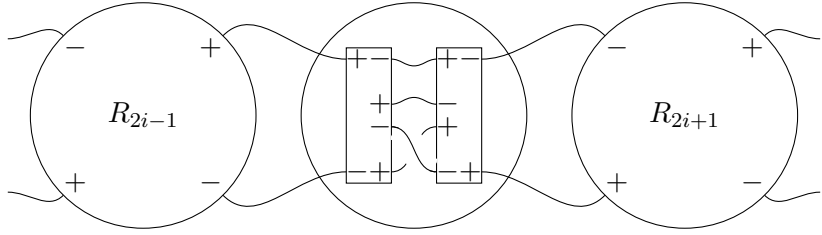
\begin{figure}[ht]
\[\begin{tikzpicture}[scale=.6]

\draw (0,0) rectangle (1,3);
\draw (2,0) rectangle (3,3);

\draw (1,2.75) to [out = 0, in = 180]
(1.5,2.6) to [out = 0, in = 180]
(2,2.75);
\draw (1,1.75) to [out = 0, in = 180]
(1.5,1.9) to [out = 0, in = 180]
(2,1.75);
\begin{knot}[
    consider self intersections,
    clip width = 5,
    ignore endpoint intersections = true,
    end tolerance = 2pt
    ]
    
    \strand (1,1.25) to [out = 0, in =180] (2,0.25);
    \strand (1,0.25) to [out = 0, in = 180] (2,1.25);
    
    \end{knot}
    
    \draw (1.5, 1.5) circle (2.5cm);
      \draw (7.5, 1.5) circle (2.5cm);
     \draw (-4.5, 1.5) circle (2.5cm);
     \draw (0,2.75) to [out = 180 , in = 45] (-2.73,3.27);
     \draw (0,.25) to [out = 180, in = -45] (-2.73, -.27);
     \draw (3,2.75) to [out = 0, in = 135] (5.73,3.27);
     \draw (3,.25) to [out = 0, in = 225] (5.73,-.27);
     \draw (-6.27,3.27) to [out = 135, in = 0] (-7.5,3.2);
     \draw (-6.27,-.27) to [out = 225, in = 0] (-7.5,-.2);
     \draw (9.27,3.27) to [out = 45, in = 180] (10.5,3.2);
     \draw (9.27,-.27) to [out = -45, in = 180] (10.5,-.2);
     
\draw (-4.5,1.5) node{$R_{2i-1}$};
\draw (7.5,1.5) node{$R_{2i+1}$};
\draw (-3,3) node{$+$};
\draw (-3,0) node{$-$};
\draw (-6,3) node{$-$};
\draw (-6,0) node{$+$};
\begin{scope}[xshift = 12cm]
\draw (-3,3) node{$+$};
\draw (-3,0) node{$-$};
\draw (-6,3) node{$-$};
\draw (-6,0) node{$+$};
\end{scope}

\draw (.25,.25) node{$-$};
\draw (.25,2.75) node{$+$};
\draw (.75,.25) node{$+$};
\draw (.75,1.25) node{$-$};
\draw (.75,1.75) node{$+$};
\draw (.75,2.75) node{$-$};

\draw (2.25,.25) node{$-$};
\draw (2.25,1.25) node{$+$};
\draw (2.25,1.75) node{$-$};
\draw (2.25,2.75) node{$+$};
\draw (2.75,.25) node{$+$};
\draw (2.75,2.75) node{$-$};


\end{tikzpicture}\]
\caption{The diagram $D_A$ in case 3 when at least one of the regions joined by the $A$-resolution of the top crossing is an interior face.}
\label{figure:case3DA}
\end{figure}

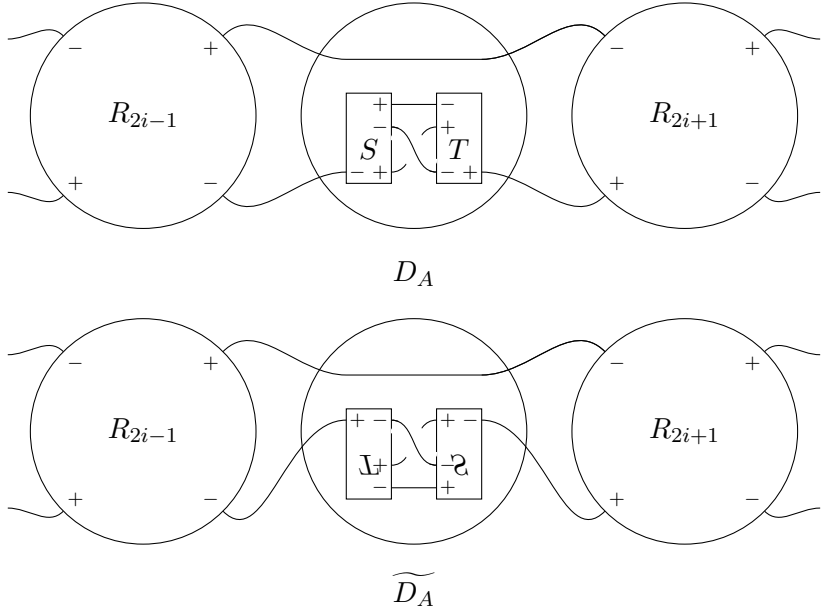
\begin{figure}[ht]
\[\begin{tikzpicture}[scale=.6]

\draw (0,0) rectangle (1,2);
\draw (2,0) rectangle (3,2);

\draw (1,1.75) --(2,1.75);
\begin{knot}[
    consider self intersections,
    clip width = 5,
    ignore endpoint intersections = true,
    end tolerance = 2pt
    ]
    
    \strand (1,1.25) to [out = 0, in =180] (2,0.25);
    \strand (1,0.25) to [out = 0, in = 180] (2,1.25);
    
    \end{knot}
    
    \draw (1.5, 1.5) circle (2.5cm);
      \draw (7.5, 1.5) circle (2.5cm);
     \draw (-4.5, 1.5) circle (2.5cm);
     \draw (5.73,3.27) to [out = 135, in = 0] 
     (3,2.75) to [out = 180, in = 0]
     (0,2.75) to [out = 180 , in = 45] (-2.73,3.27);
     \draw (0,.25) to [out = 180, in = -45] (-2.73, -.27);
     \draw (3,2.75) to [out = 0, in = 135] (5.73,3.27);
     \draw (3,.25) to [out = 0, in = 225] (5.73,-.27);
     \draw (-6.27,3.27) to [out = 135, in = 0] (-7.5,3.2);
     \draw (-6.27,-.27) to [out = 225, in = 0] (-7.5,-.2);
     \draw (9.27,3.27) to [out = 45, in = 180] (10.5,3.2);
     \draw (9.27,-.27) to [out = -45, in = 180] (10.5,-.2);
     
\draw (-4.5,1.5) node{$R_{2i-1}$};
\draw (7.5,1.5) node{$R_{2i+1}$};
\draw (-3,3) node{\tiny{$+$}};
\draw (-3,0) node{\tiny{$-$}};
\draw (-6,3) node{\tiny{$-$}};
\draw (-6,0) node{\tiny{$+$}};
\begin{scope}[xshift = 12cm]
\draw (-3,3) node{\tiny{$+$}};
\draw (-3,0) node{\tiny{$-$}};
\draw (-6,3) node{\tiny{$-$}};
\draw (-6,0) node{\tiny{$+$}};
\end{scope}

\draw (.25,.25) node{\tiny{$-$}};
\draw (.75,.25) node{\tiny{$+$}};
\draw (.75,1.25) node{\tiny{$-$}};
\draw (.75,1.75) node{\tiny{$+$}};

\draw (2.25,.25) node{\tiny{$-$}};
\draw (2.25,1.25) node{\tiny{$+$}};
\draw (2.25,1.75) node{\tiny{$-$}};
\draw (2.75,.25) node{\tiny{$+$}};

\draw (.5,.75)  node{$S$};
\draw (2.5,.75) node{$T$};
\draw (1.5,-2) node{$D_A$};


\begin{scope}[yshift = - 7cm]
\draw (0,0) rectangle (1,2);
\draw (2,0) rectangle (3,2);

\draw (1,.25) --(2,.25);
\begin{knot}[
    consider self intersections,
    clip width = 5,
    ignore endpoint intersections = true,
    end tolerance = 2pt
    ]
    
    \strand (1,1.75) to [out = 0, in =180] (2,0.75);
    \strand (1,0.75) to [out = 0, in = 180] (2,1.75);
    
    \end{knot}
    
    \draw (1.5, 1.5) circle (2.5cm);
      \draw (7.5, 1.5) circle (2.5cm);
     \draw (-4.5, 1.5) circle (2.5cm);
     \draw (5.73,3.27) to [out = 135, in = 0] 
     (3,2.75) to [out = 180, in = 0]
     (0,2.75) to [out = 180 , in = 45] (-2.73,3.27);
     \draw (0,1.75) to [out = 180, in = -45] (-2.73, -.27);
     \draw (3,2.75) to [out = 0, in = 135] (5.73,3.27);
     \draw (3,1.75) to [out = 0, in = 225] (5.73,-.27);
     \draw (-6.27,3.27) to [out = 135, in = 0] (-7.5,3.2);
     \draw (-6.27,-.27) to [out = 225, in = 0] (-7.5,-.2);
     \draw (9.27,3.27) to [out = 45, in = 180] (10.5,3.2);
     \draw (9.27,-.27) to [out = -45, in = 180] (10.5,-.2);
     
\draw (-4.5,1.5) node{$R_{2i-1}$};
\draw (7.5,1.5) node{$R_{2i+1}$};
\draw (-3,3) node{\tiny{$+$}};
\draw (-3,0) node{\tiny{$-$}};
\draw (-6,3) node{\tiny{$-$}};
\draw (-6,0) node{\tiny{$+$}};
\begin{scope}[xshift = 12cm]
\draw (-3,3) node{\tiny{$+$}};
\draw (-3,0) node{\tiny{$-$}};
\draw (-6,3) node{\tiny{$-$}};
\draw (-6,0) node{\tiny{$+$}};
\end{scope}

\draw (.25,1.75) node{\tiny{$+$}};
\draw (.75,.25) node{\tiny{$-$}};
\draw (.75,.75) node{\tiny{$+$}};
\draw (.75,1.75) node{\tiny{$-$}};

\draw (2.25,.25) node{\tiny{$+$}};
\draw (2.25,.75) node{\tiny{$-$}};
\draw (2.25,1.75) node{\tiny{$+$}};
\draw (2.75,1.75) node{\tiny{$-$}};

\draw (2.5,.75)  node[yscale=-1]{$S$};
\draw (.5,.75) node[yscale=-1]{$T$};

\draw (1.5,-2) node{$\widetilde{D_A}$};

\end{scope}

\end{tikzpicture}\]
\caption{The diagrams $D_A$ and $\widetilde{D}_A$ in case 3 when both of the regions joined by the $A$-resolution of the top crossing are exterior faces. The diagram $\widetilde{D}_A$ has the same format as the standard Turaev genus one diagram in Figure \ref{figure:tg1}.}
\label{figure:case3DA2}
\end{figure}

The diagram $D_B$ is not $A$-adequate. Let $D_{BA}$ and $D_{BB}$ be the resolutions of $D_B$ at the bottom crossing, as in Figure \ref{figure:case3DB}. Since $D_{BA}$ is $A$-adequate, it follows that $\ukh^{1,-s_A(D_{BA})}(D_{BA})$ is trivial. The diagram $D_{BB}$ is a connected sum of alternating diagrams. Therefore $\ukh^{0,-s_A(D_{BB})}(D_{BB})\cong\mathbb{Z}$ and is generated by the all-$A$ Kauffman state $S_{BB}^A$ with $-$ labels on all components. The long exact sequence for $D_{BB}$, $D_B$, and $D_{BA}$ is given by
\[ \cdots\to \ukh^{0,-s_A(D_{BB})}(D_{BB})\xrightarrow{f^*} \ukh^{1,1-s_A(D_B)}(D_B) 
\to  \ukh^{1,-s_A(D_{BA})}(D_{BA})\to\cdots.\]

\begin{figure}[ht]
\[\begin{tikzpicture}[scale=.6]

\draw (0,0) rectangle (1,3);
\draw (2,0) rectangle (3,3);

\draw (1,2.75) to [out = 0, in = 0, looseness = 1.3] (1,1.75);
\draw (2,2.75) to [out = 180, in = 180, looseness = 1.3] (2,1.75);
\draw (1.5,-1.5) node{$D_B$};

\begin{knot}[
    consider self intersections,
    clip width = 5,
    ignore endpoint intersections = true,
    end tolerance = 2pt
    ]
    
    \strand (1,1.25) to [out = 0, in =180] (2,0.25);
    \strand (1,0.25) to [out = 0, in = 180] (2,1.25);
    
    \end{knot}
    
    \draw (1.5, 1.5) circle (2.5cm);
      \draw (7.5, 1.5) circle (2.5cm);
     \draw (-4.5, 1.5) circle (2.5cm);
     \draw (0,2.75) to [out = 180 , in = 45] (-2.73,3.27);
     \draw (0,.25) to [out = 180, in = -45] (-2.73, -.27);
     \draw (3,2.75) to [out = 0, in = 135] (5.73,3.27);
     \draw (3,.25) to [out = 0, in = 225] (5.73,-.27);
     \draw (-6.27,3.27) to [out = 135, in = 0] (-7.5,3.2);
     \draw (-6.27,-.27) to [out = 225, in = 0] (-7.5,-.2);
     \draw (9.27,3.27) to [out = 45, in = 180] (10.5,3.2);
     \draw (9.27,-.27) to [out = -45, in = 180] (10.5,-.2);
     
\draw (-4.5,1.5) node{$R_{2i-1}$};
\draw (7.5,1.5) node{$R_{2i+1}$};
\draw (-3,3) node{\tiny{$+$}};
\draw (-3,0) node{\tiny{$-$}};
\draw (-6,3) node{\tiny{$-$}};
\draw (-6,0) node{\tiny{$+$}};
\begin{scope}[xshift = 12cm]
\draw (-3,3) node{\tiny{$+$}};
\draw (-3,0) node{\tiny{$-$}};
\draw (-6,3) node{\tiny{$-$}};
\draw (-6,0) node{\tiny{$+$}};
\end{scope}

\draw (.25,.25) node{\tiny{$-$}};
\draw (.25,2.75) node{\tiny{$+$}};
\draw (.75,.25) node{\tiny{$+$}};
\draw (.75,1.25) node{\tiny{$-$}};
\draw (.75,1.75) node{\tiny{$+$}};
\draw (.75,2.75) node{\tiny{$-$}};

\draw (2.25,.25) node{\tiny{$-$}};
\draw (2.25,1.25) node{\tiny{$+$}};
\draw (2.25,1.75) node{\tiny{$-$}};
\draw (2.25,2.75) node{\tiny{$+$}};
\draw (2.75,.25) node{\tiny{$+$}};
\draw (2.75,2.75) node{\tiny{$-$}};


\begin{scope}[yshift=-7cm]
\draw (0,0) rectangle (1,3);
\draw (2,0) rectangle (3,3);

\draw (1,2.75) to [out = 0, in = 0, looseness = 1.3] (1,1.75);
\draw (2,2.75) to [out = 180, in = 180, looseness = 1.3] (2,1.75);

\draw (1,1.25) to [out = 0, in = 180] (1.5,1.1) to [out =0, in = 180] (2,1.25);
\draw (1,.25) to [out = 0, in = 180] (1.5,.4) to [out =0, in = 180] (2,.25);

\draw (1.5,-1.5) node{$D_{BA}$};

\begin{knot}[
    consider self intersections,
    clip width = 5,
    ignore endpoint intersections = true,
    end tolerance = 2pt
    ]
    
    
    \end{knot}
    
    \draw (1.5, 1.5) circle (2.5cm);
      \draw (7.5, 1.5) circle (2.5cm);
     \draw (-4.5, 1.5) circle (2.5cm);
     \draw (0,2.75) to [out = 180 , in = 45] (-2.73,3.27);
     \draw (0,.25) to [out = 180, in = -45] (-2.73, -.27);
     \draw (3,2.75) to [out = 0, in = 135] (5.73,3.27);
     \draw (3,.25) to [out = 0, in = 225] (5.73,-.27);
     \draw (-6.27,3.27) to [out = 135, in = 0] (-7.5,3.2);
     \draw (-6.27,-.27) to [out = 225, in = 0] (-7.5,-.2);
     \draw (9.27,3.27) to [out = 45, in = 180] (10.5,3.2);
     \draw (9.27,-.27) to [out = -45, in = 180] (10.5,-.2);
     
\draw (-4.5,1.5) node{$R_{2i-1}$};
\draw (7.5,1.5) node{$R_{2i+1}$};
\draw (-3,3) node{\tiny{$+$}};
\draw (-3,0) node{\tiny{$-$}};
\draw (-6,3) node{\tiny{$-$}};
\draw (-6,0) node{\tiny{$+$}};
\begin{scope}[xshift = 12cm]
\draw (-3,3) node{\tiny{$+$}};
\draw (-3,0) node{\tiny{$-$}};
\draw (-6,3) node{\tiny{$-$}};
\draw (-6,0) node{\tiny{$+$}};
\end{scope}

\draw (.25,.25) node{\tiny{$-$}};
\draw (.25,2.75) node{\tiny{$+$}};
\draw (.75,.25) node{\tiny{$+$}};
\draw (.75,1.25) node{\tiny{$-$}};
\draw (.75,1.75) node{\tiny{$+$}};
\draw (.75,2.75) node{\tiny{$-$}};

\draw (2.25,.25) node{\tiny{$-$}};
\draw (2.25,1.25) node{\tiny{$+$}};
\draw (2.25,1.75) node{\tiny{$-$}};
\draw (2.25,2.75) node{\tiny{$+$}};
\draw (2.75,.25) node{\tiny{$+$}};
\draw (2.75,2.75) node{\tiny{$-$}};

\end{scope}

\begin{scope}[yshift=-14cm]
\draw (0,0) rectangle (1,3);
\draw (2,0) rectangle (3,3);

\draw (1,2.75) to [out = 0, in = 0, looseness = 1.3] (1,1.75);
\draw (2,2.75) to [out = 180, in = 180, looseness = 1.3] (2,1.75);

\draw (1,1.25) to [out = 0, in = 0, looseness = 1.3] (1,.25);
\draw (2,1.25) to [out = 180, in = 180, looseness = 1.3] (2,.25);

\draw (1.5,-1.5) node{$D_{BB}$};

\begin{knot}[
    consider self intersections,
    clip width = 5,
    ignore endpoint intersections = true,
    end tolerance = 2pt
    ]
    
    
    \end{knot}
    
    \draw (1.5, 1.5) circle (2.5cm);
      \draw (7.5, 1.5) circle (2.5cm);
     \draw (-4.5, 1.5) circle (2.5cm);
     \draw (0,2.75) to [out = 180 , in = 45] (-2.73,3.27);
     \draw (0,.25) to [out = 180, in = -45] (-2.73, -.27);
     \draw (3,2.75) to [out = 0, in = 135] (5.73,3.27);
     \draw (3,.25) to [out = 0, in = 225] (5.73,-.27);
     \draw (-6.27,3.27) to [out = 135, in = 0] (-7.5,3.2);
     \draw (-6.27,-.27) to [out = 225, in = 0] (-7.5,-.2);
     \draw (9.27,3.27) to [out = 45, in = 180] (10.5,3.2);
     \draw (9.27,-.27) to [out = -45, in = 180] (10.5,-.2);
     
\draw (-4.5,1.5) node{$R_{2i-1}$};
\draw (7.5,1.5) node{$R_{2i+1}$};
\draw (-3,3) node{\tiny{$+$}};
\draw (-3,0) node{\tiny{$-$}};
\draw (-6,3) node{\tiny{$-$}};
\draw (-6,0) node{\tiny{$+$}};
\begin{scope}[xshift = 12cm]
\draw (-3,3) node{\tiny{$+$}};
\draw (-3,0) node{\tiny{$-$}};
\draw (-6,3) node{\tiny{$-$}};
\draw (-6,0) node{\tiny{$+$}};
\end{scope}

\draw (.25,.25) node{\tiny{$-$}};
\draw (.25,2.75) node{\tiny{$+$}};
\draw (.75,.25) node{\tiny{$+$}};
\draw (.75,1.25) node{\tiny{$-$}};
\draw (.75,1.75) node{\tiny{$+$}};
\draw (.75,2.75) node{\tiny{$-$}};

\draw (2.25,.25) node{\tiny{$-$}};
\draw (2.25,1.25) node{\tiny{$+$}};
\draw (2.25,1.75) node{\tiny{$-$}};
\draw (2.25,2.75) node{\tiny{$+$}};
\draw (2.75,.25) node{\tiny{$+$}};
\draw (2.75,2.75) node{\tiny{$-$}};

\end{scope}

\end{tikzpicture}\]

\caption{The diagrams $D_B$, $D_{BA}$, and $D_{BB}$ from case 3.}
\label{figure:case3DB}
\end{figure}
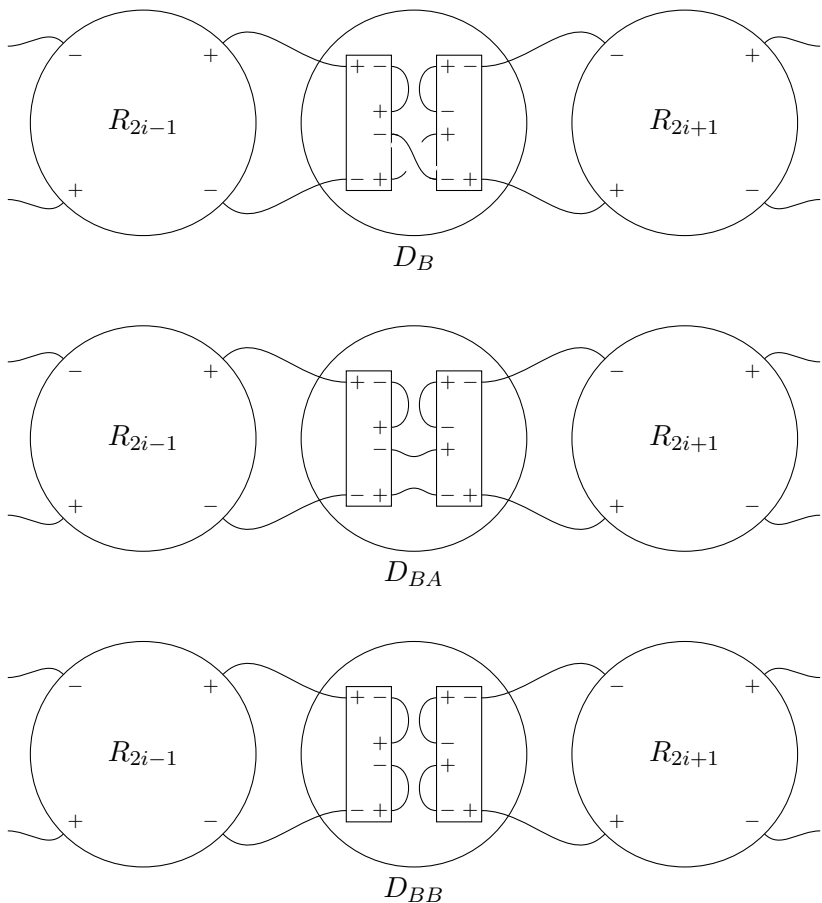

Because $\ukh^{1,-s_A(D_{BA})}(D_{BA})$ is trivial, if the image of $f^*$ is trivial, then $\ukh^{1,-s_A(D_B)}(D_B)$ is also trivial. Since $\ukh^{0,-s_A(D_{BB})}(D_{BB})\cong \mathbb{Z}$, it suffices to check that $f^*[S_{BB}^A]=[f(S_{BB}^A)]$ is trivial. The enhanced state $f(S_{BB}^A)$ is the state of $D_B$ where the bottom crossing is an $B$-resolution, all other crossings are $A$-resolutions, and every component is marked with a $-$. Let $S_{B}^A$ be the enhanced state of $D_B$ where every resolution is an $A$-resolution and every component is marked with a $-$. There is exactly one crossing (the bottom crossing) that when changed from $A$ to $B$ results in more components. Therefore $d(S_{B}^A)=\pm S_{B}^A$, and thus $f(S_{BB}^A)$ is in the image of the Khovanov differential. Therefore $f^*$ is the zero map, and thus $\ukh^{1,-s_A(D_B)}(D_B)$ is trivial.

Recall that the long exact sequence for $D_B$, $D_A$, and $D$ is 
\[\cdots \ukh^{1,-s_A(D_B)}(D_B)\to\ukh^{2,2-s_A(D)}(D)\to\ukh^{2,2-s_A(D_A)}(D_A)\to\cdots,\]
and that both $\ukh^{1,-s_A(D_B)}(D_B)$ and $\ukh^{2,2-s_A(D_A)}(D_A)$ are trivial. Therefore $\ukh^{2,2-s_A(D)}(D)$ is also trivial.\smallskip

\noindent\textbf{Case 4:} Suppose that $\girth(G_A(D))=2$, $D$ contains exactly four alternating tangles, and both $R_1$ and $R_3$ have a single crossing in the boundary of their respective eastern and western faces. Figure \ref{figure:case4tg1} depicts the diagram $D$. If either of the regions joined by the $A$-resolution are completely contained inside $R_1$, then $D_A$ is  a reduced, $A$-adequate, Turaev genus one diagram, and $\ukh^{2,2-s_A(D_A)}(D_A)$ is trivial by induction. If both of the regions joined by the $A$-resolution intersect the boundary of $R_1$, then $D_A$ is the connected sum of alternating diagrams, and thus $\ukh^{2,2-s_A(D_A)}(D_A)$ is trivial.
\begin{figure}[ht]
\[\begin{tikzpicture}[scale = .6]

\draw (0,0) rectangle (1.5,1.5);
\draw (0,2.5) rectangle (1.5,4);

\draw (5,2) circle (2cm);

\draw (8.5,0) rectangle (10,1.5);
\draw (8.5,2.5) rectangle (10,4);

\draw (13.5,2) circle (2cm);

\draw (.75,.75) node{$R_1^2$};
\draw (.75,3.25) node{$R_1^1$};

\draw (.25,.25) node{\tiny{$+$}};
\draw (1.25,.25) node{\tiny{$-$}};
\draw (.25,1.25) node{\tiny{$-$}};
\draw (1.25,1.25) node{\tiny{$+$}};

\draw (.25,2.75) node{\tiny{$+$}};
\draw (1.25,2.75) node{\tiny{$-$}};
\draw (.25,3.75) node{\tiny{$-$}};
\draw (1.25,3.75) node{\tiny{$+$}};

\begin{scope}[xshift = 8.5cm]
\draw (.25,.25) node{\tiny{$+$}};
\draw (1.25,.25) node{\tiny{$-$}};
\draw (.25,1.25) node{\tiny{$-$}};
\draw (1.25,1.25) node{\tiny{$+$}};

\draw (.25,2.75) node{\tiny{$+$}};
\draw (1.25,2.75) node{\tiny{$-$}};
\draw (.25,3.75) node{\tiny{$-$}};
\draw (1.25,3.75) node{\tiny{$+$}};

\draw (3.8,3.15) node{$+$};
\draw (3.8, .85)node{$-$};
\draw (6.2,3.15) node{$-$};
\draw (6.2,.85) node{$+$};
\end{scope}

\draw (5,2) node{$R_2$};

\draw (3.8,3.15) node{$+$};
\draw (3.8, .85)node{$-$};
\draw (6.2,3.15) node{$-$};
\draw (6.2,.85) node{$+$};

\draw (3.586,3.414) to [out = 150, in = 0] (1.5,3.75);
\draw (3.586,.586) to [out = 210, in = 0] (1.5,.25);
\draw (6.414,3.414) to [out = 30, in = 180] (8.5,3.75);
\draw (6.414,.586) to [out = -30, in = 180] (8.5,.25);

\begin{scope}[xshift = 8.5cm]
\draw (3.586,3.414) to [out = 150, in = 0] (1.5,3.75);
\draw (3.586,.586) to [out = 210, in = 0] (1.5,.25);
\end{scope}

\draw (9.25,.75) node{$R_3^2$};
\draw (9.25,3.25) node{$R_3^1$};
\draw (13.5,2) node{$R_4$};

\draw (0,.25) arc (90:270:.5cm);
\draw (0,-.75) -- (14.7,-.75);

\draw (0,3.75) arc (270:90:.5cm);
\draw (0,4.75) -- (14.7,4.75);

\draw (14.914, .586) to [out = -45, in = 0, looseness = 1.5] (14.7,-.75);
\draw (14.914,3.414) to [out = 45, in = 0, looseness=1.5] (14.7,4.75);

\begin{knot}[
    consider self intersections,
    clip width = 5,
    ignore endpoint intersections = true,
    end tolerance = 2pt
    ]
    
        \strand (8.75,1.5) to [out = 90, in = 270] (9.75,2.5);
    \strand (8.75,2.5) to [out = 270, in = 90] (9.75,1.5);

    \end{knot}

    \draw (.25,2.5) to [out = 270, in = 90] (.4,2) to [out = 270, in = 90] (.25,1.5);
    \draw (1.25,2.5) to [out = 270, in = 90] (1.1,2) to [out = 270, in = 90] (1.25,1.5);
    \draw (7.5, -1.25) node{$D_A$};

\begin{scope}[yshift=-7cm]
\draw (0,0) rectangle (1.5,1.5);
\draw (0,2.5) rectangle (1.5,4);

\draw (5,2) circle (2cm);

\draw (8.5,0) rectangle (10,1.5);
\draw (8.5,2.5) rectangle (10,4);

\draw (13.5,2) circle (2cm);

\draw (.75,.75) node{$R_1^2$};
\draw (.75,3.25) node{$R_1^1$};

\draw (.25,.25) node{\tiny{$+$}};
\draw (1.25,.25) node{\tiny{$-$}};
\draw (.25,1.25) node{\tiny{$-$}};
\draw (1.25,1.25) node{\tiny{$+$}};

\draw (.25,2.75) node{\tiny{$+$}};
\draw (1.25,2.75) node{\tiny{$-$}};
\draw (.25,3.75) node{\tiny{$-$}};
\draw (1.25,3.75) node{\tiny{$+$}};

\begin{scope}[xshift = 8.5cm]
\draw (.25,.25) node{\tiny{$+$}};
\draw (1.25,.25) node{\tiny{$-$}};
\draw (.25,1.25) node{\tiny{$-$}};
\draw (1.25,1.25) node{\tiny{$+$}};

\draw (.25,2.75) node{\tiny{$+$}};
\draw (1.25,2.75) node{\tiny{$-$}};
\draw (.25,3.75) node{\tiny{$-$}};
\draw (1.25,3.75) node{\tiny{$+$}};

\draw (3.8,3.15) node{$+$};
\draw (3.8, .85)node{$-$};
\draw (6.2,3.15) node{$-$};
\draw (6.2,.85) node{$+$};
\end{scope}

\draw (5,2) node{$R_2$};

\draw (3.8,3.15) node{$+$};
\draw (3.8, .85)node{$-$};
\draw (6.2,3.15) node{$-$};
\draw (6.2,.85) node{$+$};

\draw (3.586,3.414) to [out = 150, in = 0] (1.5,3.75);
\draw (3.586,.586) to [out = 210, in = 0] (1.5,.25);
\draw (6.414,3.414) to [out = 30, in = 180] (8.5,3.75);
\draw (6.414,.586) to [out = -30, in = 180] (8.5,.25);

\begin{scope}[xshift = 8.5cm]
\draw (3.586,3.414) to [out = 150, in = 0] (1.5,3.75);
\draw (3.586,.586) to [out = 210, in = 0] (1.5,.25);
\end{scope}

\draw (9.25,.75) node{$R_3^2$};
\draw (9.25,3.25) node{$R_3^1$};
\draw (13.5,2) node{$R_4$};

\draw (0,.25) arc (90:270:.5cm);
\draw (0,-.75) -- (14.7,-.75);

\draw (0,3.75) arc (270:90:.5cm);
\draw (0,4.75) -- (14.7,4.75);

\draw (14.914, .586) to [out = -45, in = 0, looseness = 1.5] (14.7,-.75);
\draw (14.914,3.414) to [out = 45, in = 0, looseness=1.5] (14.7,4.75);

\begin{knot}[
    consider self intersections,
    clip width = 5,
    ignore endpoint intersections = true,
    end tolerance = 2pt
    ]
    
        \strand (8.75,1.5) to [out = 90, in = 270] (9.75,2.5);
    \strand (8.75,2.5) to [out = 270, in = 90] (9.75,1.5);

    \end{knot}

    \draw (.25,2.5) to [out = 270, in = 270, looseness=1.3] (1.25,2.5);
    \draw (.25,1.5) to [out = 90, in = 90, looseness=1.3] (1.25,1.5);
    \draw (7.5, -1.25) node{$D_B$};
\end{scope}    

\end{tikzpicture}\]
    \caption{Diagrams of $D_A$ and $D_B$ in case $4$.}
    \label{figure:case4tg1}
\end{figure}
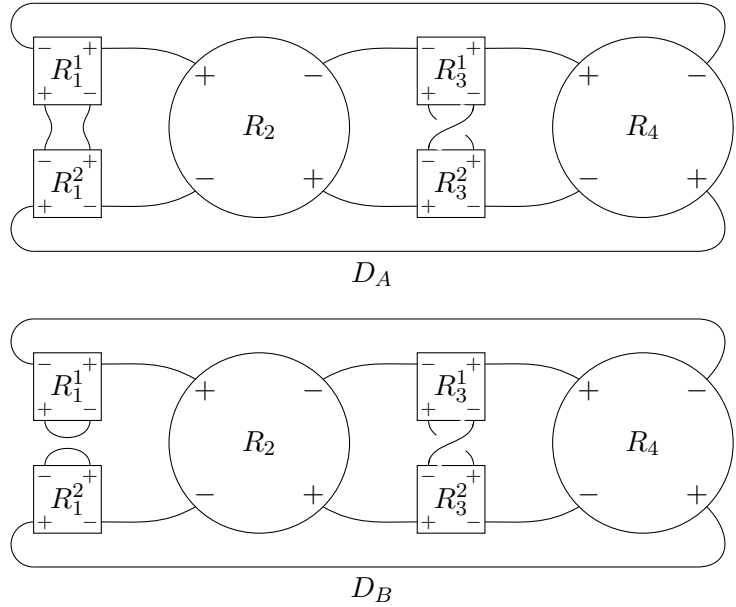

Let $D_{BA}$ and $D_{BB}$ be the diagrams obtained by resolving the pictured crossing of $D$ in $R_3$, as in Figure \ref{figure:case4tg12}. One can repeat the argument in case 3 verbatim to conclude that $\ukh^{1,-s_A(D_B)}(D_B)$ is trivial. The long exact sequence for $D_B$, $D_A$, and $D$ is 
\[\cdots \ukh^{1,-s_A(D_B)}(D_B)\to\ukh^{2,2-s_A(D)}(D)\to\ukh^{2,2-s_A(D_A)}(D_A)\to\cdots.\]
Since both $\ukh^{1,-s_A(D_B)}(D_B)$ and $\ukh^{2,2-s_A(D_A)}(D_A)$ are trivial, it follows that $\ukh^{2,2-s_A(D)}(D)$ is also trivial.

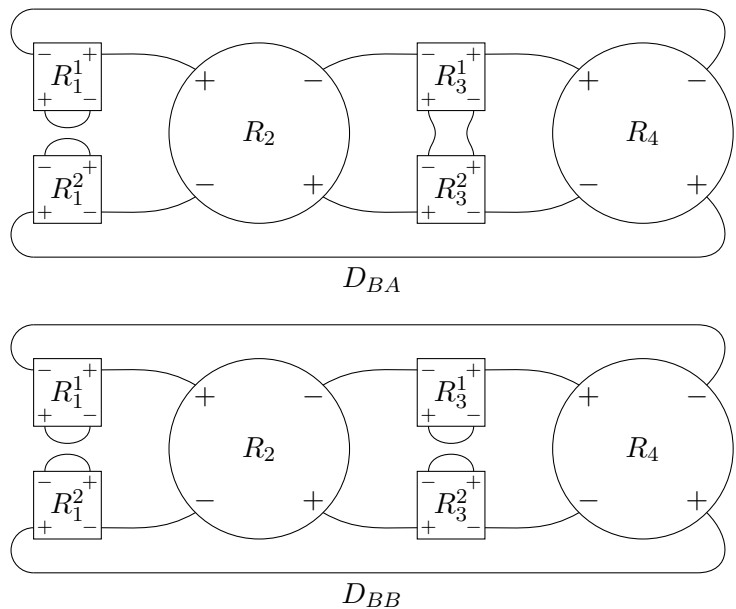
\begin{figure}[ht]
\[\begin{tikzpicture}[scale=.6]
\draw (0,0) rectangle (1.5,1.5);
\draw (0,2.5) rectangle (1.5,4);

\draw (5,2) circle (2cm);

\draw (8.5,0) rectangle (10,1.5);
\draw (8.5,2.5) rectangle (10,4);

\draw (13.5,2) circle (2cm);

\draw (.75,.75) node{$R_1^2$};
\draw (.75,3.25) node{$R_1^1$};

\draw (.25,.25) node{\tiny{$+$}};
\draw (1.25,.25) node{\tiny{$-$}};
\draw (.25,1.25) node{\tiny{$-$}};
\draw (1.25,1.25) node{\tiny{$+$}};

\draw (.25,2.75) node{\tiny{$+$}};
\draw (1.25,2.75) node{\tiny{$-$}};
\draw (.25,3.75) node{\tiny{$-$}};
\draw (1.25,3.75) node{\tiny{$+$}};

\begin{scope}[xshift = 8.5cm]
\draw (.25,.25) node{\tiny{$+$}};
\draw (1.25,.25) node{\tiny{$-$}};
\draw (.25,1.25) node{\tiny{$-$}};
\draw (1.25,1.25) node{\tiny{$+$}};

\draw (.25,2.75) node{\tiny{$+$}};
\draw (1.25,2.75) node{\tiny{$-$}};
\draw (.25,3.75) node{\tiny{$-$}};
\draw (1.25,3.75) node{\tiny{$+$}};

\draw (3.8,3.15) node{$+$};
\draw (3.8, .85)node{$-$};
\draw (6.2,3.15) node{$-$};
\draw (6.2,.85) node{$+$};
\end{scope}

\draw (5,2) node{$R_2$};

\draw (3.8,3.15) node{$+$};
\draw (3.8, .85)node{$-$};
\draw (6.2,3.15) node{$-$};
\draw (6.2,.85) node{$+$};

\draw (3.586,3.414) to [out = 150, in = 0] (1.5,3.75);
\draw (3.586,.586) to [out = 210, in = 0] (1.5,.25);
\draw (6.414,3.414) to [out = 30, in = 180] (8.5,3.75);
\draw (6.414,.586) to [out = -30, in = 180] (8.5,.25);

\begin{scope}[xshift = 8.5cm]
\draw (3.586,3.414) to [out = 150, in = 0] (1.5,3.75);
\draw (3.586,.586) to [out = 210, in = 0] (1.5,.25);
\end{scope}

\draw (9.25,.75) node{$R_3^2$};
\draw (9.25,3.25) node{$R_3^1$};
\draw (13.5,2) node{$R_4$};

\draw (0,.25) arc (90:270:.5cm);
\draw (0,-.75) -- (14.7,-.75);

\draw (0,3.75) arc (270:90:.5cm);
\draw (0,4.75) -- (14.7,4.75);

\draw (14.914, .586) to [out = -45, in = 0, looseness = 1.5] (14.7,-.75);
\draw (14.914,3.414) to [out = 45, in = 0, looseness=1.5] (14.7,4.75);

    \draw (8.75,2.5) to [out = 270, in =90] (8.9,2) to [out=270, in=90] (8.75,1.5);
    \draw (9.75,2.5) to [out = 270, in = 90] (9.6,2) to [out = 270, in=90] (9.75,1.5);

    \draw (.25,2.5) to [out = 270, in = 270, looseness=1.3] (1.25,2.5);
    \draw (.25,1.5) to [out = 90, in = 90, looseness=1.3] (1.25,1.5);
    \draw (7.5, -1.25) node{$D_{BA}$};

\begin{scope}[yshift = -7cm]
\draw (0,0) rectangle (1.5,1.5);
\draw (0,2.5) rectangle (1.5,4);

\draw (5,2) circle (2cm);

\draw (8.5,0) rectangle (10,1.5);
\draw (8.5,2.5) rectangle (10,4);

\draw (13.5,2) circle (2cm);

\draw (.75,.75) node{$R_1^2$};
\draw (.75,3.25) node{$R_1^1$};

\draw (.25,.25) node{\tiny{$+$}};
\draw (1.25,.25) node{\tiny{$-$}};
\draw (.25,1.25) node{\tiny{$-$}};
\draw (1.25,1.25) node{\tiny{$+$}};

\draw (.25,2.75) node{\tiny{$+$}};
\draw (1.25,2.75) node{\tiny{$-$}};
\draw (.25,3.75) node{\tiny{$-$}};
\draw (1.25,3.75) node{\tiny{$+$}};

\begin{scope}[xshift = 8.5cm]
\draw (.25,.25) node{\tiny{$+$}};
\draw (1.25,.25) node{\tiny{$-$}};
\draw (.25,1.25) node{\tiny{$-$}};
\draw (1.25,1.25) node{\tiny{$+$}};

\draw (.25,2.75) node{\tiny{$+$}};
\draw (1.25,2.75) node{\tiny{$-$}};
\draw (.25,3.75) node{\tiny{$-$}};
\draw (1.25,3.75) node{\tiny{$+$}};

\draw (3.8,3.15) node{$+$};
\draw (3.8, .85)node{$-$};
\draw (6.2,3.15) node{$-$};
\draw (6.2,.85) node{$+$};
\end{scope}

\draw (5,2) node{$R_2$};

\draw (3.8,3.15) node{$+$};
\draw (3.8, .85)node{$-$};
\draw (6.2,3.15) node{$-$};
\draw (6.2,.85) node{$+$};

\draw (3.586,3.414) to [out = 150, in = 0] (1.5,3.75);
\draw (3.586,.586) to [out = 210, in = 0] (1.5,.25);
\draw (6.414,3.414) to [out = 30, in = 180] (8.5,3.75);
\draw (6.414,.586) to [out = -30, in = 180] (8.5,.25);

\begin{scope}[xshift = 8.5cm]
\draw (3.586,3.414) to [out = 150, in = 0] (1.5,3.75);
\draw (3.586,.586) to [out = 210, in = 0] (1.5,.25);
\end{scope}

\draw (9.25,.75) node{$R_3^2$};
\draw (9.25,3.25) node{$R_3^1$};
\draw (13.5,2) node{$R_4$};

\draw (0,.25) arc (90:270:.5cm);
\draw (0,-.75) -- (14.7,-.75);

\draw (0,3.75) arc (270:90:.5cm);
\draw (0,4.75) -- (14.7,4.75);

\draw (14.914, .586) to [out = -45, in = 0, looseness = 1.5] (14.7,-.75);
\draw (14.914,3.414) to [out = 45, in = 0, looseness=1.5] (14.7,4.75);

    \draw (8.75,2.5) to [out = 270, in = 270, looseness=1.3] (9.75,2.5);
    \draw (8.75,1.5) to [out = 90, in = 90, looseness=1.3] (9.75,1.5);

    \draw (.25,2.5) to [out = 270, in = 270, looseness=1.3] (1.25,2.5);
    \draw (.25,1.5) to [out = 90, in = 90, looseness=1.3] (1.25,1.5);
    \draw (7.5, -1.25) node{$D_{BB}$};
\end{scope}

\end{tikzpicture}\]
    \caption{Diagrams of $D_{BA}$ and $D_{BB}$ in case $4$.}
    \label{figure:case4tg12}
\end{figure}

\end{proof}

Theorem \ref{theorem:almostalternating} and \ref{theorem:tg1} combine to prove Theorem \ref{theorem:main}.
\begin{proof}[Proof of Theorem \ref{theorem:main}]
Let $D$ be an $A$-Turaev genus one diagram of $L$. Then $D$ is either $A$-almost alternating or $D$ is $A$-adequate and has Turaev genus one. In the former case, Theorem \ref{theorem:almostalternating} implies the result, and in the latter case Theorem \ref{theorem:tg1} implies the result. 

If $D$ is a $B$-Turaev genus one diagram of $L$, then the mirror $\overline{D}$ is $A$-Turaev genus one, and Equation \ref{equation:Mirror} implies the result.
\end{proof}

\begin{example}
Let $K$ be the $14$-crossing knot depicted in Figure \ref{figure:14ex}. The Khovanov homology of $K$ is computed in Table \ref{table:14ex}. The Khovanov homology of $K$ in its maximum polynomial grading is $Kh^{*,j_{\max}(K)}(K) \cong Kh^{8,19}(K)\cong \mathbb{Z}^2$. The Khovanov homology of $K$ in its minimum polynomial grading is $Kh^{*,j_{\min}(K)}(K) \cong Kh^{-2,-1}(K) \cong \mathbb{Z}$; also $Kh^{i_{\min}(K)+2,j_{\min}(K)+2}(K) \cong Kh^{0,1}(K) \cong \mathbb{Z}$, which is nontrivial. Therefore Theorem \ref{theorem:main} implies that $K$ has Turaev genus at least two. No previous result would obstruct this knot from having Turaev genus one.
\end{example}
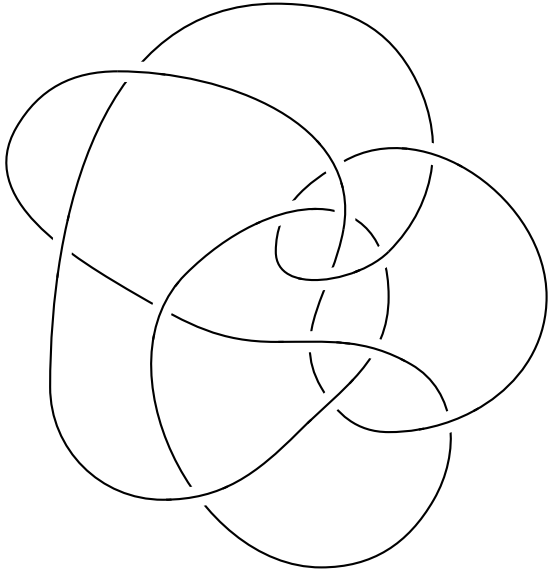
\begin{figure}[h]
\[\begin{tikzpicture}[scale = .3]
\begin{knot}[
    consider self intersections,
    clip width = 10,
    ignore endpoint intersections = true,
    end tolerance = 2pt
    ]
\flipcrossings{1,4,5,9,10,12}
\strand[thick] (16,1) to [out = 0, in = 240]
(21,4) to [out = 60, in = -30]
(20,10) to [out = 150, in = 0]
(14,11) to [out = 180, in = -30]
(8,13) to [out = 150, in = 240]
(2.5,20.5) to [out = 60, in = 180]
(7,23) to [out = 0, in =80]
(17,16) to [out = 260, in = 120]
(16,9) to [out = 300, in = 180]
(19,7) to [out = 0, in = 270]
(26,13) to [out = 90, in = 30]
(17,19) to [out = 210, in = 90]
(14,15) to [out = 270, in = 225]
(19,15) to [out = 45, in = -60]
(20,23) to [out = 120, in = 0]
(14,26) to [out = 180, in = 90]
(4,9) to [out = 270, in = 180]
(9,4) to [out = 0, in = 225]
(15,7) to [out = 45, in = 270]
(19,13) to [out = 90, in = 45, looseness=1.5]
(10,14) to [out = 225, in=180]
(16,1);
\end{knot}

\end{tikzpicture}\]
\caption{Theorem \ref{theorem:main} implies this knot has Turaev genus at least two.}
\label{figure:14ex}
\end{figure}

\begin{table}[h]
\begin{tabular}{| r || c | c | c | c | c | c | c | c | c | c | c |}
\hline
$j \backslash i$& -2 & -1 & 0 & 1 & 2 & 3 & 4 & 5 & 6 & 7 & 8\\
\hline
\hline
19 & & & & & & & & & & & 2\\
\hline
17  & & & & & & & & & & 1 & $2_2$ \\
\hline
15  & & & & & & & & & 2 & $2, 1_2$& \\
\hline
13  & & & & & & & 1 & 3 & $1,2_2$& & \\
\hline
11  & & & & & & & $1,1_2$ & $2,3_2$ & & & \\
\hline
9  & & & & & 1 & 4 & $3,1_2$ &  & & & \\
\hline
7  & & & & 1 & $1,1_2$ & $1,3_2$ & &  & & & \\
\hline
5  & & & & $1,1_2$ & $3,1_2$ &  & &  & & & \\
\hline
3  & & 1 & 2 & 1 &  &  & &  & & & \\
\hline
1  & &$1_2$ &\cellcolor{red} 1 &  &  &  & &  & & & \\
\hline
-1  & 1  & & &  &  &  & &  & & & \\
\hline

\end{tabular}
\caption{The Khovanov homology of the 14-crossing knot in Figure \ref{figure:14ex}. An entry of $k$ indicates a $\mathbb{Z}^k$ summand, and an entry of $k_2$ indicates a $\mathbb{Z}_2^k$ summand.}
\label{table:14ex}
\end{table}

\section{Rasmussen's invariant and $4$-genus}
\label{section:genus}

In this section, we prove Theorems \ref{theorem:ras} and \ref{theorem:genus}. Theorem \ref{theorem:ras} gives formulas the Rasmussen $s$ invariant of some Turaev genus one knots, and Theorem \ref{theorem:ras} gives bounds on the smooth $4$-genus of the same Turaev genus one knots. 

Let $D$ be an $A$-Turaev genus one diagram of a knot $K$. Theorem \ref{theorem:main} implies that one summand of Khovanov homology in the third nontrivial homological grading is trivial. If the third nontrivial homological grading in $Kh(D)$ is $i=0$, then Theorem \ref{theorem:ras} and its proof below show how to compute the Rasmussen $s$ invariant, which gives a lower bound on the smooth $4$-genus $g_4(K)$ of the knot $K$. If $D$ is $A$-almost alternating, then the third nontrivial homological grading is $i=0$ when $D$ has three negative crossings, and if $D$ is an $A$-adequate diagram with Turaev genus one, then the third nontrivial homological grading is $i=0$ when $D$ has two negative crossings.

The Seifert or oriented resolution of an oriented crossing $~ \tikz[baseline=.6ex, scale = .4]{
\draw[->] (0,0) -- (1,1);
\draw[->] (.3,.7) -- (0,1);
\draw (.7,.3) -- (1,0);
}~$ or $~ \tikz[baseline=.6ex, scale = .4]{
\draw[->] (.7,.7) -- (1,1);
\draw[->] (1,0) -- (0,1);
\draw (0,0) -- (.3,.3);
}~$ is given by $~ \tikz[baseline=.6ex, scale = .4]{
\draw[->] (0,0) to [out = 60, in = -60] (0,1);
\draw[->] (1,0) to [out = 120, in = 240] (1,1);
}.$ The state obtained by taking the Seifert resolution at every crossing is the Seifert state. Let $f(D)$ denote the number of components in the Seifert state of the link diagram $D$. A link is positive if it has a diagram where every crossing is positive. Nakamura \cite{Nakamura} proved the following theorem about the $4$-genus of positive links.
\begin{theorem}[Nakamura]
\label{theorem:Nakamura}
Let $D$ be a diagram of $L$ such that every crossing in $D$ is positive. The four-genus of $L$ is given by
\[g_4(L) = \frac{2-\mu(L) - f(D) +c(D)}{2}\]
where $\mu(L)$ is the number of components of $L$, $f(D)$ is the number of Seifert circles of $D$, and $c(D)$ is the number of crossings of $D$.
\end{theorem} 
Rasmussen \cite{Rasmussen:Slice} proved that if $K$ is positive, then $s(K) = 2g_4(K)$. Tagami \cite{Tagami} proved that if $K$ is almost-positive, that is, if $K$ has a diagram with one negative crossing, then $s(K) = 2g_4(K)$. Theorem \ref{theorem:genus} can be seen as an extension of this work in the case that the knot has Turaev genus one.

The following result is a well-known consequence of Rasmussen's construction of $s(K)$; see, for instance, Proposition 2.2 in Tagami \cite{Tagami}.
\begin{theorem}[Rasmussen]
\label{theorem:rasKh}
For any knot $K$, the Khovanov homology groups $Kh^{0,s(K)\pm 1}(K;\mathbb{Q})$ are nontrivial.
\end{theorem}

Theorems \ref{theorem:main} and \ref{theorem:rasKh} combine to give us the proof of Theorem \ref{theorem:ras}.
\begin{proof}[Proof of Theorem \ref{theorem:ras}]
Let $D$ be an $A$-adequate diagram such that $g_T(D)=1$ and $c_-(D)=2$. Theorem \ref{theorem:main} implies that $\ukh^{2,2-s_A(D)}(D)\cong Kh^{0,c(D)-s_A(D)-4}(D)$ is trivial. Then Inequality \ref{equation:diagonal} implies that $Kh^{0,j}(D)$ is nontrivial only when $j=c(D)-s_A(D)-2$ or $c(D)-s_A(D)$. Hence Theorem \ref{theorem:rasKh} implies that the Rasmussen $s$-invariant of $K$ is given by $s(K)=c(D)-s_A(D)-1$.

Now suppose that $D$ is an $A$-almost alternating diagram such that $c_-(D)=3$. Theorem \ref{theorem:main} implies that $\ukh^{3,4-s_A(D)}(D) \cong Kh^{0,c(D)-s_A(D)-5}(D)$ is trivial. Then Inequality \ref{equation:diagonal} implies that $Kh^{0,j}(D)$ is nontrivial only when $j=c(D)-s_A(D)-3$ or $c(D)-s_A(D)-1$. Finally Theorem \ref{theorem:rasKh} implies that $s(K) = c(D)-s_A(D)-2$.

The proofs in the cases where $D$ is $B$-almost alternating or $B$-adequate are similar.
\end{proof}

Theorem \ref{theorem:ras} gives the lower bound $\frac{|s(K)|}{2}$ for $g_4(K)$, but it remains to show that with the assumptions of Theorem \ref{theorem:genus}, $\frac{|s(K)|}{2}+1$ is an upper bound for $g_4(K)$. We begin by analyzing properties of diagrams of Turaev genus one with a small number of negative crossings.

Let $D$ be an oriented link diagram, and let $\Gamma=\Gamma(D)$ be the $4$-valent graph obtained by considering the crossings of $D$ as vertices in $\Gamma$ and the arcs of $D$ going between the crossings as the edges of $\Gamma$. The orientation of $D$ induces a direction on each edge of $\Gamma$ so that $\Gamma$ becomes a directed graph. If $D$ is $A$-adequate, the all-$A$ Kauffman state induces a cycle decomposition of $\Gamma$. A vertex of $\Gamma$ can be labeled as positive or negative according to whether the corresponding crossing in $D$ is positive or negative. 
See Figure \ref{figure:cycledecomp} for an example of this construction.
\begin{figure}[h]
\[\begin{tikzpicture}[scale = .5, thick]

\begin{knot}[
	consider self intersections,
 	clip width = 4,
 	ignore endpoint intersections = true,
	end tolerance = 1pt
 ]
 \flipcrossings{9,1,4,6,7}
 \strand (8,0) to [out = 0, in = -60] 
 (9,5) to [out = 120, in =-30]
 (5,12) to [out = 150, in = 60]
 (1,11) to [out = 240, in = 210]
 (3,7) to [out = 30, in = 225]
 (5,14) to [out = 45, in = 100]
 (13,12) to [out = -80, in=0]
 (11.5,8) to [out = 180, in = 0]
 (5,10) to [out = 180, in = 120, looseness=1.5]
 (1,4) to [out = -60, in = 225]
 (12,5) to [out = 45, in =270]
 (14,9) to [out = 90, in = 80]
 (6,4) to [out = 260, in= 180]
 (8,0);
 
 \end{knot}
 \draw (7.5,.5) -- (8,0) -- (7.5,-.5);
 
 \draw (8,-2) node{$D$};
 
 \begin{scope}[xshift= 16cm]
 
 \begin{knot}[
	consider self intersections,
 	clip width = 4,
 	ignore endpoint intersections = true,
	end tolerance = 1pt
 ]
 \flipcrossings{9,1,4,6,7}
 \strand (8,0) to [out = 0, in = -60] 
 (9,5) to [out = 120, in =-30]
 (5,12) to [out = 150, in = 60]
 (1,11) to [out = 240, in = 210]
 (3,7) to [out = 30, in = 225]
 (5,14) to [out = 45, in = 100]
 (13,12) to [out = -80, in=0]
 (11.5,8) to [out = 180, in = 0]
 (5,10) to [out = 180, in = 120, looseness=1.5]
 (1,4) to [out = -60, in = 225]
 (12,5) to [out = 45, in =270]
 (14,9) to [out = 90, in = 80]
 (6,4) to [out = 260, in= 180]
 (8,0);
 
 \end{knot}
  \draw (7.5,.5) -- (8,0) -- (7.5,-.5);
  
  \fill[red] (5.7, 1.3) circle (.2cm);
  \fill[red] (10,3) circle (.2cm);
  \fill[red] (13.5,9) circle (.2cm);
  \fill[red] (9.3,8) circle (.2cm);
  \fill[red](8.1,8) circle (.2cm);
 \fill[red] (7,9) circle (.2cm);
 \fill[red] (4.8,9.2) circle (.2cm);
 \fill[red] (1,7.2) circle (.2cm);
 
 \fill[red] (1.2,9.5) circle (.2cm);
 \fill[red] (3.7,10.5) circle (.2cm);
 \fill[red] (3.7, 12) circle (.2cm);
 
 \fill[red] (6.5,2.6) circle (.2cm);
 \fill[red] (9.1,3.4) circle (.2cm);
 \fill[red] (7.9,6.6) circle (.2cm);
 
 \fill[red] (4.7,12.5) circle (.2cm);
 \fill[red] (7.5,10) circle (.2cm);
 \fill[red] (9.3,9.2) circle (.2cm);
 \fill[red] (12.7,10.6) circle (.2cm);
 
  \begin{scope}[very thick,red, decoration={
    markings,
    mark=at position 0.5 with {\arrow{>}}}
    ] 
    \draw[postaction={decorate}] (5.7,1.3) to [out=-80, in =-80, looseness=2] (10,3);
    \draw[postaction={decorate}] (10,3) to [out=20, in =-90, looseness=1] (13.5,9);
     \draw[postaction={decorate}] (13.5,9) to [out=225, in =-30, looseness=1.3] (9.3,8);
     \draw[postaction={decorate}] (9.3,8) to (8.1,8);
     \draw[postaction={decorate}] (8.1,8) to (7,9);
     \draw[postaction={decorate}] (7,9) to [out=170, in =10, looseness=1] (4.8,9.2);
     \draw[postaction={decorate}] (1,7.2) to [out=-60, in =260, looseness=1.5] (4.8,9.2);
      \draw[postaction={decorate}] (1,7.2) to [out=-100, in =190, looseness=1.5] (5.7,1.3);
      
      \draw[postaction={decorate}] (3.7,10.5) to [out=190, in =20, looseness=1] (1.2,9.5);
       \draw[postaction={decorate}] (3.7,12) to [out=190, in =90, looseness=1] (1.2,9.5);
       \draw[postaction={decorate}] (3.7,10.5) -- (3.7,12);
       
       \draw[postaction={decorate}] (6.5,2.6) to [out = 30, in = 210] (9.1,3.4);
       \draw[postaction={decorate}] (9.1,3.4) to [out = 120, in = 300] (7.9,6.6);
       \draw[postaction={decorate}] (7.9,6.6) to [out = 240, in = 90] (6.5,2.6);
       
       \draw[postaction={decorate}] (4.7,12.5) to [out = 60, in = 100, looseness = 1.3] (12.7,10.6);
       \draw[postaction={decorate}] (12.7,10.6) to [out = 180, in = 30, looseness = 1.2] (9.3,9.2);
       \draw[postaction={decorate}] (9.3,9.2) to (7.5,10);
       \draw[postaction={decorate}] (7.5,10) to [out = 120, in = -30, looseness = 1.2] (4.7,12.5);

\end{scope}
 
  \draw (8,-2) node{$\mathcal{C}$};
 
 \end{scope}
 
\end{tikzpicture}\]
\caption{A diagram $D$ of the knot $9_{43}$ and the cycle decomposition $\mathcal{C}$ of $\Gamma(D)$ induced by the all-$A$ state of $D$. The diagram $D$ is $A$-adequate and has two negative crossings; the negative crossings correspond to sinks or sources in $\mathcal{C}$.}
\label{figure:cycledecomp}
\end{figure}
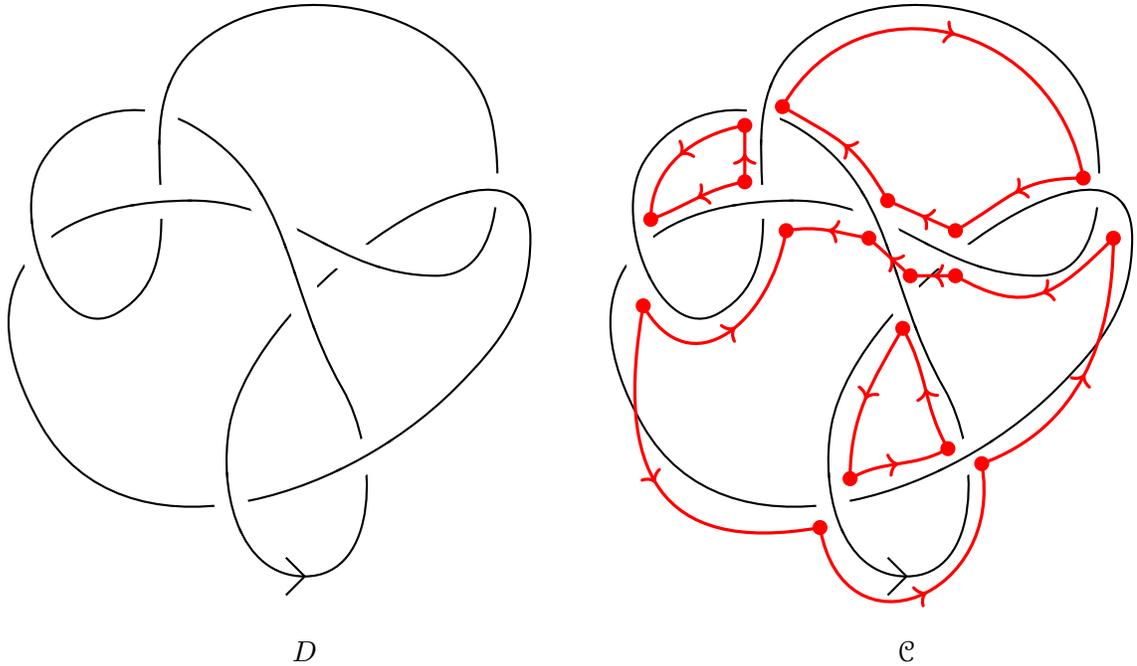

\begin{lemma}
\label{lemma:cycle}
Let $D$ be an oriented $A$-adequate link diagram. Each cycle in the cycle decomposition of $\Gamma(D)$ induced by the all-$A$ state of $D$ contains an even number of negative vertices.
\end{lemma}

\begin{proof}
Let $\mathcal{C}$ be the cycle decomposition of $\Gamma(D)$ induced by the all-$A$ state. Figure \ref{figure:negsinksource} shows that a negative crossing becomes two vertices in distinct cycles of $\mathcal{C}$ where one is a source and the other is a sink. A positive crossing becomes two vertices in distinct cycles of $\mathcal{C}$, neither of which is a sink or a source. Because every cycle whose edges are directed has an equal number of sinks and sources, the result follows.
\end{proof}

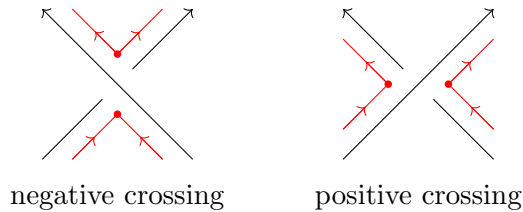
\begin{figure}
\[\begin{tikzpicture}
\draw (0,0) -- (.8,.8);
\draw[->] (1.2,1.2) -- (2,2);
\draw[->] (2,0) -- (0,2);

\fill[red] (1,1.4) circle (0.05cm);
\draw[red] (1,1.4) -- (1.6,2);
\draw[red,->] (1,1.4) -- (1.3,1.7);
\draw[red] (1,1.4) -- (.4,2);
\draw[red,->] (1,1.4) -- (.7,1.7);
\fill[red] (1,.6) circle (0.05cm);
\draw[red] (1.6,0) -- (1,.6);
\draw[red,->] (1.6,0) -- (1.3,.3);
\draw[red] (.4,0) -- (1,.6);
\draw[red,->] (.4,0) -- (.7,.3);

\draw (1,-.5) node{negative crossing};

\begin{scope}[xshift = 4cm]
\draw[->] (0,0) -- (2,2);
\draw (2,0) -- (1.2,.8);
\draw[->] (.8,1.2) -- (0,2);
\draw (1,-.5) node{positive crossing};

\fill[red] (1.4,1) circle (0.05cm);
\fill[red] (.6,1) circle (0.05cm);
\draw[red] (2,.4) -- (1.4,1) -- (2,1.6);
\draw[red] (0,.4) -- (.6,1) -- (0,1.6);
\draw[red,->] (0,.4) -- (.3,.7);
\draw[red,->] (.6,1) -- (.3,1.3);
\draw[red,->] (2,.4) -- (1.7,.7);
\draw[red,->] (1.4,1) -- (1.7,1.3);

\end{scope}

\end{tikzpicture}\]
\caption{A negative crossing results in a sink or source in the cycle decomposition induced by the all-$A$ state while a positive crossing does not.}
\label{figure:negsinksource}
\end{figure}

In order to prove Theorem \ref{theorem:genus}, we use the following lemma describing the structure of one type of knot that satisfies the assumptions of the theorem.

\begin{figure}[h]
\[\begin{tikzpicture}[scale = 1]

\draw (0,0) rectangle (1.5,1.5);
\draw (0,2.5) rectangle (1.5,4);
\draw (2.5,0) rectangle (4,1.5);
\draw (2.5,2.5) rectangle (4,4);

\draw (.75,3.25) node{$R_1$};
\draw (3.25,3.25) node{$R_2$};
\draw (3.25,.75) node{$R_3$};
\draw (.75,.75) node{$R_4$};

\draw (.4,1.25) node{\tiny{$-$}};
\draw (1.1,1.25) node{\tiny{$+$}};
\draw (1.25,1.1) node{\tiny{$-$}};
\draw (1.25,.4) node{\tiny{$+$}};

\draw (.4,2.75) node{\tiny{$+$}};
\draw (1.1,2.75) node{\tiny{$-$}};
\draw (1.25,2.9) node{\tiny{$+$}};
\draw (1.25,3.6) node{\tiny{$-$}};

\draw (2.9,1.25) node{\tiny{$-$}};
\draw (3.6,1.25) node{\tiny{$+$}};
\draw (2.75,1.1) node{\tiny{$+$}};
\draw (2.75,.4) node{\tiny{$-$}};

\draw (2.9,2.75) node{\tiny{$+$}};
\draw (3.6,2.75) node{\tiny{$-$}};
\draw (2.75,2.9) node{\tiny{$-$}};
\draw (2.75,3.6) node{\tiny{$+$}};

\begin{knot}[    
    consider self intersections,
    clip width = 5,
    ignore endpoint intersections = true,
    end tolerance = 1pt
    ]

\strand[->]
    (1.1,2.5) to [out = 270, in =90 ] (.4,1.5);
\strand[->]
    (1.1,1.5) to [out = 90, in =270 ] (.4,2.5);    

\strand[->] (2.5,3.6) to [out = 180, in =0] (1.5,2.9);
\strand[->] (1.5,3.6) to [out = 0, in= 180] (2.5,2.9);

\strand[->] (2.9,1.5) to [out = 90, in =270] (3.6,2.5);
\strand [->] (2.9,2.5) to [out = 270, in =90] (3.6,1.5);

\strand[->] (2.5,.4) to [out = 180, in=0] (1.5,1.1);
\strand[->] (1.5,.4) to [out = 0, in = 180] (2.5,1.1);

\end{knot}

\draw[red, dashed] (2,3.25) circle (.3cm);

\begin{scope}[xshift = 5cm]

\draw (0,0) rectangle (6,1.5);
\draw (0,2.5) rectangle (6,4);

\draw (3,.75) node{$R_2$};
\draw (3,3.25) node{$R_1$};

\draw (.25,1.25) node{\tiny{$-$}};
\draw (1.25,1.25) node{\tiny{$+$}};
\draw (1.75,1.25) node{\tiny{$-$}};
\draw (2.75,1.25) node{\tiny{$+$}};
\draw (3.25,1.25) node{\tiny{$-$}};
\draw (4.25,1.25) node{\tiny{$+$}};
\draw (4.75,1.25) node{\tiny{$-$}};
\draw (5.75,1.25) node{\tiny{$+$}};

\draw (.25,2.75) node{\tiny{$+$}};
\draw (1.25,2.75) node{\tiny{$-$}};
\draw (1.75,2.75) node{\tiny{$+$}};
\draw (2.75,2.75) node{\tiny{$-$}};
\draw (3.25,2.75) node{\tiny{$+$}};
\draw (4.25,2.75) node{\tiny{$-$}};
\draw (4.75,2.75) node{\tiny{$+$}};
\draw (5.75,2.75) node{\tiny{$-$}};

\begin{knot}[    
    consider self intersections,
    clip width = 5,
    ignore endpoint intersections = true,
    end tolerance = 1pt
    ]
    \strand[->] (.25,2.5) to [out = 270,in=90] (1.25,1.5);
    \strand[->] (.25,1.5) to [out = 90, in=270] (1.25,2.5);
    
    \strand[->] (1.75,1.5) to [out = 90, in = 270] (2.75,2.5);
    \strand[->] (1.75, 2.5) to [out = 270, in =90] (2.75,1.5);
    
    \strand[->] (3.25,1.5) to [out = 90, in = 270] (4.25,2.5);
    \strand[->] (3.25, 2.5) to [out = 270, in =90] (4.25,1.5);
    
    \strand[->] (4.75,1.5) to [out = 90, in = 270] (5.75,2.5);
    \strand[->] (4.75, 2.5) to [out = 270, in =90] (5.75,1.5);
    
    \end{knot}
    
    \draw[red, dashed] (.75,2) circle (.3cm);

\end{scope}

\end{tikzpicture}\]
\caption{The two possible formats for an $A$-almost alternating diagram with positive dealternator and three negative crossings. Each tangle $R_i$ is a positive alternating tangle. In each case, the dealternator is circled.}
\label{figure:aa3}
\end{figure}
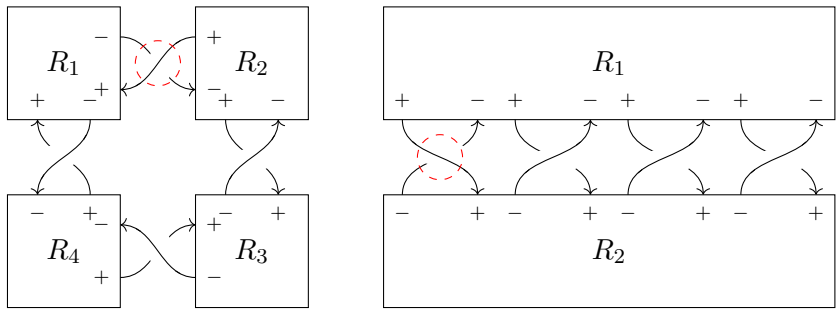

\begin{lemma}
\label{lemma:aa3pos}
Let $D$ be an $A$-almost alternating diagram such that $c_-(D)=3$ and the dealternator is a positive crossing. Then $D$ is one of the two diagrams depicted in Figure \ref{figure:aa3} where each tangle $R_i$ is an alternating tangle with all positive crossings. 
\end{lemma}

\begin{proof}
Let $D_{\alt}$ be the alternating diagram obtained from $D$ by changing the dealternator. Since the dealternator is positive, $D_{\alt}$ has four negative crossings, corresponding to four negative vertices $v_1$, $v_2$, $v_3$, and $v_4$ in $\Gamma(D_{alt})$. Since each vertex in $\Gamma(D_{\alt})$ appears twice in the cycle decomposition $\mathcal{C}$ induced by the all-$A$ state of $D_{\alt}$, it follows that the cycle decomposition $\mathcal{C}$ contains eight negative vertices, with at most four being contained in a single cycle.  Lemma \ref{lemma:cycle} states that each cycle in $\mathcal{C}$ contains an even number of negative vertices, and thus there are three possibilities for the cycles containing negative crossings.
\begin{enumerate}
\item Two cycles contain four negative vertices each.
\item One cycle contains four negative vertices and two cycles contain two negative vertices each.
\item Four cycles contain two negative vertices each.
\end{enumerate}
If there are two cycles with four negative vertices each in $\mathcal{C}$, then $D_{\alt}$ is in the format of $D_1$ in Figure \ref{figure:4neg}, and changing one of the negative crossings in $D_1$ results in the first diagram in Figure \ref{figure:aa3}. If there is one cycle with four negative vertices and two cycles with two negative vertices each in $\mathcal{C}$, then $D_{\alt}$ is in the format $D_3$ in Figure \ref{figure:4neg}. However, changing any of the four negative crossings in such a diagram results in the diagram not being $A$-almost alternating because there will be a region that shares a crossing in its border with both the regions $u_1$ and $u_2$. 

Suppose that $\mathcal{C}$ has four cycles each with two negative vertices. Either there are two vertices $v_i$ and $v_j$ that are both contained in two different cycles of $\mathcal{C}$ or no such pair of vertices exist. In the former case, the $D_{\alt}$ is in the format of $D_4$ in Figure \ref{figure:4neg}. Changing any of the four negative crossings in such a diagram results in the regions $u_1$ and $u_2$ being adjacent, and hence the resulting diagram is not $A$-almost alternating. In the latter case, $D_{\alt}$ is in the format of $D_2$ in Figure \ref{figure:4neg}, and changing one of the negative crossings in $D_2$ results in the second diagram of Figure \ref{figure:aa3}.
\end{proof}

\begin{figure}[h]
\[\begin{tikzpicture}[scale = 1]

\draw (2,-.5) node{$D_1$};

\draw (0,0) rectangle (1.5,1.5);
\draw (0,2.5) rectangle (1.5,4);
\draw (2.5,0) rectangle (4,1.5);
\draw (2.5,2.5) rectangle (4,4);

\draw (.75,3.25) node{$R_1$};
\draw (3.25,3.25) node{$R_2$};
\draw (3.25,.75) node{$R_3$};
\draw (.75,.75) node{$R_4$};

\draw (.4,1.25) node{\tiny{$-$}};
\draw (1.1,1.25) node{\tiny{$+$}};
\draw (1.25,1.1) node{\tiny{$-$}};
\draw (1.25,.4) node{\tiny{$+$}};

\draw (.4,2.75) node{\tiny{$+$}};
\draw (1.1,2.75) node{\tiny{$-$}};
\draw (1.25,2.9) node{\tiny{$+$}};
\draw (1.25,3.6) node{\tiny{$-$}};

\draw (2.9,1.25) node{\tiny{$-$}};
\draw (3.6,1.25) node{\tiny{$+$}};
\draw (2.75,1.1) node{\tiny{$+$}};
\draw (2.75,.4) node{\tiny{$-$}};

\draw (2.9,2.75) node{\tiny{$+$}};
\draw (3.6,2.75) node{\tiny{$-$}};
\draw (2.75,2.9) node{\tiny{$-$}};
\draw (2.75,3.6) node{\tiny{$+$}};

\begin{knot}[    
    consider self intersections,
    clip width = 5,
    ignore endpoint intersections = true,
    end tolerance = 1pt
    ]

\strand[->]
    (1.1,1.5) to [out = 90, in =270 ] (.4,2.5);  
   \strand[->]
    (1.1,2.5) to [out = 270, in =90 ] (.4,1.5);

\strand[->] (2.5,3.6) to [out = 180, in =0] (1.5,2.9);
\strand[->] (1.5,3.6) to [out = 0, in= 180] (2.5,2.9);

\strand[->] (2.9,1.5) to [out = 90, in =270] (3.6,2.5);
\strand [->] (2.9,2.5) to [out = 270, in =90] (3.6,1.5);

\strand[->] (2.5,.4) to [out = 180, in=0] (1.5,1.1);
\strand[->] (1.5,.4) to [out = 0, in = 180] (2.5,1.1);

\end{knot}


\begin{scope}[xshift = 5cm]

\draw (3,-.5) node{$D_2$};

\draw (0,0) rectangle (6,1.5);
\draw (0,2.5) rectangle (6,4);

\draw (3,.75) node{$R_2$};
\draw (3,3.25) node{$R_1$};

\draw (.25,1.25) node{\tiny{$-$}};
\draw (1.25,1.25) node{\tiny{$+$}};
\draw (1.75,1.25) node{\tiny{$-$}};
\draw (2.75,1.25) node{\tiny{$+$}};
\draw (3.25,1.25) node{\tiny{$-$}};
\draw (4.25,1.25) node{\tiny{$+$}};
\draw (4.75,1.25) node{\tiny{$-$}};
\draw (5.75,1.25) node{\tiny{$+$}};

\draw (.25,2.75) node{\tiny{$+$}};
\draw (1.25,2.75) node{\tiny{$-$}};
\draw (1.75,2.75) node{\tiny{$+$}};
\draw (2.75,2.75) node{\tiny{$-$}};
\draw (3.25,2.75) node{\tiny{$+$}};
\draw (4.25,2.75) node{\tiny{$-$}};
\draw (4.75,2.75) node{\tiny{$+$}};
\draw (5.75,2.75) node{\tiny{$-$}};

\begin{knot}[    
    consider self intersections,
    clip width = 5,
    ignore endpoint intersections = true,
    end tolerance = 1pt
    ]

    \strand[->] (.25,1.5) to [out = 90, in=270] (1.25,2.5);
        \strand[->] (.25,2.5) to [out = 270,in=90] (1.25,1.5);
    
    \strand[->] (1.75,1.5) to [out = 90, in = 270] (2.75,2.5);
    \strand[->] (1.75, 2.5) to [out = 270, in =90] (2.75,1.5);
    
    \strand[->] (3.25,1.5) to [out = 90, in = 270] (4.25,2.5);
    \strand[->] (3.25, 2.5) to [out = 270, in =90] (4.25,1.5);
    
    \strand[->] (4.75,1.5) to [out = 90, in = 270] (5.75,2.5);
    \strand[->] (4.75, 2.5) to [out = 270, in =90] (5.75,1.5);
    
    \end{knot}
    

\end{scope}

\begin{scope}[yshift = -5cm, xshift = -1.5cm]

\draw (2.75,-.5) node{$D_3$};

\draw (0,0) rectangle (1,3);
\draw (2,0) rectangle (3.5,3);
\draw (4.5,0) rectangle (5.5,3);
\draw (.4,1.5) node{$R_1$};
\draw (2.75,1.5) node{$R_2$};
\draw (5.1,1.5) node{$R_3$};

\draw (.75,.25) node{\tiny{$+$}};
\draw (.75,1.25) node{\tiny{$-$}};
\draw (.75,1.75) node{\tiny{$+$}};
\draw (.75,2.75) node{\tiny{$-$}};

\draw (2.25,.25) node{\tiny{$-$}};
\draw (2.25,1.25) node{\tiny{$+$}};
\draw (2.25,1.75) node{\tiny{$-$}};
\draw (2.25,2.75) node{\tiny{$+$}};

\draw (3.25,.25) node{\tiny{$+$}};
\draw (3.25,1.25) node{\tiny{$-$}};
\draw (3.25,1.75) node{\tiny{$+$}};
\draw (3.25,2.75) node{\tiny{$-$}};

\draw (4.75,.25) node{\tiny{$-$}};
\draw (4.75,1.25) node{\tiny{$+$}};
\draw (4.75,1.75) node{\tiny{$-$}};
\draw (4.75,2.75) node{\tiny{$+$}};

\begin{knot}[    
    consider self intersections,
    clip width = 5,
    ignore endpoint intersections = true,
    end tolerance = 1pt
    ]

    \strand[->] (2,.25) to [out = 180, in=0] (1,1.25);
        \strand[->] (1,.25) to [out = 0,in=180] (2,1.25);
    
    \strand[->] (2,1.75) to [out = 180, in=0] (1,2.75);
        \strand[->] (1,1.75) to [out = 0,in=180] (2,2.75);
    
    \strand[->] (3.5,2.75) to [out = 0, in = 180] (4.5,1.75);
    \strand[->] (4.5,2.75) to [out =180, in = 0] (3.5,1.75);
    
    \strand[->] (3.5,1.25) to [out = 0, in = 180] (4.5,.25);
    \strand[->] (4.5,1.25) to [out =180, in = 0] (3.5,.25);
    
    \end{knot}

\end{scope}

\begin{scope}[xshift = 5cm, yshift = -5cm]

\draw (3,-.5) node{$D_4$};

\draw (0,0) rectangle (6,4);
\draw (.5,.5) rectangle (2.5,3.5);
\draw (3.5,.5) rectangle (5.5,3.5);
\draw (1,1.5) rectangle (2,2.5);
\draw (4,1.5) rectangle (5,2.5);

\draw (1.5,2) node{$R_1$};
\draw (4.5,2) node{$R_2$};
\draw (3,2) node{$R_3$};

\draw (1.2,.3) node{\tiny{$-$}};
\draw (1.8,.3) node{\tiny{$+$}};

\draw (1.2,1.7) node{\tiny{$+$}};
\draw (1.8,1.7) node{\tiny{$-$}};

\draw (1.2,2.3) node{\tiny{$-$}};
\draw (1.8,2.3) node{\tiny{$+$}};

\draw (1.2,3.7) node{\tiny{$+$}};
\draw (1.8,3.7) node{\tiny{$-$}};

\draw (4.2,.3) node{\tiny{$-$}};
\draw (4.8,.3) node{\tiny{$+$}};

\draw (4.2,1.7) node{\tiny{$+$}};
\draw (4.8,1.7) node{\tiny{$-$}};

\draw (4.2,2.3) node{\tiny{$-$}};
\draw (4.8,2.3) node{\tiny{$+$}};

\draw (4.2,3.7) node{\tiny{$+$}};
\draw (4.8,3.7) node{\tiny{$-$}};

\begin{knot}[    
    consider self intersections,
    clip width = 5,
    ignore endpoint intersections = true,
    end tolerance = 1pt
    ]

    \strand[->] (1.8,1.5) to [out = 270, in=90] (1.2,.5);
        \strand[->] (1.8,.5) to [out = 90,in=270] (1.2,1.5);
    
    \strand[->] (1.2,2.5) to [out = 90, in=270] (1.8,3.5);
        \strand[->] (1.2,3.5) to [out = 270,in=90] (1.8,2.5);
    
    \strand[->] (4.8,1.5) to [out = 270, in=90] (4.2,.5);
        \strand[->] (4.8,.5) to [out = 90,in=270] (4.2,1.5);
    
    \strand[->] (4.2,2.5) to [out = 90, in=270] (4.8,3.5);
        \strand[->] (4.2,3.5) to [out = 270,in=90] (4.8,2.5);
    
    \end{knot}

\end{scope}

\end{tikzpicture}\]
\caption{Reduced alternating diagrams with four negative crossings. The tangle $R_i$ is positive and alternating.}
\label{figure:4neg}
\end{figure}
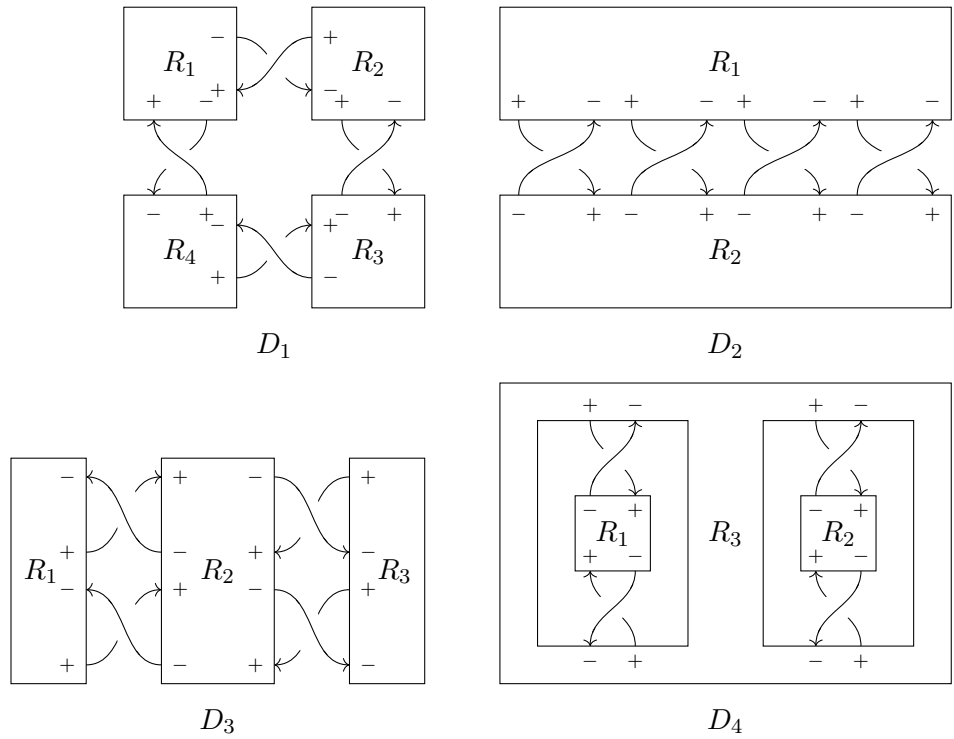

In order to prove Theorem \ref{theorem:genus}, we need two strategies that yield upper bounds on $g_4(K)$. The first is that $g_4(K)\leq g_3(K)$ where $g_3(K)$ is the Seifert genus of $K$, i.e. the minimum genus of any Seifert surface of $K$. If $D$ is a diagram of $K$, then define $g_3(D)$ to be the genus of the surface obtained by applying Seifert's algorithm to $D$. We have
\[g_3(D) = \frac{1}{2}(1 + c(D) - f(D))\]
where $c(D)$ is the number of crossings of $D$ and $f(D)$ is the number of components in the oriented or Seifert state of $D$. It follows that if $D$ is a diagram of $K$, then
\begin{equation}
\label{eq:g4bound1}
g_4(K) \leq \frac{1}{2}(1 + c(D)-f(D)).
\end{equation}

The second strategy we use is modifying a diagram via a saddle move, that is replacing two adjacent and oppositely oriented arcs in a diagram as in Figure \ref{figure:saddlemove}. Let $K_1$ be a knot with diagram $D_1$. Let $D_2$ be the diagram obtained after performing two saddle moves to $D_1$, and let $K_2$ be the link with diagram $D_2$. A saddle move changes the number of components of a link by one, and thus the link $K_2$ will either have one or three components. Figure \ref{figure:saddlemove} indicates the following inequality
\begin{equation}
\label{eq:g4bound2}
g_4(K_1) \leq \begin{cases} g_4(K_2) + 1&\text{if $K_2$ is a knot,}\\
g_4(K_2) + 2 & \text{if $K_2$ is a link with three components.}
\end{cases}
\end{equation}

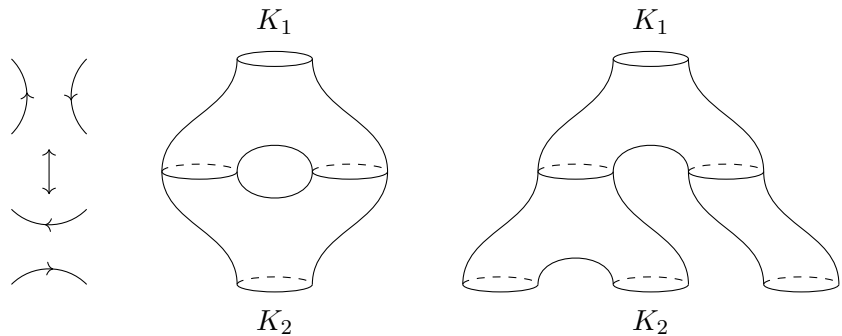
\begin{figure}[h]
\[\begin{tikzpicture}

    \begin{scope}[decoration={
    markings,
    mark=at position 0.55 with {\arrow{>}}}
    ] 
    \draw[postaction={decorate}] (0,0) to [out=45, in =135] (1,0);
    \draw[postaction={decorate}] (1,1) to [out=225, in =-45] (0,1);
    
    \draw[postaction={decorate}] (0,2) to [out=45, in =-45] (0,3);
    \draw[postaction={decorate}] (1,3) to [out=225, in =135] (1,2);

\end{scope}

\draw[<->] (.5,1.2) -- (.5,1.8);

\draw (3.5,3) ellipse (.5cm and .1cm);
\draw (2,1.5) arc (180:360:.5cm and .1cm);
\draw[dashed] (2,1.5) arc (180:0:.5cm and .1cm);

\draw (4,1.5) arc (180:360:.5cm and .1cm);
\draw[dashed] (4,1.5) arc (180:0:.5cm and .1cm);

\draw (3,0) arc (180:360:.5cm and .1cm);
\draw[dashed] (3,0) arc (180:0:.5cm and .1cm);

\draw (3,3) to [out = 270, in = 90] (2,1.5) to [out = 270, in = 90] (3,0);
\draw (4,3) to [out = 270, in = 90] (5,1.5) to [out = 270, in = 90] (4,0);
\draw (3,1.5) to [out = 90, in = 90, looseness = 1.2] (4,1.5) to [out = 270, in = 270, looseness=1.2] (3,1.5);

\draw (3.5,3.5) node{$K_1$};
\draw (3.5,-.5) node{$K_2$};

\begin{scope}[xshift = 5cm]
\draw (3.5,3) ellipse (.5cm and .1cm);
\draw (2,1.5) arc (180:360:.5cm and .1cm);
\draw[dashed] (2,1.5) arc (180:0:.5cm and .1cm);

\draw (4,1.5) arc (180:360:.5cm and .1cm);
\draw[dashed] (4,1.5) arc (180:0:.5cm and .1cm);

\draw (3,0) arc (180:360:.5cm and .1cm);
\draw[dashed] (3,0) arc (180:0:.5cm and .1cm);

\draw (1,0) arc (180:360:.5cm and .1cm);
\draw[dashed] (1,0) arc (180:0:.5cm and .1cm);

\draw (5,0) arc (180:360:.5cm and .1cm);
\draw[dashed] (5,0) arc (180:0:.5cm and .1cm);

\draw(3,3) to [out = 270, in = 90] (2,1.5) to [out = 270, in = 90] (1,0);
\draw (4,3) to [out = 270, in = 90] (5,1.5) to [out = 270, in = 90] (6,0);
\draw (5,0) to [out = 90, in = 270] (4,1.5) to [out = 90, in = 90,looseness = 1.2] (3,1.5) to [out = 270, in = 90] (4,0);
\draw (3,0) to [out = 90, in = 90, looseness = 1.2] (2,0);

\draw (3.5,3.5) node{$K_1$};
\draw (3.5,-.5) node{$K_2$};

\end{scope}

\end{tikzpicture}\]
\caption{A saddle move and the induced resulting cobordism between $K_1$ and $K_2$ in the cases where $K_2$ is a knot or a three component link.}
\label{figure:saddlemove}
\end{figure}

\begin{proof}[Proof of Theorem \ref{theorem:genus}]
Suppose that $D$ is $A$-adequate, has Turaev genus one, and has two negative crossings. Since $D$ is $A$-adequate, the cycle decomposition of $\Gamma(D)$ induced by the all-$A$ states has two cycles each with two negative vertices. For positive crossings, the $A$-resolution and the oriented resolution agree. The Seifert state for $D$ and the all-$A$ state for $D$ are related by changing the two oriented or $B$-resolutions at the negative crossings. Because the vertices corresponding to those circles are in two cycles, performing both switches preserves the number of components in the state. Thus the number $f(D)$ of components in the Seifert state of $D$ is the same as the number $s_A(D)$ of components in the all-$A$ state of $D$. Equation \ref{eq:g4bound1} implies
\[g_4(K) \leq g_3(D) = \frac{1}{2}(1+c(D) - f(D)) = \frac{1}{2}(1+c(D) - s_A(D)) =\frac{s(K)}{2}+1,\]
where the last equality follows from Theorem \ref{theorem:ras}.

Suppose that $D$ is $A$-almost alternating and has three negative crossings. Furthermore, suppose the dealternator of $D$ is negative, and let $D_{\alt}$ be the diagram obtained from $D$ be changing the dealternator. Then $D_{\alt}$ is a reduced alternating, and thus adequate, diagram with two negative crossings. Hence $s_A(D_{\alt})=f(D_{\alt})$ as in the previous paragraph. Because $D$ and $D_{\alt}$ are related by a crossing change, the genera of the surfaces obtained by Seifert's algorithm are the same, that is $g_3(D) = g_3(D_{\alt})$. Also, $s_A(D_{\alt}) = s_A(D) + 1$. Thus
\begin{align*}
g_3(D) = & \; g_3(D_{\alt})\\
 = & \; \frac{1}{2}(1+c(D_{\alt}) - f(D_{\alt}))\\
 = & \; \frac{1}{2}(1+c(D_{\alt}) - s_A(D_{\alt}) )\\
 = &\;  \frac{1}{2}(c(D) -s_A(D)).
 \end{align*}
Hence Equation \ref{eq:g4bound1} implies
\[g_4(K) \leq g_3(D) = \frac{1}{2}(c(D) -s_A(D)) = \frac{s(K)}{2}+1,\]
where again the last equality follows from Theorem \ref{theorem:ras}.

Suppose that $D$ is $A$-almost alternating and has three negative crossings, but now also suppose the dealternator of $D$ is positive. Lemma \ref{lemma:aa3pos} states that $D$ is of the format of one of the two diagrams in Figure \ref{figure:aa3}. Suppose that $D$ has the format of the diagram on the left of Figure \ref{figure:aa3}. Performing a flype and a Reidemeister 2 move around the tangle $R_1$ results in a positive knot, and thus $g_4(K)=s(K)$ in this case.

Finally suppose that $D$ has the format of the diagram on the right in Figure \ref{figure:aa3}. Unfortunately, the approach using Equation \ref{eq:g4bound1} yields $g_4(K) \leq \frac{s(K)}{2} + 2$, and so in order to prove the result, we use a strategy involving Inequality \ref{eq:g4bound2}. As shown in Figure \ref{figure:aa32}, one can perform two saddle moves followed by a flype and transform $D$ into a positive diagram $D_{\pos}$. Let $L_{\pos}$ be the link whose diagram is $D_{\pos}$. By Theorem \ref{theorem:Nakamura},
\[g_4(L_{\pos}) = \frac{1}{2}(2-\mu(L_{\pos}) - f(D_{\pos}) + c(D_{pos}))\]
where the number of components $\mu(L_{\pos})$ of $L_{\pos}$ is either one or three. We have $c(D_{\pos}) = c(D) -4$ and $f(D_{\pos})=s_A(D_{\pos}) = s_A(D)-1$. If $L_{\pos}$ is a knot, then Inequality \ref{eq:g4bound2} implies that
\[g_4(K) \leq g_4(L_{\pos})+1 = \frac{1}{2}(3 + c(D_{\pos}) - f(D_{\pos})) = \frac{1}{2}(c(D) - s_A(D)) = \frac{s(K)}{2}+1,\]
where the last equality follows from Theorem \ref{theorem:ras}. If $L_{\pos}$ is a link with three components, then Inequality \ref{eq:g4bound2} implies that
\[g_4(K) \leq g_4(L_{\pos})+2 = \frac{1}{2}(3+c(D_{\pos}) - f(D_{\pos}))=\frac{s(K)}{2}+1,\]
as above. 

Therefore $\frac{s(K)}{2} \leq g_4(K) \leq \frac{s(K)}{2}+1$ when $K$ is in case $1$ or $2$ of the theorem. The result is proved analogously when $K$ is in case $3$ or $4$.
\end{proof}

\begin{figure}
\[\begin{tikzpicture}

\draw (0,0) rectangle (4,1);
\draw (0,2) rectangle (4,3);

\draw (.2,.8) node{\tiny{$-$}};
\draw (.8,.8) node{\tiny{$+$}};
\draw (1.2,.8) node{\tiny{$-$}};
\draw (1.8,.8) node{\tiny{$+$}};
\draw (2.2,.8) node{\tiny{$-$}};
\draw (2.8,.8) node{\tiny{$+$}};
\draw (3.2,.8) node{\tiny{$-$}};
\draw (3.8,.8) node{\tiny{$+$}};

\draw (2,.4) node{$R_2$};

\draw (.2,2.2) node{\tiny{$+$}};
\draw (.8,2.2) node{\tiny{$-$}};
\draw (1.2,2.2) node{\tiny{$+$}};
\draw (1.8,2.2) node{\tiny{$-$}};
\draw (2.2,2.2) node{\tiny{$+$}};
\draw (2.8,2.2) node{\tiny{$-$}};
\draw (3.2,2.2) node{\tiny{$+$}};
\draw (3.8,2.2) node{\tiny{$-$}};

\draw (2,2.6) node{$R_1$};

\begin{knot}[    
    consider self intersections,
    clip width = 5,
    ignore endpoint intersections = true,
    end tolerance = 1pt
    ]

    \strand[->] (.2,2) to [out = 270, in = 90] (.8,1);
    \strand[->] (.2,1) to [out = 90, in = 270] (.8,2);
    
    \strand[->] (1.2,1) to [out = 90, in = 270] (1.8,2);
    \strand[->] (1.2,2) to [out = 270, in = 90] (1.8,1);

    \strand[->] (2.2,1) to [out = 90, in = 270] (2.8,2);
    \strand[->] (2.2,2) to [out = 270, in = 90] (2.8,1);
    
    \strand[->] (3.2,1) to [out = 90, in = 270] (3.8,2);
    \strand[->] (3.2,2) to [out = 270, in = 90] (3.8,1);
    
    \end{knot}
    
    \draw (2,-.5) node{$D$};
    
\begin{scope}[xshift = 6cm]

\draw (.6,0) rectangle (3.4,1);
\draw (.6,2) rectangle (3.4,3);

\draw (.8,.2) node{\tiny{$-$}};
\draw (.8,.8) node{\tiny{$+$}};
\draw (1.2,.8) node{\tiny{$-$}};
\draw (1.8,.8) node{\tiny{$+$}};
\draw (2.2,.8) node{\tiny{$-$}};
\draw (2.8,.8) node{\tiny{$+$}};
\draw (3.2,.8) node{\tiny{$-$}};
\draw (3.2,.2) node{\tiny{$+$}};

\draw (.8,2.2) node{\tiny{$-$}};
\draw (.8,2.8) node{\tiny{$+$}};
\draw (1.2,2.2) node{\tiny{$+$}};
\draw (1.8,2.2) node{\tiny{$-$}};
\draw (2.2,2.2) node{\tiny{$+$}};
\draw (2.8,2.2) node{\tiny{$-$}};
\draw (3.2,2.8) node{\tiny{$-$}};
\draw (3.2,2.2) node{\tiny{$+$}};

\draw (2,2.6) node{$R_1$};
\draw (2,.4) node{$R_2$};

\begin{knot}[    
    clip width = 4,
    ignore endpoint intersections = false,
    end tolerance = 1pt
    ]

    \strand[->] (.6,2.8) to [out = 180, in= 180,looseness=1.4] (.6,.8);
    \strand[->] (.6,.2) to [out = 180, in = 180,looseness=1.4] (.6,2.2);
    
    \strand[->] (1.2,1) to [out = 90, in = 270] (1.8,2);
    \strand[->] (1.2,2) to [out = 270, in = 90] (1.8,1);

    \strand[->] (2.2,1) to [out = 90, in = 270] (2.8,2);
    \strand[->] (2.2,2) to [out = 270, in = 90] (2.8,1);
    
   \strand[->] (3.4,.8) to [out=0, in=0, looseness = 1.4 ] (3.4,2.8);
    \strand[->] (3.4,2.2) to [out=0, in=0, looseness = 1.4] (3.4,.2);

    \end{knot}
    
        \draw (2,-.5) node{$D$};

\end{scope}

\begin{scope}[yshift = -4cm]

\draw (.6,0) rectangle (3.4,1);
\draw (.6,2) rectangle (3.4,3);

\draw (.8,.2) node{\tiny{$-$}};
\draw (.8,.8) node{\tiny{$+$}};
\draw (1.2,.8) node{\tiny{$-$}};
\draw (1.8,.8) node{\tiny{$+$}};
\draw (2.2,.8) node{\tiny{$-$}};
\draw (2.8,.8) node{\tiny{$+$}};
\draw (3.2,.8) node{\tiny{$-$}};
\draw (3.2,.2) node{\tiny{$+$}};

\draw (.8,2.2) node{\tiny{$-$}};
\draw (.8,2.8) node{\tiny{$+$}};
\draw (1.2,2.2) node{\tiny{$+$}};
\draw (1.8,2.2) node{\tiny{$-$}};
\draw (2.2,2.2) node{\tiny{$+$}};
\draw (2.8,2.2) node{\tiny{$-$}};
\draw (3.2,2.8) node{\tiny{$-$}};
\draw (3.2,2.2) node{\tiny{$+$}};

\draw (2,2.6) node{$R_1$};
\draw (2,.4) node{$R_2$};

\begin{knot}[    
    clip width = 4,
    ignore endpoint intersections = false,
    end tolerance = 1pt
    ]

    \strand[->] (.6,2.8) to [out = 180, in= 180,looseness=1.4] (.6,.8);
    \strand[->] (.6,.2) to [out = 180, in = 180,looseness=1.4] (.6,2.2);

   \strand[->] (3.4,.8) to [out=0, in=0, looseness = 1.4 ] (3.4,2.8);
    \strand[->] (3.4,2.2) to [out=0, in=0, looseness = 1.4] (3.4,.2);

    \end{knot}
    
    \draw[->] (1.2,1) to [out = 90, in = 90] (1.8,1);
    \draw[->] (1.2,2) to [out = 270, in = 270] (1.8,2);
    \draw[->] (2.2,1) to [out = 90, in = 90] (2.8,1);
    \draw[->] (2.2,2) to [out = 270, in = 270] (2.8,2);
    
        \draw (2,-.5) node{$\widetilde{D}$};

\end{scope}

\begin{scope}[yshift = -4cm, xshift = 6cm]

\draw (.6,0) rectangle (3.4,1);
\draw (.6,2) rectangle (3.4,3);

\draw (.8,.2) node{\tiny{$-$}};
\draw (.8,.8) node{\tiny{$+$}};
\draw (1.2,.8) node{\tiny{$-$}};
\draw (1.8,.8) node{\tiny{$+$}};
\draw (2.2,.8) node{\tiny{$-$}};
\draw (2.8,.8) node{\tiny{$+$}};
\draw (3.2,.8) node{\tiny{$-$}};
\draw (3.2,.2) node{\tiny{$+$}};

\draw (.8,2.2) node{\tiny{$-$}};
\draw (.8,2.8) node{\tiny{$+$}};
\draw (1.2,2.8) node{\tiny{$-$}};
\draw (1.8,2.8) node{\tiny{$+$}};
\draw (2.2,2.8) node{\tiny{$-$}};
\draw (2.8,2.8) node{\tiny{$+$}};
\draw (3.2,2.8) node{\tiny{$-$}};
\draw (3.2,2.2) node{\tiny{$+$}};

\draw (2,2.4) node[yscale=-1]{$R_1$};
\draw (2,.4) node{$R_2$};
    
    \draw[->] (1.2,1) to [out = 90, in = 90] (1.8,1);
    \draw[->] (1.2,3) to [out = 90, in = 90] (1.8,3);
    \draw[->] (2.2,1) to [out = 90, in = 90] (2.8,1);
    \draw[->] (2.2,3) to [out = 90, in = 90] (2.8,3);
    
    \draw[->] (.6,2.2) to [out =180, in = 180] (.6,.8);
    \draw[->] (.6,.2) to [out = 180, in = 180] (.6,2.8);
    \draw[->] (3.4,.8) to [out= 0, in=0] (3.4,2.2);
    \draw[->] (3.4,2.8) to [out = 0, in = 0](3.4,.2);
    
        \draw (2,-.5) node{$D_{\text{pos}}$};

\end{scope}

\end{tikzpicture}\]
\caption{Performing two saddle moves transforms $D$ into a positive diagram $D_{\pos}$.}
\label{figure:aa32}
\end{figure}
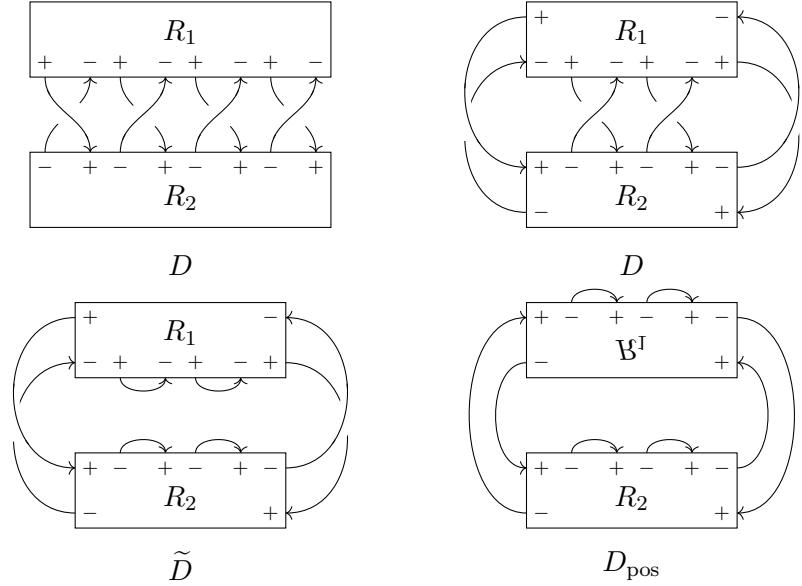

Theorem \ref{theorem:main} gives information about the first three nontrivial homological gradings of Khovanov homology. One might expect to repeat the arguments of this section when $D$ is $A$-Turaev genus one and the $i=0$ homological grading occurs in the second nontrivial homological grading. In the case where $D$ is $A$-adequate and Turaev genus one, this happens when the knot has one negative crossing. Tagami \cite{Tagami} proved that $2g_4(K)=s(K)$ for such knots. In the case where $D$ is $A$-almost alternating, the $i=0$ homological grading occurs in the second homological grading when $D$ has two negative crossings. Alas, this does not yield new interesting examples because of the following theorem.
\begin{theorem}
There are no $A$-almost alternating knots with exactly two negative crossings.
\end{theorem}
\begin{proof}
Let $D$ be an $A$-almost alternating diagram with two negative crossings, and let $D_{\alt}$ be the reduced alternating diagram obtained by changing the dealternator in $D$. Then $D_{\alt}$ either has one or three negative crossings. Since $D_{\alt}$ is a reduced alternating diagram, it is $A$-adequate. Lemma \ref{lemma:cycle} states that each cycle in the cycle decomposition of $\Gamma(D)$ has an even number of negative vertices, which is impossible if $D$ only has one negative crossing. If $D$ has three negative crossings, then there are three cycles in the cycle decomposition of $\Gamma(D)$, each of which contains two negative vertices. In this case, there will be a region in the alternating tangle of $D$ adjacent to both $u_1$ and $u_2$, and thus $D$ is not $A$-almost alternating. Thus there is no diagram $A$-almost alternating diagram $D$ with two negative crossings.
\end{proof}

\bibliographystyle{amsalpha}
\bibliography{bib}

\end{document}